\documentclass[12pt]{amsart}

\usepackage{hyperref}
\usepackage{fullpage}
\usepackage[mathscr]{euscript}
\usepackage{amsmath}
\usepackage{amssymb}
\usepackage{pstricks}
\usepackage{enumerate}
\usepackage{graphicx}%[demo]
\usepackage{verbatim}
\usepackage{tikz}
\usetikzlibrary{matrix,arrows}  	
\usepackage{slashed}
\usepackage{mathrsfs}
\theoremstyle{plain}
\newtheorem{theorem}{Theorem}[subsection]
\newtheorem{lemma}[theorem]{Lemma}
\newtheorem{proposition}[theorem]{Proposition}

\newtheorem{corollary}[theorem]{Corollary}

\theoremstyle{definition}

\newtheorem{definition-proposition}[theorem]{Definition/Proposition}

\theoremstyle{remark}
\newtheorem{remark}[theorem]{Remark}

\numberwithin{equation}{subsection}
\numberwithin{theorem}{subsection}
\newcommand{\oD}{ \!{\buildrel \circ 
\over D}^{_{_{_{\mbox{{\small $_{q+1}$}}}}}}}

\newcommand{\p}{\partial}

\newcommand{\cA}{\mathcal{A}}

\newcommand{\cC}{\mathcal{C}}

\newcommand{\cR}{\mathcal{R}}

\newcommand{\cT}{\mathcal{T}}

\def\cyl{\mathrm{c}\mathrm{y}\mathrm{l}}

\def\bend{\mathrm{b}\mathrm{e}\mathrm{n}\mathrm{d}}
\def\Euc{\mathrm{E}\mathrm{u}\mathrm{c}}
\newcommand{\std}{\mathrm{std}}
\newcommand{\Estd}{\mathrm{Estd}}

\newcommand{\Astd}{\mathrm{Astd}}

\newcommand{\boot}{\mathrm{boot}}

\newcommand{\con}{\mathrm{con}}

\newcommand{\Diff}{\mathrm{Diff}}

\newcommand{\GL}{\mathrm{GL}}

\newcommand{\Id}{\mathrm{id}}

\newcommand{\Mod}{\mathcal{M}}

\newcommand{\pre}{\mathrm{pre}}

\newcommand{\Riem}{\cR}
\newcommand{\res}{\mathrm{res}}

\newcommand{\step}{\mathrm{step}}

\newcommand{\stret}{\mathrm{stretch}}

\newcommand{\tor}{\mathrm{torp}}
\newcommand{\toe}{\mathrm{toe}}

\newcommand\lra{\longrightarrow}

\newcommand{\image}{\mathrm{image}}

\newcommand{\CircNum}[1]{\ooalign{\hfil\raise .00ex\hbox{\scriptsize #1}\hfil\crcr\mathhexbox20D}}

\newcommand{\rh}{\hookrightarrow}

\usepackage[all]{xy}
\SelectTips{cm}{10}
\CompileMatrices

\title[The Space of Positive Scalar Curvature Metrics on a Manifold with Boundary]{The Space of Positive Scalar Curvature Metrics on a Manifold with Boundary}

\author{Mark Walsh}

\email{Mark.Walsh@mu.ie}
\address{
Mathematics and Statistics\\
Maynooth University\\
Maynooth, County Kildare\\
Ireland
}

\date{\today}

\keywords{space of Riemannian metrics of positive scalar curvature, manifold with boundary, surgery, bordism, spin, Gromov-Lawson construction, weak homotopy equivalence\\\\
\indent The author acknowledges support from Simons Foundation Collaboration Grant No. 280310.}

\begin{document}
\begin{abstract}
We study the space of Riemannian metrics with positive scalar curvature on a compact manifold with boundary. These metrics extend a fixed boundary metric and take a product structure on a collar neighbourhood of the boundary. We show that the weak homotopy type of this space is preserved by certain surgeries on the boundary in co-dimension at least three. Thus, there is a weak homotopy equivalence between the space of such metrics on a simply connected spin manifold $W$, of dimension $n\geq 6$ and with simply connected boundary, and the corresponding space of metrics of positive scalar curvature on the standard disk $D^{n}$. Indeed, for certain boundary metrics, this space is weakly homotopy equivalent to the space of all metrics of positive scalar curvature on the standard sphere $S^{n}$. Finally, we prove analogous results for the more general space where the boundary metric is left unfixed.
\end{abstract}
\maketitle
\newpage
\tableofcontents
\newpage

\section{Introduction}

Recently, much progress has occurred in understanding the topology of the space $\Riem^{+}(X)$, of Riemannian metrics of positive scalar curvature (psc-metrics) on a smooth manifold $X$; see for example \cite{BERW},  \cite{CS}, \cite{HSS} and \cite{Marques}. With some exceptions (for example \cite{BERW} and \cite{EF}), these results deal mostly with the case when $X$ is closed. 
 In this paper, we replace $X$ with a manifold $W$, whose boundary $\p W$ is non-empty and, after imposing certain boundary conditions, we study an analogous space of psc-metrics on $W$, $\Riem^{+}(W, \p W)$. The results of this paper aim to shed some light on the problem of understanding the topology of this space and some relevant subspaces, as with the  analogous case for closed manifolds.

More precisely, let $W$ be a smooth compact manifold with
$\dim W =n+1$, and boundary $\partial W = X$, a closed smooth manifold with $\dim X=n$. We specify a collar $c:X \times  [0,2) \rh W$ and 
  denote by $\Riem^{+}(W, X)$, the space of all psc-metrics on $W$
  which take a product structure on the
  {image} $c\left(X \times  [0,1]\right)$. Thus, $h\in \Riem^{+}(W, X)$ if $c^{*}h=g+dt^{2}$ restricted to
    $X\times [0,1]$ for some
      $g\in\Riem^{+}(X)$. A further boundary condition we impose is to fix a psc-metric $g\in\Riem^{+}(X)$. We then define the subspace $\Riem^{+}(W, X)_g\subset \Riem^{+}(W, X)$ of all psc-metrics $h\in\Riem^{+}(W, X)$ where $(c^{*}h)|_{X\times \{0\}}=g$. Note that we allow for the possibility that the space $\Riem^{+}(W, X)_{g}$, or $\Riem^{+}(W, X)$, may be empty. To formulate our main theorem, we consider another smooth compact $(n+1)$-dimensional manifold $Z$ whose boundary $\partial Z=X_0\sqcup X_1$, is a disjoint union of closed $n$-manifolds. Thus $Z$ is a cobordism of $X_0$ and $X_1$, sometimes denoted as the triple $(Z; X_0, X_1)$. Here we specify a pair of disjoint collars $c_0: X_0 \times [0,2) \rh Z$, $c_1: X_1\times [0,2) \rh Z$ around $X_0$ and $X_1$ respectively. We fix a pair of psc-metrics $g_0\in\Riem^{+}(X_0)$ and $g_1\in \Riem^{+}(X_1)$ and denote by $\Riem^{+}(Z, \p Z)_{g_0, g_1}$, the space of psc-metrics $\bar{g}$ on $Z$ so that $c_i^{*}\bar{g}=g_i+dt^{2}$ restricted on $X_i\times [0,1]$ for $i=0,1$. We assume for now that $\Riem^{+}(Z, \p Z)_{g_0, g_1}$ is non-empty, although this need not be the case. Returning to the manifold $W$, we further suppose that $\p W = X = X_0$, one of the boundary components of $Z$. Let $W\cup Z$ denote the manifold obtained by gluing $Z$ to $W$ along this boundary component. Denoting by $c$ the collar $c_1$, we consider the subspace $\Riem^{+}(W\cup Z, X_1)_{g_1}$ of $\Riem^{+}(W\cup Z, X_1)$, consisting of psc-metrics which restrict as ${g_1}+dt^{2}$ on $c(X_1\times [0,1])$. For any element $\bar{g}\in\Riem^{+}(Z, \p Z)_{g_0, g_1}$, there is a map:
      \begin{equation}\label{gluemap}
      \begin{split}
      \mu_{Z,\bar{g}}:\Riem^{+}(W, X_0)_{g_0}&\longrightarrow \Riem^{+}(W\cup Z, X_1)_{g_1}\\
      h&\longmapsto h\cup \bar{g},
      \end{split}
      \end{equation} 
      where $h\cup \bar{g}$ is the metric obtained on $W\cup Z$ by the obvious gluing depicted in Fig.\ref{bordism}.
       \begin{figure}[!htbp]
\vspace{0.0cm}
\hspace{1cm}
\begin{picture}(0,0)
\includegraphics{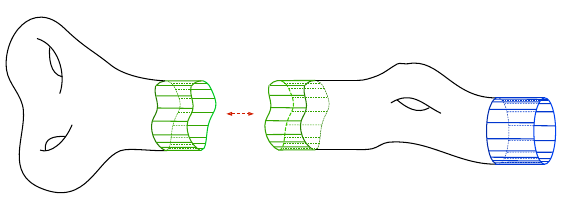}%
\end{picture}
\setlength{\unitlength}{3947sp}%
\begingroup\makeatletter\ifx\SetFigFont\undefined%
\gdef\SetFigFont#1#2#3#4#5{%
  \reset@font\fontsize{#1}{#2pt}%
  \fontfamily{#3}\fontseries{#4}\fontshape{#5}%
  \selectfont}%
\fi\endgroup%
\begin{picture}(5079,1559)(1902,-7227)
\put(1400,-6000){\makebox(0,0)[lb]{\smash{{\SetFigFont{10}{8}{\rmdefault}{\mddefault}{\updefault}{\color[rgb]{0,0,0}$(W, h)$}%
}}}}
\put(3050,-7000){\makebox(0,0)[lb]{\smash{{\SetFigFont{10}{8}{\rmdefault}{\mddefault}{\updefault}{\color[rgb]{0,0,0}$g_0+dt^{2}$}%
}}}}
\put(4000,-7000){\makebox(0,0)[lb]{\smash{{\SetFigFont{10}{8}{\rmdefault}{\mddefault}{\updefault}{\color[rgb]{0,0,0}$g_0+dt^{2}$}%
}}}}
\put(5700,-7130){\makebox(0,0)[lb]{\smash{{\SetFigFont{10}{8}{\rmdefault}{\mddefault}{\updefault}{\color[rgb]{0,0,0}$g_1+dt^{2}$}%
}}}}
\put(5000,-5900){\makebox(0,0)[lb]{\smash{{\SetFigFont{10}{8}{\rmdefault}{\mddefault}{\updefault}{\color[rgb]{0,0,0}$(Z, \bar{g})$}%
}}}}
\end{picture}%
\caption{Attaching $(W, h)$ to $(Z,\bar{g})$ along a common boundary}
\label{bordism}
\end{figure}   

\subsection{Results}

 Suppose $\phi: S^{p}\times D^{q+1}\rightarrow X$ is an embedding where $p+q+1=n$ and $q\geq 2$. Let $T_{\phi}$ be the trace of the surgery on $X$ with respect to $\phi$. Thus $\partial T_{\phi}=X\sqcup X'$ where $X'$ is the manifold obtained from $X$ by surgery. We will now let the cobordism $(T_{\phi}; X, X')$ play the role of $(Z; X_0, X_1)$ above. In the case when $q\geq 2$, the Surgery Theorem of Gromov and Lawson \cite{GL} describes a technique for constructing, from any psc-metric $g\in \Riem^{+}(X)$, a psc-metric $g'\in \Riem^{+}(X')$. A useful strengthening of this technique allows for the determination of a particular psc-metric  $\bar{g}\in\Riem^{+}(T_{\phi}, \p T_{\phi})_{g,g'}$, known as a {\em Gromov-Lawson trace} (or, more generally, a {\em Gromov-Lawson cobordism}); see \cite{Gajer}, \cite{Walsh1}. Metrics which are accessible from each other by a sequence of Gromov-Lawson surgeries are said to be {\em Gromov-Lawson cobordant.}  The map (\ref{gluemap}) now takes the form: 
$$\mu_{T_{\phi,\bar{g}}}:\Riem^{+}(W, X)_{g}\rightarrow\Riem^{+}(W', X')_{g'},$$ 
where $W'=W\cup T_{\phi}$. Our results are as follows.
\vspace{0.2cm}

\noindent{\bf Theorem A.} {\em Suppose $p,q\geq 2$. For any $g\in\Riem^{+}{(X)}$, there exist psc-metrics $g'\in\Riem^{+}(X')$ and $\bar{g}\in\Riem^{+}(T_{\phi})_{g, g'}$ so that the map $\mu_{T_{\phi},\bar{g}}$ is a weak homotopy equivalence:
\begin{equation*} 
\Riem^{+}(W, X)_{g}\simeq \Riem^{+}(W', X')_{g'}.
\end{equation*}}

\noindent{\bf Theorem B.}
{\em Suppose $W$ is a smooth compact spin manifold with closed boundary $X$. We further assume that $W$ and $X$ are both simply connected and that $\dim W=n+1\geq 6$. 
\begin{enumerate}
\item[(i.)] For any $g\in\Riem^{+}(X)$ where $\Riem^{+}(W, \p W)_{g}$ is non-empty, there is a psc-metric $g'\in\Riem^{+}(S^{n})$ and a weak homotopy equivalence:
\begin{equation*}
\Riem^{+}(W, X)_{g}\simeq \Riem^{+}(D^{n+1}, \p D^{n+1}=S^{n})_{g'}.
\end{equation*}
\item[(ii.)] Furthermore, if $g$ is Gromov-Lawson cobordant to the standard round metric $ds_{n}^{2}$, there is a weak homotopy equivalence:
\begin{equation*}
\Riem^{+}(W, X)_{g}\simeq \Riem^{+}(S^{n+1}).
\end{equation*}
\end{enumerate} }
 \vspace{0.2cm}

\noindent{\bf Theorem C.}
{\em When $p,q\geq 2$, the spaces $\Riem^{+}(W, X)$ and $\Riem^{+}(W', X')$ are weakly homotopy equivalent.}
\vspace{0.2cm}

\noindent{\bf Corollary D.} {\em When $W$ and $X$ satisfy the hypotheses of Theorem B, the spaces $\Riem^{+}(W, X)$ and $\Riem^{+}(D^{n+1}, S^{n})$ are weakly homotopy equivalent.}

\subsection{Background}\label{background}  
We begin with a brief discussion of the original problem for a closed $n$-dimensional manifold $X$. The space $\Riem^{+}(X)$ is an open subspace of the space of all Riemannian metrics on $X$, denoted $\Riem(X)$, under its usual $C^{\infty}$-topology. An old question in this subject is whether or not $X$ admits any psc-metrics, i.e.  whether or not $\Riem^{+}(X)$ is non-empty. Although work continues on this problem, in the case when $X$ is simply connected and $n\geq 5$, necessary and sufficient conditions are known: $\Riem^{+}(X)\neq \emptyset$ if and only if $X$ is either non-spin or $X$ is spin with Dirac index $\alpha(X)\in KO_{n}$ equal to zero. This result is due to Stolz \cite{S}, following important work by Gromov, Lawson \cite{GL} and others. For a survey of this problem, see \cite{RS}.

In the case when the space $\Riem^{+}(X)$ is non-empty, one may inquire about its topology. Up until recently, very little was known about this space beyond the level of path-connectivity. Hitchin showed for example, in \cite{Hit}, that if $X$ is spin, $\pi_{0}(\Riem^{+}(X))\neq 0$ when $n\equiv 0,1 (\mathrm{mod} 8)$ and that $\pi_{1}(\Riem^{+}(X))\neq 0$ when $n\equiv 0,-1 (\mathrm{mod} 8)$. It is worth noting that all of these non-trivial elements disappear once one descends to $\Mod^{+}(X):=\Riem^{+}(X)/\Diff(X)$, the moduli space of psc-metrics. Here, $\Diff(X)$ is the group of self-diffeomorphisms on $X$ and acts on $\Riem^{+}(X)$ by pulling back metrics. Later, Carr showed in \cite{Carr} that when $X$ is the sphere $S^{n}$, $\pi_{0}(\Riem^{+}(S^{4k-1}))$ is infinite for all $k\geq 2$ and all but finitely many of these non-trivial elements survive in the moduli space. Various generalisations of this result have been achieved. In particular, Botvinnik and Gilkey showed that $\pi_{0}(\Riem^{+}(X))\neq 0$ in the case when $X$ is spin and $\pi_1(X)$ is finite; see \cite{BG}. It is also worth mentioning the Kreck-Stolz s-invariant, defined in \cite{KreckStolz}, which
distinguishes path components of the space $\Mod^{+}(X)$ under certain circumstances. More recently, there have been a number of significant results which exhibit the non-triviality of higher homotopy groups of both $\Riem^{+}(X)$ and $\Mod^{+}(X)$ for a variety of manifolds $X$; see \cite{BHSW}, \cite{CS} and \cite{HSS}. Most of these results (\cite{Hit}, \cite{CS}, \cite{HSS}) involve showing that for certain closed spin manifolds $X$ and certain psc-metrics $g\in\Riem^{+}(X)$, a particular variation of the Dirac index, introduced by Hitchin in \cite{Hit}, often induces non-trivial homomorphisms:
\begin{equation*}
A_{k}(X,g): \pi_k (\Riem^+ (X), g) \lra KO_{k+n+1}.
\end{equation*}
Most recently of all, Botvinnik, Ebert and Randal-Williams in \cite{BERW}, show that this map is always non-trivial when the codomain is non-trivial. Their methods are new and make use of work done by Randal-Williams and Galatius on moduli spaces of manifolds; see \cite{GRW}. 

One result which the authors in \cite{BERW} make use of and which is of particular relevance here, is the following theorem of Chernysh which utilises a family version of the Gromov-Lawson construction. 
\begin{theorem}\label{Chernysh} {\bf (Chernysh \cite{Che})}
Let $X$ be a smooth compact manifold of dimension $n$. Suppose $X'$ is obtained from $X$ by surgery on a sphere $i:S^{p}\hookrightarrow {X}$ with $p+q+1=n$ and  $p,q\geq 2$. Then the spaces $\Riem^{+}(X)$ and $\Riem^{+}(X')$ are homotopy equivalent.
\end{theorem}
\noindent  This theorem was originally proved by Chernysh in \cite{Che} but was never published. Later, this author provided a short version of the proof in \cite{Walsh3} based on work done in \cite{Walsh1}. Admittedly, this version was rather terse and did not adequately address all details. Quite recently however, Ebert and Frenck have provided a comprehensive proof of this theorem; see \cite{EF}. Their paper also contains a strengthening of another relevant result of Chernysh \cite{Che2} as well as a correction to a computational error found in expositions of the original Gromov-Lawson Surgery Theorem (\cite{RS}, \cite{Walsh1}). Theorem \ref{Chernysh} and the techniques used to prove it play a fundamental role in proving the results of this paper. Indeed, Theorem A is effectively a generalisation of Chernysh's theorem to work for certain types of ``boundary surgery". Thus, it will be necessary to provide an overview of the main steps in proving Theorem \ref{Chernysh} as well as the original Gromov-Lawson construction.  

We close by recalling some fundamental questions which motivate this work.\vspace{0.2cm} 
\begin{enumerate}
\item[(1.)] { Given some $g\in\Riem^{+}(\p W)$, is the space $\Riem^{+}(W, \p W)_{g}$ non-empty?}\vspace{0.2cm}
\item[(2.)] { If $\Riem^{+}(W, \p W)_{g}\neq \emptyset$, what can we say about its topology?}\vspace{0.2cm}
\item[(3.)] { What can we say about the topology of the space $\Riem^{+}(W, \p W)$?}\vspace{0.2cm}
\end{enumerate}
Although not strictly the focus of this work, Question (1.) is relevant here as its answer is often negative. For example, the methods used by Carr in \cite{Carr} give rise to psc-metrics on $S^{4k-1}$ which do not extend to elements of $\Riem^{+}(D^{4k})$, for all $k\geq 2$. Questions (2.) and (3.) are posed in problem 3, section 2.1 of the survey article \cite{RS}, and our results are a contribution to answering these questions.

\subsection{Acknowledgements} This work was originally intended as an appendix to the paper, \cite{BERW}, by B. Botvinnik, J. Ebert and O. Randal-Williams but grew into a larger project of independent interest. I would like to express gratitude to all three authors for their comments and suggestions as this project developed, and for the original invitation to submit this work as an appendix. I would like to express my utmost gratitude B. Botvinnik, for originally suggesting this project to me and for some significant contributions to its development. A number of insights relating to this work were gained by attendance of the workshop ``Analysis and Topology in Interaction" in Cortona, June 2014. My thanks to the organisers, epecially W. L{\"u}ck, P. Piazza and T. Schick, for their kind invitation. Finally, it is a pleasure to thank D. Wraith and C. Escher for their helpful comments.

\section{Preliminary Details}\label{prelim}
In this section we will take care of preliminary details concerning certain objects and notions which will be used throughout the paper. In particular, we will recall the notions of isotopy and concordance on spaces of psc-metrics. Unless otherwise stated, all manifolds in this paper are smooth and compact. In particular, $X$ always denotes a smooth closed manifold of dimension $n$, while $W$ denotes a smooth compact $(n+1)$-dimensional manifold with a non-empty closed boundary; usually $\p W=X$. 

\subsection{Spaces of Metrics} 
Given a smooth compact $n$-dimensional manifold $X$, we denote by $\Riem(X)$, the space of all Riemannian metrics on $X$. The space $\Riem(X)$ is equipped with the standard $C^{\infty}$-topology, giving it the structure of a Fr{\'e}chet manifold; see chapter 1 of \cite{TW} for details. There are various moduli spaces one may wish to study, obtained by quotienting $\Riem(X)$ by the usual action of some or other subgroup of the diffeomorphism group $\Diff(X)$; see \cite{TW}. In this paper however, we will only focus on $\Riem(X)$ proper. For each metric $g\in \Riem(X)$, we denote by $s_{g}:X\rightarrow \mathbb{R}$, the smooth function which is the scalar curvature on $X$ of the metric $g$. Finally, we denote the space of metrics of positive scalar curvature (psc-metrics) on $X$ by:
$$\Riem^{+}(X):=\{g\in \Riem(X): s_{g}>0\}.$$
This is the open subspace of $\Riem(X)$ consisting of Riemannian metrics on $X$ whose scalar curvature function is everywhere positive. As mentioned in the introduction, the space $\Riem^{+}(X)$ may or may not be empty. Throughout this paper however, the reader should assume we are working with a manifold $X$ for which $\Riem^{+}(X)\neq \emptyset$.

Suppose that $X$ forms the boundary of a smooth compact $(n+1)$-dimensional manifold $W$; thus $\p W=X$. As above, we denote by $\Riem(W)$, the space of Riemannian metrics on $W$ under the usual $C^{\infty}$-topology. However, to make our work meaningful we require an additional constraint on metrics near the boundary. We specify a collar, $c:X \times  [0,2) \rh W$, of the boundary $\p W=X$. Thus, $c$ is an embedding and $c(X\times\{0\})=\p W\subset W$. Letting $I=[0,1]$, we denote by $\Riem(W, \p W)$, the subspace of all Riemannian metrics on $W$
  which restrict as a product structure on the
  image $c\left(X \times  I\right)$. Thus, for any $\bar{g}\in\Riem(W)$, $\bar{g}\in \Riem(W, \p W)$ if $c^{*}\bar{g}=g+dt^{2}$ on
    $X\times I$ for some
      $g\in\Riem(X)$. 
The corresponding space of psc-metrics on $W$, denoted $\Riem^{+}(W, \p W)$, is now defined by:
      $$\Riem^{+}(W, \p W):=\{\bar{g}\in \Riem(W, \p W): s_{\bar{g}}>0\}.$$
For each $h\in \Riem^{+}(X)$, we denote by $\Riem^{+}(W, \p W)_h\subset \Riem^{+}(W, \p W)$ the subspace of all psc-metrics $\bar{g}\in\Riem^{+}(W, \p W)$ where $(c^{*}\bar{g})|_{X\times \{0\}}=h$. It is important to remember that the space $\Riem^{+}(W, \p W)$ may be empty. Moreover, even when $\Riem^{+}(W, \p W)\neq \emptyset$, it is possible that for some $h\in \Riem^{+}(X)$, the space 
$\Riem^{+}(W, \p W)_h$ is empty. As mentioned in the introduction, the problem of deciding for a given $h\in \Riem^{+}(X)$, whether or not the space $\Riem^{+}(W, \p W)_h$ is non-empty is highly non-trivial. Although we assume that $\p W=X$ admits psc-metrics, we will make no a priori assumptions about the emptiness or otherwise of the spaces $\Riem^{+}(W, \p W)$ or $\Riem^{+}(W, \p W)_h$. 

Another way of thinking about all of this is to consider the natural restriction map: $$\res: \Riem^{+}(W, \p W)\longrightarrow \Riem^{+}(\p W), $$ where $\res(\bar{g})=\bar{g}|_{\p W}$. Thus, $\res^{-1}(h)=\Riem^{+}(W, \p W)_h$. It is fact, due to Chernysh \cite{Che2} and Ebert and Frenck \cite{EF} that the map $\res$ is actually a Serre Fibration. This is something we will make use of later on where we draw conclusions about the space $\Riem^{+}(W, \p W)$ based on results about  $\Riem^{+}(W, \p W)_h$.

We close this section by considering a special case of a manifold with boundary. Consider an $(n+1)$-dimensional manifold $Z$ whose boundary $\partial Z=X_0\sqcup X_1$, is a disjoint union of closed $n$-dimensional manifolds. Thus, $Z$ is a cobordism of $X_0$ and $X_1$, sometimes denoted as the triple $(Z; X_0, X_1)$. Here we specify a pair of disjoint collars $c_0: X_0 \times [0,2) \rh Z$, $c_1: X_1\times [0,2) \rh Z$ around $X_0$ and $X_1$ respectively. In this case, $\Riem(Z, \p Z)$ denotes the space of Riemannian metrics on $Z$ which restrict as a product structure on each of the neighbourhoods $c(X_{i}\times I)$, where $i=0,1$. Thus, $\bar{g}\in\Riem(Z, \p Z)$ satisfies:
$$c_0^{*}\bar{g}=g_0+dt^{2} \text{ on } X_0\times I\quad \text{ and }\quad  c_1^{*}\bar{g}=g_1+dt^{2} \text{ on } X_1\times I, $$ 
for some pair of metrics $g_0\in \Riem(X_0)$ and $g_1\in \Riem(X_1)$.
As usual, the corresponding space of psc-metrics on $Z$ is denoted:
$$\Riem^{+}(Z, \p Z):=\{\bar{g}\in \Riem(Z, \p Z): s_{\bar{g}}>0\}.$$
After fixing a pair of psc-metrics $g_0\in\Riem^{+}(X_0)$ and $g_1\in \Riem^{+}(X_1)$, we consider the subspace $\Riem^{+}(Z, \p Z)_{g_0, g_1}\subset \Riem^{+}(Z, \p Z) $ defined as follows:
$$\Riem^{+}(Z, \p Z)_{g_0, g_1}:=\{ \bar{g}\in\Riem^{+}(Z, \p Z):
c_i^{*}\bar{g}=g_i+dt^{2} \text{ on } X_i\times [0,1], \text{ where } i=0,1\}.$$ 
As before, we point out that the space $\Riem^{+}(Z, \p Z)_{g_0, g_1}$ may be empty.

\subsection{Weak homotopy equivalence} Most of the results of this paper involve exhibiting {weak homotopy equivalence} between topological spaces. Recall that a continuous map $f:A\rightarrow B$ of topological spaces $A$ and $B$ is a {\em weak homotopy equivalence} if it induces isomorphisms $\pi_{m}(A, a_0)\rightarrow \pi_{m}(B, f(a_0))$ for all $m\geq 0$ and all choices of basepoint $a_0\in A$. In the case of an inclusion $A\subset B$, this is equivalent to showing that the relative homotopy groups $\pi_{m}(B,A)$ are trivial for all $m$. Recall that an element $\alpha\in \pi_{m}(B,A)$ is a homotopy class of commutative diagrams of continuous maps:
\begin{equation*}
\xymatrix{
  S^{m-1} \ar[d]\ar@{^{(}->}[r]  & D^{m} \ar[d]
  \\
 A \ar@{^{(}->}[r]  &B }
\end{equation*}   
Thus to show that $\alpha$ is trivial, we must show that any such commutative diagram is homotopy equivalent to one where the image of the right vertical map lies entirely in $A$. Put another way, if $f:D^{m}\rightarrow B$ is a continuous map satisfying $f(x)\in A$ for all $x\in \p D^{m}=S^{m-1}$, we must exhibit a homotopy $F:I\times D^{m}\rightarrow A$ so that:
\begin{enumerate}
\item[(i.)] $F(0,x)=f(x)$ for all $x\in D^{m}$, 
\item[(ii.)] $F(1,x)\in A$ for all $x\in D^{m}$,
\item[(iii.)] $F(\tau,x)\in A$ for all $x\in \p D^{m}=S^{m-1}$ and all $\tau\in I$.
\end{enumerate}
In all cases in this paper, we will demonstrate weak homotopy equivalence of an inclusion $A\subset B$ by proving the following more general but entirely sufficient condition. That is, that $A\subset B$ will be a weak homotopy equivalence if for any compact space $K$ and any continuous map $f:K\rightarrow B$, there is a homotopy of $F:I\times K\rightarrow B$ satisfying:
\begin{enumerate}
\item[(i.)] $F(0,k)=f(k)$ for all $k\in K$ 
\item[(ii.)] $F(1,k)\in A$ for all $k\in K$,
\item[(iii.)] $F(\tau,k)\in A$ for all $k\in K$ satisfying $F(0, k)\in A$ and all $\tau\in I$.
\end{enumerate}

%For any smooth manifolds, $M$ and $N$, $C^{\infty}(M,N)$ denotes the space of smooth functions from $M$ to $N$ under the usual $C^{\infty}$ function space topology; see chapter 2 of \cite{Hirsch} for details. Throughout the paper, when we consider subspaces of smooth functions from $M$ to $N$, such as smooth embeddings, we assume the subspace topology.

\subsection{Isotopy, concordance and compact families}\label{isoconcfam}
We will concentrate initially on the space $\Riem^{+}(X)$, of positive scalar curvature metrics on $X$. Everything we say here has an obvious analogue in terms of the spaces $\Riem^{+}(W, \p W)$, $\Riem^{+}(W, \p W)_g$, $\Riem(Z, \p Z)$ and $\Riem^{+}(Z, \p Z)_{g_0, g_1}$ defined above. A {\em compact family} of psc-metrics is a continuous map, $K\rightarrow \Riem^{+}(X)$, where $K$ is some compact space. We say that the family is parameterised by the space $K$. In this paper we will be concerned only with the case when $K$ is a disk. In particular, an important special case of this is when $K$ is the interval $I=[0,1]$. Two psc-metrics $g_0, g_1\in\Riem^{+}(X)$ are {\em isotopic} if there exists a path, $I\longrightarrow\Riem^{+}(X)$, defined by $t\mapsto g_t$, connecting $g_0$ to $g_1$. Such a path is called an {\em isotopy}. Two psc-metrics $g_0, g_1\in\Riem^{+}(X)$ are said to be {\em concordant} if there is a psc-metric on the cylinder $X\times I$ which takes the form of a product $g_0+dt^{2}$ and $g_1+dt^{2}$ near the respective ends $X\times\{0\}$ and $X\times \{1\}$. It will often be useful to use an equivalent form of the  definition, namely that two psc-metrics $g_0, g_1\in\Riem^{+}(X)$ are {\em concordant} if, for some $L>0$, there is a psc-metric $\bar{g}$ on the cylinder $X\times[0,L+2]$ so that: $$ \bar{g}|_{X\times[0,1]}=g_0+dt^{2}\quad\text{and}\quad \bar{g}|_{X\times[L+1,L+2]}=g_0+dt^{2}.$$ The metric $\bar{g}$ is called a {\em concordance} of $g_0$ and $g_1$. Defining collars $c_0:X\times [0,2)\rh X\times[0,L+2]$ and $c_1:X\times [0, 2)\rh X\times[0,L+2]$ so that:
$$c_0|_{X\times [0,1]}(x,t)=(x,t) \text{ and } c_1|_{X\times [0,1]}(x,t)=(x, L+2-t),$$
we see that $\bar{g}\in\Riem^{+}(X\times[0,L+2], (X\times\{0\})\sqcup (X\times\{L+2\}) )_{g_0, g_1}$, the space of concordances between $g_0$ and $g_1$ on $X\times[0,L+2]$. 

It is a well-known fact that isotopic metrics are concordant; see Lemma 3 of \cite{GL}. In particular, there is a fairly straightforward process for turning an isotopy into a concordance. Suppose $t\mapsto g_t\in\Riem^{+}(X)$, where $t\in I$, is an isotopy. Consider initially the metric $g_t+dt^{2}$ on $X\times I$. This metric does not necessarily have positive scalar curvature, as negative curvature may arise in the $t$-direction. It also likely lacks the appropriate product structure near the boundary. However, by appropriately ``slowing down" change in the $t$ direction we can minimize negative curvature and use the slices to obtain overall positivity. This is the subject of the following lemma, which allows us to turn an isotopy into a concordance. Various versions of this lemma may be found in the literature; see for example  \cite{Gajer} or \cite{GL}. 
%\textcolor{red}{\bf Provide better reference.}
%Given its importance in our work, we provide a detailed proof in an appendix to this paper, section \ref{append}.

\begin{lemma}\label{isotopyimpliesconc}
Let $g_{r}, r\in I$ be a smooth path in $\Riem^{+}(X)$. Then there exists a constant $0<\Lambda\leq 1$ so that
for every smooth function $f:\mathbb{R}\rightarrow[0,1]$ with
$|\dot{f}|,|\ddot{f}|\leq\Lambda$, the metric $\bar{g}=g_{f(t)}+dt^{2}$ on
$X\times\mathbb{R}$ has positive scalar curvature.
\end{lemma}
%\begin{corollary}
%Isotopic psc-metrics are concordant.
%\end{corollary}

It is useful to have a well-defined way of obtaining a concordance from an isotopy. Indeed, we will require a method of converting a compact family of continuously parameterised isotopies into a corresponding family of continuously parameterised concordances. With this in mind, we fix a family of
appropriate smooth cut-off functions $\nu_{L}:[0,L+2]\rightarrow
[0,1]$, with $L> 0$, as shown in Fig. \ref{cutofffunc}. Each function is non-decreasing and satisfies $\nu_L(t)=0$ when $t\in [0,1]$ and $\nu_L(t)=1$ when $t\in[0,L+2]$.
This is best done by specifying $\nu_{1}$ and then defining $\nu_{L}$ by:
\[ 
\nu_{L}(t)=
\begin{cases} 
      0 &  0\leq t\leq 1, \\
       \nu_{1}(\frac{t+L-1}{L})& 1\leq t\leq L+1, \\
      1 & L+1\leq t\leq L+2.
   \end{cases}
\]
 %Lemma \ref{isotopyimpliesconc} shows that the scalar curvature $\bar{R}$ of this metric takes the form 
%\begin{equation*}
%\bar{R}=R+O(|\dot{{\nu_{\lambda}}}|)+O(|\dot{{{\nu_{\lambda}}}}|^{2})+O(|\ddot{{{\nu_{\lambda}}}}|),
%\end{equation*}
Replacing $f$ in Lemma \ref{isotopyimpliesconc} with $\nu_{L}$, there is a constant $\Lambda$, so that the scalar curvature of the metric $g_{{\nu_{L}(t)}}+dt^{2}$ on $X\times[0,L+2]$ is positive whenever 
$|\dot{{\nu_{L}}}|,|\ddot{{\nu_{L}}}|\leq\Lambda.$ By choosing sufficiently large $L>0$ these inequalities can be made to hold, resulting in a psc-metric on $X\times[0,L+2]$: the desired concordance.  

\begin{figure}[htb!]
\vspace{0cm}
\hspace{0cm}
\begin{picture}(0,0)%
\includegraphics{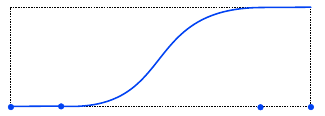}%
\end{picture}%
\setlength{\unitlength}{3947sp}%
\begingroup\makeatletter\ifx\SetFigFont\undefined%
\gdef\SetFigFont#1#2#3#4#5{%
  \reset@font\fontsize{#1}{#2pt}%
  \fontfamily{#3}\fontseries{#4}\fontshape{#5}%
  \selectfont}%
\fi\endgroup%
\begin{picture}(2500,1000)(1902,-4570)
\put(1900,-4700){\makebox(0,0)[lb]{\smash{{\SetFigFont{10}{8}{\rmdefault}{\mddefault}{\updefault}{\color[rgb]{0,0,0}$0$}%
}}}}
\put(2350,-4700){\makebox(0,0)[lb]{\smash{{\SetFigFont{10}{8}{\rmdefault}{\mddefault}{\updefault}{\color[rgb]{0,0,0}$1$}%
}}}}
\put(1800,-4500){\makebox(0,0)[lb]{\smash{{\SetFigFont{10}{8}{\rmdefault}{\mddefault}{\updefault}{\color[rgb]{0,0,0}$0$}%
}}}}
\put(1800,-3800){\makebox(0,0)[lb]{\smash{{\SetFigFont{10}{8}{\rmdefault}{\mddefault}{\updefault}{\color[rgb]{0,0,0}$1$}%
}}}}
\put(3700,-4700){\makebox(0,0)[lb]{\smash{{\SetFigFont{10}{8}{\rmdefault}{\mddefault}{\updefault}{\color[rgb]{0,0,0}$L+1$}%
}}}}

\put(4300,-4700){\makebox(0,0)[lb]{\smash{{\SetFigFont{10}{8}{\rmdefault}{\mddefault}{\updefault}{\color[rgb]{0,0,0}$L+2$}%
}}}}
\end{picture}%
\caption{The cutoff function $\nu_L$ }
\label{cutofffunc}
\end{figure} 

Lemma \ref{isotopyimpliesconc} and the concordance construction described above work just as well for compact families. More precisely, let $K$ be a compact space and $K\times I\rightarrow \Riem^{+}(X)$, $(k,t)\mapsto g_{k,t}\in\Riem^{+}(X)$ be a compact family of psc-metrics. This is best thought of as a compact family of isotopies $t\mapsto g_{k, t}$ parameterised by $k\in K$. In turn, this leads to a parameterised analogue, $\Lambda_{k}$, of the $\Lambda$ term above. By compactness, there is a constant $\Lambda_{K}>0$ satisfying $\Lambda_{k}\geq \Lambda_{K}$ for all $k\in K$. Then by choosing $L>0$ so that $|\dot{{\nu_{L}}}|,|\ddot{{\nu_{L}}}|\leq\Lambda_{K}$, we guarantee that the metric $g_{k, {\nu_{L}(t)}}+dt^{2}$ has positive scalar curvature for all $k\in K$. This gives us the following lemma.
\begin{lemma}\label{welldefconc} Let $K$ be a compact space and let $g_{k,t}\in\Riem^{+}(X)$ denote a continuous family of metrics with respect to the parameter $(k,t)\in K\times I$. Then there is a constant $L_{K}\geq 1$, for which the map $K\rightarrow \Riem^{+}(X\times[0,L+2], X\times\{0\}\sqcup X\times\{L_{K}+2\} )$ defined by:
$$k\longmapsto \bar{g}_{k}:=g_{k, {\nu_{L_{K}}(t)}}+dt^{2},$$
defines a continuous family of concordances on $X\times [0,L_{K}+2]$.
\end{lemma}
\noindent The constant $L_{K}$ above will also bound the first and second derivatives of $\tau\nu_{L_{K}}$ where $\tau\in[0,1]$ is a constant. Thus, we obtain the following useful corollary.
\begin{corollary}\label{conchomot}
The map $K\times I\rightarrow \Riem^{+}(X\times[0,L+2], X\times\{0\}\sqcup X\times\{L_{K}+2\} )$ defined by:
$$ (k,\tau)\longmapsto \bar{g}_{k,\tau}:= g_{k, {\tau\nu_{L_{K}}(t)}}+dt^{2},$$
determines a homotopy through concordances between the family of trivial concordances $g_k+dt^{2}$ on $X\times [0,L_{K}+2]$ and the family $\bar{g}_{k}$ above, where $k\in K$.
\end{corollary}

Although we frequently consider concordances on cylinders of the form $X\times[0, L+2]$, for some $L>0$, it is worthwhile having a means of viewing all such concordances on the same cylinder, $X\times I$. With this in mind, we specify a family of diffeomorphisms:
$$\xi_{L}:[0,1]\longrightarrow [0,L],$$
parameterised by $L>0$ and satisfying:
\begin{enumerate}
\item[(i.)] $\xi_{L}(t)=t$ when $t$ is near zero,
\item[(ii.)] $\xi_{L}(t)=L-1+t$ when $t$ is near $1$.
\end{enumerate}
Thus, any concordance $\bar{g}\in\Riem^{+}(X\times[0,L+2], X\times\{0\}\sqcup X\times\{L_{K}+2\} )$ gives rise to a concordance, $(\Id_{X}\times\xi_{L+2})^{*}\bar{g}$ on $X\times I$. Later, when dealing various concordances on cylinders $X\times[0,L+2]$ for varying $L$, we will make use of this identification to compare concordances on the same space.

\section{Standard metrics on the disk and sphere}\label{disksection}

In this section we recall some well known standard metrics on the disk and sphere. These metrics will be rotationally symmetric and include ``torpedo metrics" on the disk as well as a generalisation called an ``almost torpedo metric".

\subsection{Embeddings, Torpedos and Warping Functions}
For any $\rho>0$, we denote by $D^{n}(\rho):=\{x\in\mathbb{R}^{n}:|x|\leq \rho\}$ and $S^{n}(\rho):=\{x\in\mathbb{R}^{n+1}:|x|=\rho\}$, the $n$-dimensional Euclidean disk and sphere of radius $\rho$. As usual, $D^{n}:=D^{n}(1)$ and $S^{n}:=S^{n}(1)$ denote the standard unit objects. Let $ds_{n}^{2}$ denote the standard round metric of radius $1$ on $S^{n}$. This metric may be obtained by the embedding into Euclidean space:
\begin{equation*}
\begin{split}
%F_{n\mbox{-}\rround}: 
(0,\pi)\times{S^{n-1}}&\longrightarrow\mathbb{R}\times\mathbb{R}^{n}\\
(r,\theta)&\longmapsto(\cos{r},\sin{r}.\theta),
\end{split}
\end{equation*}
and computed as:
\begin{equation*}
ds_n^{2}=dr^{2}+\sin^{2}(r)ds_{n-1}^{2}.
\end{equation*}
Strictly speaking, in these coordinates this metric is defined on the cylinder
$(0,\pi)\times S^{n-1}$. However, the behaviour of the function, $\sin$,
near the end points of the interval $(0,\pi)$ gives the cylinder
$(0,\pi)\times S^{n-1}$ the geometry of a round $n$-dimensional sphere which is missing a pair of antipodal points. Such a metric
extends uniquely onto the sphere. 

We now consider a generalisation of this embedding. We begin by replacing $\cos(r)$ with $\alpha(r)$ and $\sin(r)$ with $\beta(r)$, where $\alpha, \beta:[0,b]\rightarrow [0,\infty]$ are smooth functions satisfying the following conditions:
\begin{enumerate}
\item[]\begin{equation}
\begin{array}{clll}\label{beta0}
\mathrm{(i.)} \quad &\beta(r)>0, \text{ for all } r\in(0,b), &&\\
\mathrm{(ii.)} \quad &\beta(0)=0, \quad\dot{\beta}(0)=1, \quad\beta^{(even)}(0)=0,&&\\
\mathrm{(iii.)} \quad &\beta(b)=0, \quad\dot{\beta}(b)=-1, \quad\beta^{(even)}(b)=0.&&
\end{array}
\end{equation}
\item[]\begin{equation}\label{alpha0}
\alpha(r)=\alpha_0-\int_{0}^{r}\sqrt{1-\dot{\beta}(u)^{2}}du, \quad \text{ where $\alpha_0=\int_{0}^{\frac{b}{2}}\sqrt{1-\dot{\beta}(u)^{2}}du$.  }
\end{equation}
\end{enumerate}
The functions $\alpha$ and $\beta$ behave like $\cos$ and $\sin$ at the endpoints. Moreover, $\alpha$ is determined completely by $\beta$ so as to satisfy $\dot{\alpha}^{2}+\dot{\beta}^{2}=1$. Thus, the curve $[0,b]\rightarrow \mathbb{R}^{2}$ given by $r\mapsto (\alpha(r), \beta(r))$ is a unit speed curve. The constant $\alpha_0$ is somewhat arbitrary; see remark \ref{alphaconstant}. We now consider the map, $F_{\beta}$, defined by:
\begin{equation*}
\begin{split}
F_\beta:(0,b)\times{S^{n-1}}&\longrightarrow\mathbb{R}^{n}\times\mathbb{R},\\
(r,\theta)&\longmapsto(\beta(r).\theta, \alpha(r)).
\end{split}
\end{equation*}
\begin{proposition}
For any smooth functions $\alpha, \beta:[0,b]\rightarrow [0,\infty)$ satisfying the conditions laid out in \ref{alpha0} and \ref{beta0}, the map $F_\beta$ above is an embedding.
\end{proposition}
\begin{proof}
Injectivity of $F_{\beta}$ is guaranteed by the fact that
$\beta(r)>0$ when $r\in(0,b)$ and that $\alpha$ is strictly monotonic. The maximality of
the rank of the derivative of $F_{\beta}$ follows from an easy
calculation. 
\end{proof}
\noindent The induced metric $g_{\beta}$, obtained by pulling back the Euclidean metric on $\mathbb{R}^{n+1}$ via $F_{\beta}$, is computed as:
\begin{equation*}
\begin{split}
g_{\beta}:=F_{\beta}^{*}(dx_1^{2}+dx_{2}^{2}+\cdots+dx_{n}^{2}+dx_{n+1}^{2})&=(\dot{\alpha}(r)^{2}+\dot{\beta}(r)^{2})dr^{2}+\beta(r)^{2}ds_{n-1}^{2}\\
&=dr^{2}+\beta(r)^{2}ds_{n-1}^{2}.
\end{split}
\end{equation*}
The following proposition is proved in Chapter 1, Section 3.4 of \cite{P}.
\begin{proposition}
Provided the smooth function $\beta:[0,b]\rightarrow [0,\infty)$ satisfies the conditions laid out in \ref{beta0}, the metric $g_{\beta}$ extends uniquely to a rotationally symmetric metric on $S^{n}$. Furthermore, if we drop condition (iii) of \ref{beta0} and simply insist that $\beta(b)>0$, this metric is now a smooth rotationally symmetric metric on the disk $D^{n}$.
\begin{remark}\label{alphaconstant}
The constant $\alpha_0$ in the definition of $\alpha$ above is defined simply to ``centre" the image of $F_{\beta}$ around the origin. Replacing it with zero or any other constant would not affect the induced metric. %In fact, later we will want to shift this embedding away from the origin by adding a large positive constant to $\alpha_0$.
\end{remark}
\end{proposition}
A straightforward calculation gives that the scalar curvature, $s_\beta$, of the warped product metric, $dr^{2}+\beta(r)^{2}ds_{n-1}^{2}$, is given by the formula: 
\begin{equation}\label{Rcurv}
s_\beta(r,\theta)=-2(n-1)\frac{{\ddot{\beta}}(r)}{{\beta}(r)}+(n-1)(n-2)\frac{1-{\dot{\beta}}(r)^{2}}{{\beta}(r)^{2}}.
\end{equation}
Below we set out some sufficient conditions on the smooth function $\beta:[0,b]\rightarrow[0,\infty)$, which along with \ref{beta0} guarantee the metric $g_{\beta}$ has positive scalar curvature.
\begin{equation}\label{beta}
\begin{array}{ll}
\mathrm{(i.)} &\ddot{\beta}\leq 0 \text{ and } \dddot{\beta}(0)<0.\\ 
\mathrm{(ii.)} &\text{When $r$ is near but not at $0$}, \ddot{\beta}(r)<0.\\
\mathrm{(iii.)} &\dddot{\beta}(b)>0, \text{ while } \ddot{\beta}(r)<0  \text{ when $r$ is near but not at $b$}.
\end{array}
\end{equation}
In \cite[Proposition
  1.6]{Walsh1} we prove the following.
\begin{proposition}
Let $n\geq 3$. For any  smooth function $\beta:[0,b]\rightarrow[0,\infty)$ satisfying conditions \ref{beta0} and \ref{beta} above, the metric $dr^{2}+\beta(r)^{2}ds_{n-1}^{2}$ on $(0,b)\times S^{n-1}$ determines a smooth rotationally symmetric metric on $S^{n}$ with positive scalar curvature. Furthermore, if we drop condition (iii) of \ref{beta0} and instead insist that $\beta(b)>0$, this metric determines a smooth rotationally symmetric psc-metric on $D^{n}$.
\end{proposition}

We now consider an important example of a rotationally symmetric metric on the disk.
For any $\delta>0$ and $\lambda\geq 0$, let $\eta_{\delta,\lambda}:[0, \frac{\pi}{2}+\lambda]\rightarrow[0,1]$ be {\bf any} smooth function
which satisfies the following conditions:
\begin{enumerate}
\item[(i.)] $\eta_{\delta,\lambda}(r)=\delta\sin{\frac{r}{\delta}}$ when $r$ is near $0$,
\item[(ii.)] $\eta_{\delta, \lambda}(r)=\delta$ when $r\geq\delta\frac{\pi}{2}$,
\item[(iii.)] $\ddot{\eta}_{\delta,\lambda}(r)\leq 0$,
\item[(iv.)] the $k^{\mathrm{th}}$ derivative at $\delta\frac{\pi}{2}$, $\eta_{\delta,\lambda}^{(k)}(\delta\frac{\pi}{2})=0$ for all $k\geq 1$.
\end{enumerate} 
We will assume that for any pair $\lambda,\lambda'\geq 0$, ${\eta_{\delta,\lambda}}$ and $\eta_{\delta,\lambda'}$ agree on the interval $[0,\delta\frac{\pi}{2}]$.  
%More
%generally, we define a family of functions:
%$$\eta_{\delta,\lambda}:[0,\delta\frac{\pi}{2}+\lambda]\rightarrow[0,\delta],$$ by the formula:
%\begin{equation*}
%\begin{array}{c}
%\eta_{\delta,\lambda}(r)=\delta \eta_{1, \lambda}(\frac{r}{\delta}),
%\end{array}
%\end{equation*}
%parameterised by $(\delta, \lambda)\in(0,\infty)\times[0,\infty)$.

The function $\eta_{\delta,\lambda}$ is known as a {\em $\delta$-torpedo function} with {\em neck length} $\lambda$ or $\delta-\lambda$-torpedo function. As it satisfies conditions (i.) and (ii.) of \ref{beta0} and has $\eta_{\delta,\lambda}( \frac{\pi}{2}+\lambda)>0$, it gives rise to a smooth metric on $D^{n}$. The resulting metric is called a {\em torpedo metric of radius $\delta$} and {\em neck length} $\lambda$ (or $\delta-\lambda$-torpedo metric). It is denoted $g_{\tor}^{n}(\delta)_{\lambda}$ and given by the formula:
$$g_{\tor}^{n}(\delta)_{\lambda}=dr^{2}+\eta_{\delta, \lambda}(r)^{2}ds_{n-1}^{2},$$ 
where $r\in [0, \delta\frac{\pi}{2}+\lambda]$.
Such a metric is rotationally symmetric metric on the disk $D^{n}$ and roughly, a round hemisphere of radius $\delta$ near the centre of the disk and a round cylinder of radius $\delta$ near its boundary. Indeed, the metric always takes this cylindrical form on the annular region where $r\in[\delta\frac{\pi}{2}, \delta\frac{\pi}{2}+\lambda]$. This region of the disk is known known as the {\em neck} of the torpedo metric and is isometric to a round cylinder of radius $\delta$ and length $\lambda$; see Fig. \ref{torpfunc}. 

\begin{figure}[!htbp]
\vspace{-1cm}
\hspace{1.5cm}
\begin{picture}(0,0)
\includegraphics{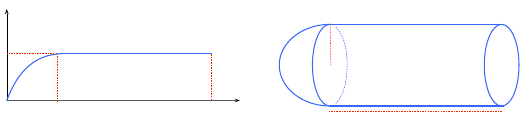}%
\end{picture}
\setlength{\unitlength}{3947sp}%
\begingroup\makeatletter\ifx\SetFigFont\undefined%
\gdef\SetFigFont#1#2#3#4#5{%
  \reset@font\fontsize{#1}{#2pt}%
  \fontfamily{#3}\fontseries{#4}\fontshape{#5}%
  \selectfont}%
\fi\endgroup%
\begin{picture}(5079,1559)(1902,-7227)
\put(1850,-7250){\makebox(0,0)[lb]{\smash{{\SetFigFont{10}{8}{\rmdefault}{\mddefault}{\updefault}{\color[rgb]{0,0,0}$0$}%
}}}}
\put(1800,-6750){\makebox(0,0)[lb]{\smash{{\SetFigFont{10}{8}{\rmdefault}{\mddefault}{\updefault}{\color[rgb]{0,0,0}$\delta$}%
}}}}
\put(4500,-6680){\makebox(0,0)[lb]{\smash{{\SetFigFont{10}{8}{\rmdefault}{\mddefault}{\updefault}{\color[rgb]{1,0,0}$\delta$}%
}}}}
\put(2260,-7250){\makebox(0,0)[lb]{\smash{{\SetFigFont{10}{8}{\rmdefault}{\mddefault}{\updefault}{\color[rgb]{0,0,0}$\delta\frac{\pi}{2}$}%
}}}}
\put(3400,-7250){\makebox(0,0)[lb]{\smash{{\SetFigFont{10}{8}{\rmdefault}{\mddefault}{\updefault}{\color[rgb]{0,0,0}$\delta\frac{\pi}{2}+\lambda$}%
}}}}
\put(5100,-7320){\makebox(0,0)[lb]{\smash{{\SetFigFont{10}{8}{\rmdefault}{\mddefault}{\updefault}{\color[rgb]{1,0,0}$\lambda$}%
}}}}
\end{picture}%
\caption{A $\delta$-torpedo function $\eta_{\delta, \lambda}$ (left) and the resulting torpedo metric $g_{\tor}^{n}(\delta)_{\lambda}$ on the disk (right)}
\label{torpfunc}
\end{figure}   
To avoid any misunderstanding, we emphasise that the torpedo metric depicted in the right image of this figure is {\em not} obtained by rotating the curve depicted in the left image which intersects the horizontal at an angle of $\frac{\pi}{4}$. Instead it is obtained by rotating a curve which intersects the horizontal at $0$ as a circular arc (and thus at an angle of $\frac{\pi}{2}$). We now make a number of elementary observations about torpedo metrics in the following proposition.

\begin{remark}
It is convenient, for certain topological arguments later on, that our definition of the torpedo metric above is slightly more general than that given in other sources such as \cite{Walsh1}. Instead of specifying only one torpedo function, we allow for each pair $(\delta, \lambda)$, $\eta_{\delta, \lambda}$ to be any function which satsfies properties (i) through (iv) above. Thus for each pair $(\delta, \lambda)$, we have a collection of torpedo functions (metrics) differing only marginally from each other; see part (i.) of Lemma \ref{torprop} below. In particular, for any $(\delta, \lambda)$, any one of the torpedo functions (metrics) in the associated collection would suffice for all purposes in \cite{Walsh1}. 
\end{remark}

%Returning to the sphere $S^{n}$, we denote by $g_{\Dtor}^{n}(\delta)_{\lambda}$ the metric on the sphere obtained by gluing a pair of torpedo metrics $g_{\tor}^{n}(\delta)_{\frac{\lambda}{2}}$ on $D^{n}$ along the boundary in the obvious way. This metric is known as the {\em double torpedo metric} of radius $\delta$ and neck length $\lambda$ on the sphere $S^{n}$. As a warped product metric, it takes the form:
%$$g_{D\tor}^{n}(\delta)_{\frac{\lambda}{2}}:=dl^{2}+\bar{\eta}_{\delta, \lambda}(l)^{2}ds_{n-1}^{2},$$
%where $\bar{\eta}_{\delta, \lambda}:[0,\lambda+\delta\pi]\rightarrow[0,\infty)$ is the smooth function defined:
%\[ 
%\bar{\eta}_{\delta, \lambda}(l)=
%\begin{cases} 
  %    \eta_{\delta, \lambda} &  0\leq l\leq \delta\frac{\pi}{2}, \\
     %  \delta&\delta\frac{\pi}{2}\leq l\leq \lambda+\delta\frac{\pi}{2}, \\
      %  \eta_{\delta, \lambda}(\delta{\pi}+\lambda-l)& \lambda+ \delta\frac{\pi}{2}\leq l\leq \lambda+\delta{\pi}.
  % \end{cases}
%\]
%We will briefly make use of this double torpedo metric again in section \ref{sec5}. Returning to the case of the disk, $D^{n}$, 

\begin{proposition}\label{torprop} Let $(\delta, \lambda)\in(0,\infty)\times[0,\infty)$ be an arbitrary pair. 
\begin{enumerate}
\item[(i.)] The set of all $\delta-\lambda$-torpedo functions forms a convex subspace of the space of smooth functions: $C^{\infty}([0, \delta\frac{\pi}{2}+\lambda],[0,\infty))$. 
\item[(ii.)] Any torpedo metric $g_{\tor}^{n}(\delta)_{\lambda}$ has positive scalar curvature. 
\item[(iii.)] For any constant $B>0$, there exists a sufficiently small $\delta>0$ so that the scalar curvature of any torpedo metric $g_{\tor}^{n}(\delta)_{\lambda}$ (for any $\lambda\geq 0$) is everywhere greater than $B$.
\end{enumerate}
\end{proposition}
\begin{proof}
Part (i.) is easily verified by checking that conditions (i.) through (iv.) above are closed under linear combination.
Parts (ii.) and (iii.) follow immediately from the formula for the scalar curvature (\ref{Rcurv}) applied to $\eta_{\delta, \lambda}$.
\end{proof}

We will now consider rotationally symmetric metrics on the disk more generally. Using the previously defined notation, $\Riem(D^{n})$ denotes the space of Riemannian metrics on the disk $D^{n}$. Recall that, unlike the space $\Riem(D^{n}, \p D^{n})$, we impose no condition on the behaviour of metrics near boundary of the disk. The particular dimension $n$ will be important later. For now we assume that $n$ is fixed and $n\geq 3$.
We consider the subspace, $\Riem_{O(n)}(D^{n})$, of metrics on $D^{n}$ which are invariant under the obvious action of the orthogonal group $O(n)$. Each metric $g\in \Riem_{O(n)}(D^{n})$ is rotationally symmetric and takes the form:
$$ g=\alpha_1(r)^{2}dr^{2}+\alpha_2(r)^{2}ds_{n-1}^{2},$$
where $r\in[0,1]$ is the radial distance coordinate on $D^{n}$ and $\alpha_1, \alpha_2:[0,1]\rightarrow [0, \infty)$ are smooth functions. We now make a change of coordinates by defining: $$l(r):=\int_{0}^{r}\alpha_1(u)du.$$ The function $l(r)$ is defined on $[0,1]$ and satisfies $l'(r)>0$ for all $r\in[0,1]$. Thus, it is invertible and we denote its inverse by $r(l)$ defined on $[0, b]$, where $b=l(1)$, the radius of the disk ${D^{n}}$ under the metric $g$. In these new coordinates, the metric $g$ takes the form:
$$g=dl^{2}+\omega(l)^{2}ds_{n-1}^{2},$$ where $\omega(l)=\alpha_2(r(l))$ and $l\in[0,b]$. The function $\omega:[0,b]\rightarrow[0,\infty)$ is called the {\em warping function} for the metric.
\begin{remark}
Strictly speaking, the metrics $g=\alpha_1(r)^{2}dr^{2}+\alpha_2(r)^{2}ds_{n-1}^{2}$ and $dl^{2}+\omega(l)^{2}ds_{n-1}^{2}$ above are not equal, only isometric. The former is a metric on $D^{n}(1)$ and the latter on $D^{n}(b)$. However, the map $r\mapsto l(r)$ which identifies the radial coordinates provides a canonical isometry. Thus we feel it is reasonable to slightly abuse notation and write 
$g=dl^{2}+\omega(l)^{2}ds_{n-1}^{2}$. \end{remark}
\noindent We summarise the above discussion in the following proposition.
\begin{proposition}
For each metric $g\in \Riem_{O(n)}(D^{n})$, there is a unique warping function, $\omega_g:[0,b_g]\rightarrow[0,\infty)$, where $b_g$ is the radius of $D^{n}$ with respect to $g$.
\end{proposition}

From earlier, we know that $\omega:[0,b]\rightarrow [0,\infty)$ is a warping function for a metric in $\Riem_{O(n)}(D^{n})$ if and only if it satisfies condition (i) of \ref{beta0} and $\omega(b)>0$. 
Consider now the subspace $\Riem_{O(n)}^{+}(D^{n})\subset\Riem_{O(n)}(D^{n})$
consisting of rotationally symmetric metrics on $D^{n}$ with positive scalar curvature. Thus:
$$\Riem_{O(n)}^{+}(D^{n}):=\Riem_{O(n)}(D^{n})\cap\Riem^{+}(D^{n}).$$
%Consider also a metric $g\in\Riem_{O(n)}(D^{n})$ with warping function $\omega$. %The scalar curvature, $s_g$, of the warped product metric, $g=dr^{2}+\omega(r)^{2}ds_{n-1}^{2}$, is computed as:
%\begin{equation}\label{Rcurv}
%s_g(r)=s_{\omega}(r)=-2(n-1)\frac{{\ddot{\omega}}(r)}{{\omega}(r)}+(n-1)(n-2)\frac{1-{\dot{\omega}}(r)^{2}}{{\omega}(r)^{2}}.
%\end{equation}
%It is reasonable to regard the formula on the right hand side of \ref{Rcurv} as determining a function $s_{\omega}$, arising only from $\omega$ (and technically $n$ although as we are assuming a fixed $n\geq 3$ we feeling justified in suppressing it from the notation), without reference to the metric. 
The subspace of $\Riem_{O(n)}^{+}(D^{n})$ consisting of {\em torpedo metrics} on the disk $D^{n}$ is denoted $\Riem_{\cT}^{+}(D^{n})$. More precisely:
$$ \Riem_{\cT}^{+}(D^{n}):= \{g\in \Riem_{O(n)}^{+}D^{n}:\omega_g=\eta_{\delta, \lambda} \text{ for some } (\delta, \lambda)\in(0,\infty)\times [0,\infty) \}.$$
It is also useful to specify the subspaces $\Riem_{\cT(1)}^{+}(D^{n})$, consisting of torpedo metrics of radius $1$ and arbitrary neck-length and $\Riem_{\cT(1,0)}^{+}(D^{n})$ which consists of torpedo metrics with radius $1$ and neck-length zero.

\begin{proposition} \label{torpcontract}
Assuming $n\geq 3$, the space $\Riem_{\cT}^{+}(D^{n})$ is a contractible subspace of $\Riem_{O(n)}^{+}(D^{n})$. 
\end{proposition} 
\begin{proof}
First we will describe a deformation retract from the space $\Riem_{\cT}^{+}(D^{n})$ to the subspace $\Riem_{\cT(1)}^{+}(D^{n})$.
Suppose $\eta:[0,b]\rightarrow [0,\infty)$ is a torpedo function of radius $\delta$. We will not concern ourselves yet with the neck-length of this torpedo function except to observe that $b\geq\delta\frac{\pi}{2}$. For any $\kappa>0$, a straightforward calculation shows that the mapping $r\mapsto \kappa\eta(\frac{r}{\kappa})$ determines a torpedo function of radius $\kappa\delta$ defined on the domain $[0, \kappa b]$. In turn this determines a new torpedo metric of radius $\kappa\delta$. In particular, as $\eta$ has radius $\delta$, setting $\kappa=\frac{1}{\delta}$ results in a torpedo function of radius $1$. The metrics resulting from this process (which are always torpedo metrics) are precisely the result of a homothetic rescaling of other torpedo metrics. Now suppose $g\in \Riem_{\cT}^{+}(D^{n})$ is a torpedo metric with warping function $\eta_{g}:[0,b_g]\rightarrow[0,\infty)$ and radius $\delta_{g}=\eta_{g}(b_g)$. By replacing $\eta_{g}$ with the warping function given by: $$r\mapsto \left[\tau+\frac{1-\tau}{\delta_{g}}\right]\eta_{g}\left(\frac{r}{\tau+\frac{1-\tau}{\delta_{g}}}\right),$$ where $r\in[0, {b_g}(\tau+\frac{1-\tau}{\delta_{g}})]$ and $\tau\in[0,1]$, we obtain a deformation retract of the space $\Riem_{\cT}^{+}(D^{n})$ to the subspace $\Riem_{\cT(1)}^{+}(D^{n})$.

By continuously shrinking all torpedo necks to zero, we obtain a further deformation retract from $\Riem_{\cT(1)}^{+}(D^{n})$ onto the subspace $\Riem_{\cT(1,0)}^{+}(D^{n})$. Every such torpedo metric is described by a torpedo function of the form $\eta:[0,\frac{\pi}{2}]\rightarrow [0,\infty)$. An elementary calculation shows that this set of torpedo functions is closed under linear combination. Thus, any inclusion $\{\eta\} \hookrightarrow \Riem_{\cT(1,0)}(D^{n})$ forms part of a deformation retract.
\end{proof}

\noindent To simplify the notation, we will write $g_{\tor}^{n}(\delta):=g_{\tor}^{n}(\delta)_{1}$ to denote a radius $\delta$ torpedo metric with neck-length $1$. For many of our purposes (especially in the earlier sections of the paper) the neck-length of the torpedo will not matter, only the radius. Later on the neck-length will matter slightly and we will reintroduce appropriate notation. Finally, we will write simply $g_{\tor}^{n}:=g_{\tor}^{n}(1)$ to denote a torpedo metric with radius and neck-length $1$.

\subsection{Almost Torpedo Metrics}
We close by defining something of a generalisation of the torpedo metric which we will make use of in the next section. A smooth function $\omega:[0,b]\rightarrow [0,\infty)$ is called an {\em almost torpedo function}, if it satisfies the following properties.
\begin{enumerate}
\item[(i.)] $j_{0}^{\infty}(\omega)=j_{0}^{\infty}(\sin)$.
\item[(ii.)] $\dot{\omega}(r)\geq 0$ for all $r\in[0,b]$.
\item[(iii.)] $\ddot{\omega}(r)<0$ when $r$ is near but not at zero.
%\item[(iv.)] $\omega(b)\frac{\pi}{2}\leq b.$
\item[(iv.)] The corresponding scalar curvature function, $s_{\omega}$, satisfies $s_{\omega}(r)>0$ for all $r\in[0,b]$. 
\end{enumerate}

Finally, we denote by $\Riem_{\cA\cT}^{+}(D^{n})$, the subspace of $\Riem_{O(n)}^{+}(D^{n})$ defined:
$$ \Riem_{\cA\cT}^{+}(D^{n}) := \{g\in \Riem_{O(n)}^{+}(D^{n}): \omega_{g} \text{ is an almost torpedo function} \}.$$
We call $\Riem_{\cA\cT}^{+}(D^{n})$ the space of {\em almost torpedo metrics} on the disk $D^{n}$.
In conclusion, we recall that the spaces of this section include as follows:
$$\Riem_{\cT(1,0)}^{+}(D^{n})\subset\Riem_{\cT(1)}^{+}(D^{n})\subset\Riem_{\cT}^{+}(D^{n})\subset\Riem_{\cA\cT}^{+}(D^{n})\subset \Riem_{O(n)}^{+}(D^{n}).$$

\section{Revisiting the Theorems of Gromov-Lawson and Chernysh}\label{GLreview}

In this section, we will briefly review the technique of geometric surgery on positive scalar curvature metrics pioneered by  Gromov-Lawson in \cite{GL} as well as an important theorem of Chernysh \cite{Che} which strengthens the original work. For the reader interested in more detail, there are a variety of sources. As mentioned in the introduction, recent work by Ebert and Frenck in \cite{EF} contains an extremely thorough proof of Chernysh's Theorem as well as a comprehensive recounting of the work of Gromov and Lawson in \cite{GL}. Regarding the latter, the authors draw from work done in \cite{RS} and \cite{Walsh1} which also contain detailed accounts of the original Gromov-Lawson construction. 

As stated earlier, $X$ is always a smooth compact $n$-dimensional manifold with empty boundary, while $W$ is a smooth compact $(n+1)$-dimensional manifold with non-empty closed boundary. Throughout, $\p W=X$. 

\subsection{Surgery} 
Suppose $\phi:S^{p}\times D^{q+1}\hookrightarrow X$ is an embedding, where $\dim X=n=p+q+1$.
Recall that a {\em surgery} on a smooth $n$-dimensional manifold $X$, with respect to the embedding $\phi$, is the construction of a manifold $X'$ by removing the image of $\phi$ from $X$ and using the restricted map $\phi|_{S^{p}\times S^{q}}$ to attach $D^{p+1}\times S^{q}$ along the common boundary (making appropriate smoothing adjustments). The resulting manifold $X'$, depicted in Fig. \ref{surgerydef}, is therefore defined as:
\begin{equation*}
X':=(X\setminus\phi(S^{p}\times {\oD}))\cup_{\phi}D^{p+1}\times S^{q}.
\end{equation*}
It is always possible to reverse such a surgery by performing a {\em complementary surgery.} The newly attached $D^{p+1}\times S^{q}\subset X'$ can be regarded as the image of an embedding $\phi':D^{p+1}\times S^{q}\rh X'$ whose restriction to the boundary coincides with $\phi$. Peforming a surgery with respect to $\phi'$ involves removing $D^{p+1}\times S^{q}$ and reattaching the originally removed $S^{p}\times D^{q+1}$. The resulting manifold is diffeomorphic to the original manifold $X$.

\begin{figure}[!htbp]
\vspace{5cm}
\hspace{1.5cm}
\begin{picture}(0,0)
\includegraphics{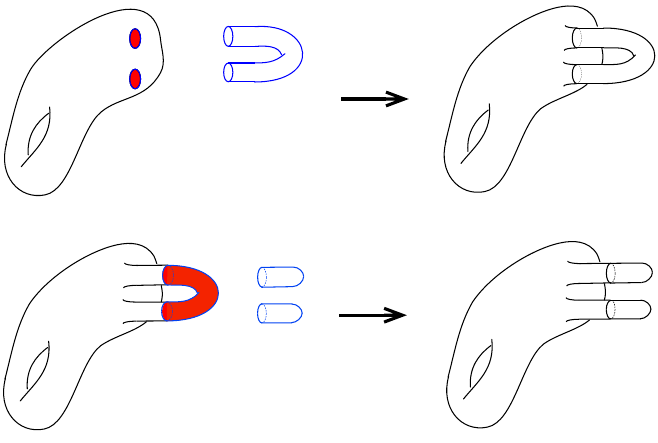}%
\end{picture}
\setlength{\unitlength}{3947sp}%
\begingroup\makeatletter\ifx\SetFigFont\undefined%
\gdef\SetFigFont#1#2#3#4#5{%
  \reset@font\fontsize{#1}{#2pt}%
  \fontfamily{#3}\fontseries{#4}\fontshape{#5}%
  \selectfont}%
\fi\endgroup%
\begin{picture}(5079,1559)(1902,-7227)
\put(1900,-4000){\makebox(0,0)[lb]{\smash{{\SetFigFont{10}{8}{\rmdefault}{\mddefault}{\updefault}{\color[rgb]{0,0,0}$X$}%
}}}}
\put(2200,-4250){\makebox(0,0)[lb]{\smash{{\SetFigFont{10}{8}{\rmdefault}{\mddefault}{\updefault}{\color[rgb]{1,0,0}$S^{p}\times D^{q+1}$}%
}}}}
\put(3450,-4600){\makebox(0,0)[lb]{\smash{{\SetFigFont{10}{8}{\rmdefault}{\mddefault}{\updefault}{\color[rgb]{0,0,1}$D^{p+1}\times S^{q}$}%
}}}}
\put(7200,-4000){\makebox(0,0)[lb]{\smash{{\SetFigFont{10}{8}{\rmdefault}{\mddefault}{\updefault}{\color[rgb]{0,0,0}$X'$}%
}}}}
\put(1900,-6000){\makebox(0,0)[lb]{\smash{{\SetFigFont{10}{8}{\rmdefault}{\mddefault}{\updefault}{\color[rgb]{0,0,0}$X'$}%
}}}}
\put(2800,-6500){\makebox(0,0)[lb]{\smash{{\SetFigFont{10}{8}{\rmdefault}{\mddefault}{\updefault}{\color[rgb]{1,0,0}$D^{p+1}\times S^{q}$}%
}}}}
\put(3800,-6500){\makebox(0,0)[lb]{\smash{{\SetFigFont{10}{8}{\rmdefault}{\mddefault}{\updefault}{\color[rgb]{0,0,1}$S^{p}\times D^{q+1}$}%
}}}}
\put(7200,-6000){\makebox(0,0)[lb]{\smash{{\SetFigFont{10}{8}{\rmdefault}{\mddefault}{\updefault}{\color[rgb]{0,0,0}$X$}%
}}}}
\end{picture}%
\caption{Performing a surgery on an embedded $S^{p}\times D^{q+1}$ in $X$  to obtain $X'$ (top) and performing a {\em reverse} surgery on an embedded $D^{p+1}\times S^{q}$ in $X'$ to restore the smooth topology of $X$ (bottom)}
\label{surgerydef}
\end{figure}   

The {\em trace of the surgery on $\phi$} is the manifold $T_{\phi}$ obtained by gluing the cylinder $X\times [0,1]$ to the disk product $D^{p+1}\times D^{q+1}$ via the embedding $\phi$. This is done by attaching $X\times\{1\}$ to the boundary component $S^{p}\times D^{q+1}$ via $\phi:S^{p}\times D^{q+1}\hookrightarrow X\hookrightarrow X\times\{1\}$. After appropriate smoothing we obtain $T_{\phi}$, a smooth manifold with boundary diffeomorphic to the disjoint union $X\sqcup X'$, i.e. an (elementary) cobordism of $X$ and $X'$. As suggested in the introduction, we will be particularly interested in the following case. Suppose $X$ forms the boundary of a smooth $(n+1)$-dimensional manifold $W$. Thus $\p W=X$. Let $\phi$ and $T_{\phi}$ be as above. Then, we can form a smooth manifold $W'$ by gluing $W$ to $T_{\phi}$ by an appropriate identification of $\partial W$ with $X\times \{0\}\subset T_{\phi}$; see Fig. \ref{surgerytracedef}. There are, as mentioned, various ``smoothing" issues involved in such a construction which we will not go into here; see Ch. 8, Sec. 2 of \cite{Hirsch} for details.
\begin{figure}[!htbp]
\vspace{-0.5cm}
\hspace{1.5cm}
\begin{picture}(0,0)
\includegraphics{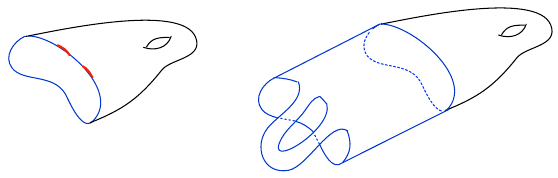}%
\end{picture}
\setlength{\unitlength}{3947sp}%
\begingroup\makeatletter\ifx\SetFigFont\undefined%
\gdef\SetFigFont#1#2#3#4#5{%
  \reset@font\fontsize{#1}{#2pt}%
  \fontfamily{#3}\fontseries{#4}\fontshape{#5}%
  \selectfont}%
\fi\endgroup%
\begin{picture}(5079,1559)(1902,-7227)
\put(2800,-7000){\makebox(0,0)[lb]{\smash{{\SetFigFont{10}{8}{\rmdefault}{\mddefault}{\updefault}{\color[rgb]{0,0,0}$W$}%
}}}}
\put(4800,-6600){\makebox(0,0)[lb]{\smash{{\SetFigFont{10}{8}{\rmdefault}{\mddefault}{\updefault}{\color[rgb]{0,0,1}$T_{\phi}$}%
}}}}
\put(5800,-7000){\makebox(0,0)[lb]{\smash{{\SetFigFont{10}{8}{\rmdefault}{\mddefault}{\updefault}{\color[rgb]{0,0,0}$W'=W\cup T_{\phi}$}%
}}}}
\end{picture}%
\caption{The manifold with boundary, $W'$, obtained by attaching to $W$ the trace of the surgery on $\phi$, $T_{\phi}$}
\label{surgerytracedef}
\end{figure}   

\subsection{The Gromov-Lawson Construction}\label{GLconstruction}
We begin with an embedding $\phi:S^{p}\times D^{q+1}\rh X^{n}$ satisfying $p+q+1=n$ and $q\geq 2$. 
%We denote $N_{\rho}:=\phi(S^{p}\times D^{q+1}(\rho))$ (with $N:=\phi(S^{p}\times D^{q+1})$) where $\rho\in(0,1]$. Consider $g\in\Riem^{+}(X)$, an arbitrary metric of positive scalar curvature on $X$. 
The Gromov-Lawson construction allows for the construction of a new psc-metric $g'$ on the manifold $X'$ obtained from $X$ by surgery on the embedding $\phi$. Before describing it further, it will make our work a little neater if we introduce the following family of rescaling maps: 
\begin{equation*}
\begin{split}
\sigma_\rho:S^{p}\times D^{q+1}&\longrightarrow S^{p}\times D^{q+1}\\
(x,y)&\longmapsto (x, \rho y),
\end{split}
\end{equation*}
where $\rho\in(0,1]$.
We set $\phi_{\rho}:=\phi\circ\sigma_\rho$ and $N_{\rho}:=\phi_{\rho}(S^{p}\times D^{q+1})$, abbreviating $N:=N_{1}$. Thus, for any meric $g$ on $X$ and any $\rho\in(0,1]$, $\phi_{\rho}^{*}g$ is just the metric obtained by taking the restriction metric $g|_{N_{\rho}}$, pulling it back via $\phi$ to obtain the metric $\phi^{*}g|_{N_{\rho}}$ on $S^{p}\times D^{q+1}(\rho)$ and finally, via the obvious rescaling map $\sigma_\rho$, pulling it back to obtain a metric on $S^{p}\times D^{q+1}$. The benefit of this is that we always consider pull-backs of metrics which are restricted on neighbourhoods $N_{\rho}=\phi(S^{p}\times D^{q+1}(\rho))\subset X$ as metrics on the same space $S^{p}\times D^{q+1}$.
%Once again, the metric $\phi_{\rho}^{*}g$ is (up to a trivial isometry), just the metric $\phi^{*}(g|_{N_\rho})$. 

At the heart of the Gromov-Lawson surgery construction is the following fact.
\begin{theorem}[Gromov-Lawson \cite{GL}]\label{GLthm}
Let $X^{n}$ be a smooth manifold and $\phi:S^{p}\times D^{q+1}\rh X$ an embedding with $p+q+1=n$ and $q\geq 2$. Then for any psc-metric $g\in\Riem^{+}(X)$, there is a psc-metric $g_{\std}\in\Riem^{+}(X)$ so that:
\begin{enumerate}
\item[(i.)] In the neighbourhood $N_{\frac{1}{2}}=\phi_{\frac{1}{2}}(S^{p}\times D^{q+1})$,  $g_{\std}$ pulls back to the metric:
$$\phi_{\frac{1}{2}}^{*}g_{\std}=ds_{p}^{2}+g_{\tor}^{q+1}.$$ 
\item[(ii.)] Outside $N=\phi(S^{p}\times D^{q+1})$, $g_{\std}=g$.
\end{enumerate}
\end{theorem}
\noindent The metric $g_{\std}$ is thus prepared for surgery (or standardised on $N_{\frac{1}{2}}$). By removing part of the standard piece taking the form $(S^{p}\times D^{q+1}, ds_{p}^{2}+g_{\tor}^{q+1})$ and replacing it with $(D^{p+1}\times S^{q}, g_{\tor}^{p+1}+ds_{q}^{2}$), we obtain a psc-metric $g'\in\Riem^{+}(X')$; see Fig. \ref{GLsurgery} below.
\begin{figure}[htb!]
%\vspace{-2cm}
\hspace{-1cm}
\begin{picture}(0,0)%
\includegraphics{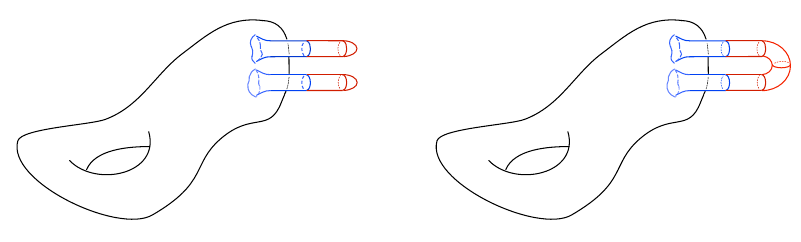}%
\end{picture}%
\setlength{\unitlength}{3947sp}%
\begingroup\makeatletter\ifx\SetFigFont\undefined%
\gdef\SetFigFont#1#2#3#4#5{%
  \reset@font\fontsize{#1}{#2pt}%
  \fontfamily{#3}\fontseries{#4}\fontshape{#5}%
  \selectfont}%
\fi\endgroup%
\begin{picture}(6000,2000)(0, 0)
\put(1500,200){\makebox(0,0)[lb]{\smash{{\SetFigFont{10}{8}{\rmdefault}{\mddefault}{\updefault}{\color[rgb]{0,0,0}$(X,g_{\std})$}%
}}}}
\put(2300,900){\makebox(0,0)[lb]{\smash{{\SetFigFont{10}{8}{\rmdefault}{\mddefault}{\updefault}{\color[rgb]{1,0,0}$ds_{p}^{2}+g_{\tor}^{q+1}$}%
}}}}

\put(4800,200){\makebox(0,0)[lb]{\smash{{\SetFigFont{10}{8}{\rmdefault}{\mddefault}{\updefault}{\color[rgb]{0,0,0}$(X',g')$}%
}}}}
\put(5600,900){\makebox(0,0)[lb]{\smash{{\SetFigFont{10}{8}{\rmdefault}{\mddefault}{\updefault}{\color[rgb]{1,0,0}$g_{\tor}^{p+1}+ds_{q}^{2}$}%
}}}}

\end{picture}%
\caption{The standardised psc-metric $g_{\std}$ (left) and the new psc-metric $g'$ on $X'$ (right)}
\label{GLsurgery}
\end{figure} 
We provide here a very brief summary of the main steps in constructing $g_{\std}$. As mentioned before, detailed accounts of this construction are contained in \cite{RS}, \cite{Walsh1} and \cite{EF} as well as the original paper \cite{GL}. 

\begin{enumerate}
\item{} Working entirely inside $N=\phi(S^{p}\times D^{q+1})$, we make a series of adjustments to the metric $g$. The first adjustment is to replace $g$ with a psc-metric which, for some $\bar{r}\in(0,1]$, satisfies the condition that the smooth curve $[0,\bar{r}]\rh X, t\rightarrow \phi(x,ty)$ is a unit speed geodesic for each $x,y\in S^{p}\times D^{q+1}$. In \cite{EF}, Ebert and Frenck define such a metric as {\em normalised} with respect to $\phi|_{S^{p}\times D^{q+1}(\bar{r})}$. That this is possible follows from standard results of Differential Topology, concerning the uniqueness of tubular neighbourhoods up to isotopy; see chapter 4, section 5 of \cite{Hirsch}. 
It is worth pointing out that in many accounts of the Gromov-Lawson construction (such as that in \cite{Walsh1}), the embedding $\phi$ arises from an embedding of $S^{p}$ with trivial normal bundle and via the exponential map with respect to the metric $g$. In this case, the metric is alreadly normalised with respect to the embedding and so this initial step is unnecessary.

\item{} The next stage is to construct a hypersurface $M\subset N\times[0,\infty)$ where $N\times [0,\infty)$ is equipped with the metric $g+dt^{2}$.
Letting $r$ denote the radial distance coordinate from $S^{p}\times\{0\}$ in $N$, $M$ is obtained by pushing out geodesic sphere bundles of radius $r$ along the $t$-axis with respect to a unit speed smooth curve $\gamma:[0,b]\rightarrow[0,\infty)\times[0,\infty)$ in the $(t-r)$-plane of the type shown in Fig. \ref{GLcurve}; see section 2.4 of \cite{Walsh1} for a detailed description. Such a curve, called a {\em Gromov-Lawson curve}, is assumed to begin at the point $\gamma(0)=(\bar{t}, 0)$, for some possibly large $\bar{t}>0$, as a piece of circular arc (thus intersecting the horizontal $t$-axis at an angle of $\frac{\pi}{2}$). It proceeds roughly as shown in the figure, with two distinct upward bends, before finishing as a vertical line segment defined $\gamma(l)=(0, \bar{r}+l-b)$ when $l$ is near $b$. This latter condition means that this metric smoothly transitions to $g$ near the $\p N$. It is not difficult to construct an obvious diffeomorphism to pull back the metric to $N$ and via the smooth transition to a metric on $X$. A lengthy computation then shows that the curve $\gamma$ can be chosen to ensure that the resulting metric on $X$ has positive scalar curvature. 

\begin{figure}[!htbp]
\vspace{1cm}
\hspace{-0cm}
\begin{picture}(0,0)
\includegraphics{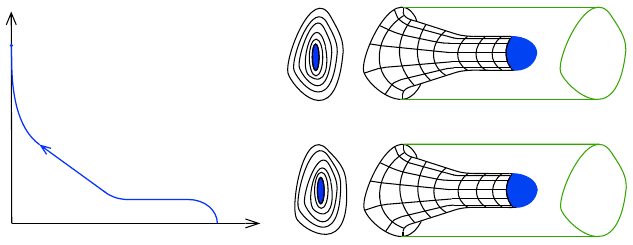}%
\end{picture}
\setlength{\unitlength}{3947sp}%
\begingroup\makeatletter\ifx\SetFigFont\undefined%
\gdef\SetFigFont#1#2#3#4#5{%
  \reset@font\fontsize{#1}{#2pt}%
  \fontfamily{#3}\fontseries{#4}\fontshape{#5}%
  \selectfont}%
\fi\endgroup%
\begin{picture}(5079,1559)(1902,-7227)
\put(1750,-5500){\makebox(0,0)[lb]{\smash{{\SetFigFont{10}{8}{\rmdefault}{\mddefault}{\updefault}{\color[rgb]{0,0,0}${r}$}%
}}}}
\put(3800,-7250){\makebox(0,0)[lb]{\smash{{\SetFigFont{10}{8}{\rmdefault}{\mddefault}{\updefault}{\color[rgb]{0,0,0}$t$}%
}}}}
\put(2750,-7200){\makebox(0,0)[lb]{\smash{{\SetFigFont{10}{8}{\rmdefault}{\mddefault}{\updefault}{\color[rgb]{0,0,1}$\gamma(0)=(\bar{t},0)$}
}}}}
\put(2000,-5700){\makebox(0,0)[lb]{\smash{{\SetFigFont{10}{8}{\rmdefault}{\mddefault}{\updefault}{\color[rgb]{0,0,1}$\gamma(b)=(0, \bar{r})$}%
}}}}
\end{picture}%
\caption{The curve $\gamma$ (left), geodesic spheres on the fibres of the neighbourhood $S^{p}\times D^{q+1}\cong N$ (middle) and the hypersurface $M$ obtained by pushing out the geodesic spheres with respect to $\gamma$ (right)}
\label{GLcurve}
\end{figure}

\item{} The metric thus far constructed is likely not a product metric near $S^{p}\times \{0\}$. However it is one which is increasingly ``torpedo-like" on smaller and smaller disk fibres. It is possible, via certain isotopy arguments on geodesic sphere bundles, to replace this metric first with one which is a Riemannian submersion metric with base $ds_{p}^{2}$ and fibre $g_{\tor}^{q+1}(\delta)$ for some sufficiently small $\delta>0$. Then, using the formulae of O'Neill (and possibly choosing a smaller $\delta$), we can adjust this submersion metric to obtain the product metric $ds_{p}^{2}+g_{\tor}^{q+1}(\delta)$ near $S^{p}\times \{0\}$. 

\item{} This construction works precisely when $q\geq 2$ because the geodesic fibre spheres have dimension at least two and thus carry some scalar curvature. As these geodesic spheres are small, they are close to being round and so their  contribution to the total scalar curvature is positive and large. Indeed, by careful rescaling, their contribution can always be made to compensate for the various adjustments we make. Hence, the quantity $\delta>0$ may need to be small. Once this standard form is realised however, a final isotopy, very gradually increasing the radius $\delta$ of the torpedo neck on the $D^{q+1}$ factor, gives us the desired product $ds_{p}^{2}+g_{\tor}^{q+1}$ near to $S^{p}\times \{0\}$. Thus, we get that for some $r_{\std}\in(0,\frac{1}{2})$, we have that 
$\phi_{r_{\std}}^{*}g=ds_{p}^{2}+g_{\tor}^{q+1}$. 

\item{} Although not strictly necessary, we perform a rescaling isotopy (making adjustments in the radial directions only) to drag the standard region over the rest of $N_{\frac{1}{2}}$. Essentially we specify a continuous family of smooth increasing functions: 
$$\upsilon_{\tau}:[0,1]\rightarrow [0, 1],$$
as depicted in Fig. \ref{diffslide} and which satisfy the following conditions.
\begin{enumerate}
\item[(i.)] $\upsilon_{0}(r)=r$ for all $r\in[0,1]$.
\item[(ii.)] For all $\tau\in I$, $\upsilon_{\tau}(r)=r$ when $r$ is near $0$ and when $r\in[\frac{3}{4}, 1]$
\item[(iii.)] For all $\tau\in I$, the map $\upsilon_{\tau}$ restricts as a diffeomorphism $[0,\frac{1}{2}]\rightarrow [0, \frac{1}{2}+\tau (r_{\std}-\frac{1}{2})]$.
\end{enumerate}
This gives rise to a family of self-diffeomorphisms, $\bar{\upsilon}_{\tau}$, on $N\cong S^{p}\times D^{q+1}$ defined (under the usual radial coordinates $(x, (r, \theta))\in S^{p}\times D^{q+1}$) by the formula:
$$\bar{\upsilon}_{\tau}(x, (r, \theta)):= (x,(\upsilon_{\tau}(r), \theta)).$$ Extending each map as the identity off $N$ leads to a family of self-diffeomorphisms on $X$:
$$\bar{\upsilon}_{\tau}:X\rightarrow X,$$
where $\tau\in I$.  
Finally, the isotopy through pull-back metrics $\bar{\upsilon}_{\tau}^{*}g$ moves the metric $g=\bar{\upsilon}_{0}^{*}g$ to the desired metric $g_{\std}=\bar{\upsilon}_{1}^{*}g$ with standard region on the neighbourhood $N_{\frac{1}{2}}$.

\begin{figure}[!htbp]
\vspace{2.5cm}
\hspace{4cm}
\begin{picture}(0,0)
\includegraphics{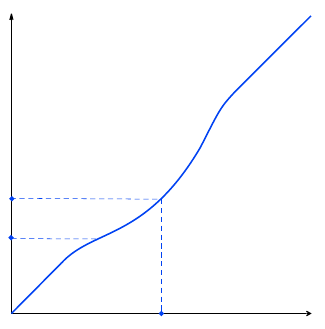}%
\end{picture}
\setlength{\unitlength}{3947sp}%\includegraphics{RelGLpic/
\begingroup\makeatletter\ifx\SetFigFont\undefined%
\gdef\SetFigFont#1#2#3#4#5{%
  \reset@font\fontsize{#1}{#2pt}%
  \fontfamily{#3}\fontseries{#4}\fontshape{#5}%
  \selectfont}%
\fi\endgroup%
\begin{picture}(5079,1559)(1902,-7227)
\put(1600,-6500){\makebox(0,0)[lb]{\smash{{\SetFigFont{10}{8}{\rmdefault}{\mddefault}{\updefault}{\color[rgb]{0,0,0}$r_{\std}$}%
}}}}
\put(800,-6200){\makebox(0,0)[lb]{\smash{{\SetFigFont{10}{8}{\rmdefault}{\mddefault}{\updefault}{\color[rgb]{0,0,0}$\frac{1}{2}+\tau (r_{\std}-\frac{1}{2})$}%
}}}}
\put(1850,-7250){\makebox(0,0)[lb]{\smash{{\SetFigFont{10}{8}{\rmdefault}{\mddefault}{\updefault}{\color[rgb]{0,0,0}$0$}%
}}}}
\put(4230,-7250){\makebox(0,0)[lb]{\smash{{\SetFigFont{10}{8}{\rmdefault}{\mddefault}{\updefault}{\color[rgb]{0,0,0}$1$}%
}}}}
\put(3040,-7300){\makebox(0,0)[lb]{\smash{{\SetFigFont{10}{8}{\rmdefault}{\mddefault}{\updefault}{\color[rgb]{0,0,0}$\frac{1}{2}$}%
}}}}
\end{picture}%
\caption{The rescaling function $\upsilon_{\tau}$}
\label{diffslide}
\end{figure}

\item{} Although not part of the original construction, it is important to mention that not only the last step but actually the entire construction can be performed so as to provide an explicit isotopy from the starting metric $g$ to the standardised metric $g_{\std}$; see \cite{Gajer}, \cite{Che}, \cite{Walsh1} and \cite{EF} for versions of this. We will briefly discuss one aspect. This concerns step 2 above and involves constructing a homotopy of the curve $\gamma$ back to the curve, $l\mapsto (0,l)$, which runs along the vertical axis; see Fig. \ref{GLcurveadj}. Importantly, this must be done so that at each stage in the homotopy, the resulting hypersurface metric still has positive scalar curvature. This is not difficult to achieve and full details can be found in \cite{Walsh1}. Essentially, it involves specifying a path-connected space of {\em admissible curves}, each of which leads to a positive scalar curvature hypersurface metric, containing the vertical axis curve as well as all Gromov-Lawson curves described above. Appropriate paths in this space back to the vertical axis, give the isotopies we desire. The following properties of an admissible curve, $\gamma:[0,b]\rightarrow[0,\infty)\times[0,\infty)$ with $\gamma(l)=(\gamma_{t}(l), \gamma_{r}(l))$, are worth mentioning.
\begin{enumerate}
\item[(i.)] The curve $\gamma$ is unit speed.
\item[(ii.)] Although $\gamma$ does not necessarily begin as a piece of circular arc, at $\gamma(0)$ it does meet the horizontal axis at an angle of $\frac{\pi}{2}$. 
\item[(iii.)] When $l$ is near $b$, $\gamma(l)=(0,\bar{r}-b+l)$.
\item[(iv.)] The smooth function $\gamma_{r}:[0,b]\rightarrow[0,\bar{r}]$ always satisfies:   
\begin{enumerate}
\item[a.] $\gamma_{r}(0)=0$, $\dot\gamma(0)=1$ and even order derivatives of $\gamma_r$ satisfy $\gamma_{r}^{(even)}(0)=0$.
\item[b.] $\dot{\gamma_r}(l)\geq 0$ for all $l\in[0,b]$.
\item[c.] $\ddot{\gamma_r}(l)\leq 0$ when $l$ is near $0$.
\end{enumerate}
\end{enumerate}
\end{enumerate}
\begin{figure}[!htbp]
\vspace{1cm}
\hspace{5cm}
\begin{picture}(0,0)
\includegraphics{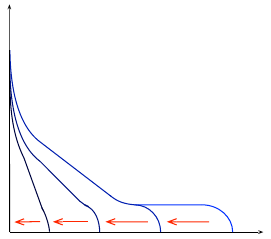}%
\end{picture}
\setlength{\unitlength}{3947sp}%
\begingroup\makeatletter\ifx\SetFigFont\undefined%
\gdef\SetFigFont#1#2#3#4#5{%
  \reset@font\fontsize{#1}{#2pt}%
  \fontfamily{#3}\fontseries{#4}\fontshape{#5}%
  \selectfont}%
\fi\endgroup%
\begin{picture}(5079,1559)(1902,-7227)
\put(1750,-5500){\makebox(0,0)[lb]{\smash{{\SetFigFont{10}{8}{\rmdefault}{\mddefault}{\updefault}{\color[rgb]{0,0,0}${r}$}%
}}}}
\put(3800,-7350){\makebox(0,0)[lb]{\smash{{\SetFigFont{10}{8}{\rmdefault}{\mddefault}{\updefault}{\color[rgb]{0,0,0}$t$}%
}}}}
\end{picture}%
\caption{A path through admissible curves which induces an isotopy of psc-hypersurface metrics back to the original one}
\label{GLcurveadj}
\end{figure}   

An isotopy form of the Gromov-Lawson construction was first shown by Gajer in \cite{Gajer} although proofs can also be found in \cite{Che}, \cite{Walsh1} and \cite{EF}, where this isotopy form of the construction is applied to compact families of psc-metrics, a crucial ingredient in the proof of Chernysh's theorem. One issue that arises when applying the Gromov-Lawson construction to a compact family of psc-metrics (not adequately addressed in \cite{Walsh1}) is the choice of neighbourhood in which to work. Suppose instead of applying the construction to a lone metric $g$, we try to apply it to a compact family of psc-metrics $K\rightarrow \Riem^{+}(X)$. In Proposition 3.4 of \cite{EF}, Ebert and Frenck show that for some sufficienly small $\bar{r}>0$,  this family can be adjusted by isotopy to one which is normalised with respect to $\phi|_{S^{p}\times D^{q+1}(\bar{r})}$. Thus, for any $(x,y)\in S^{p}\times D^{q+1}$, the curve $[0, \bar{r}]\rh X, t\mapsto \phi(x,ty)$ is a unit speed geodesic curve in $X$ for {\bf each} metric in the family. The Gromov-Lawson construction then goes through for the entire family giving rise to the following ``Family Surgery Theorem".

\begin{theorem}\label{GLcompact}\cite{Walsh1}\cite{Che}\cite{EF}
Let $X^{n}$ be a smooth manifold and $\phi:S^{p}\times D^{q+1}\hookrightarrow X^{n}$ an embedding with $p+q+1=n$ and $q\geq 2$. Let $K\rightarrow \Riem^{+}(X), k\mapsto g(k)$ be a compact family of psc-metrics. Then there is a  compact family $K\rightarrow \Riem^{+}(X), k\mapsto g_{\std}(k)$ so that:
\begin{enumerate}
\item[(i.)] For each $k\in K$, the metric $g_{std}(k)$ satisfies the condition that on the neighbourhood $N_{\frac{1}{2}}=\phi(S^{p}\times D^{q+1}(\frac{1}{2}))$, $\phi_{\frac{1}{2}}^{*}g_{\std}(k)=ds_{p}^{2}+g_{\tor}^{q+1}$, while outside of the neighbourhood $N=\phi(S^{p}\times D^{q+1})$, $g_{\std}(k)=g(k)$.
\item[(ii.)] The maps $K\rightarrow \Riem^{+}(X)$, given by $k\mapsto g(k)$ and $k\mapsto g_{\std}(k)$ are homotopy equivalent.
\end{enumerate} 
\end{theorem}

There are two important strengthenings of the Gromov-Lawson construction. One of these is the theorem of Chernysh which we will shortly discuss. Before this, it is worth mentioning a different kind of strengthening, concerning the trace of a surgery. As before, $\phi :S^{p}\times D^{q+1}\hookrightarrow X^{n}$ is an embedding with $p+q+1=n$ and $q\geq 2$ and $X'$ denotes the manifold obtained from $X$ by surgery on $\phi$. We denote by $T_{\phi}$, the trace of the surgery on $X$. Thus, $T_{\phi}$ is an $(n+1)$-dimensional manifold with boundary $X\sqcup X'$. We now fix disjoint collars $c:X\times [0,2)\rh T_{\phi}$ and $c':X'\times [0,2)\rh T_{\phi}$ around the boundary components $X$ and $X'$ respectively. As in the introduction, we denote by $\Riem^{+}(T_{\phi})$, the space of psc-metrics on $T_{\phi}$ which satisfy the condition that each $\bar{g}\in \Riem^{+}(T_{\phi})$ pulls back as:
$$c^{*}\bar{g}=g+dt^{2}  \quad \text{ and } \quad (c')^{*}\bar{g}=g'+dt^{2},$$
when restricted on $X\times I$ and $X'\times I$ respectively for some psc-metrics $g\in \Riem^{+}(X)$ and $g'\in \Riem^{+}(X')$. 
It is a fact that given a psc-metric $g\in\Riem^{+}(X)$, it is possible to construct a psc-metric $\bar{g}\in \Riem^{+}(T_{\phi})$ so that $\bar{g}|_{X'}=g'$ is a psc-metric obtained from $g$ by way of the Gromov-Lawson construction.
This fact was originally proved by Gajer in \cite{Gajer} and later by this author in \cite{Walsh1} where a mistake in the version from \cite{Gajer} is corrected. Another version of this theorem is proved by Carr in \cite{Carr}. The construction easily goes through for compact families of psc-metrics and we state it here in the form of a theorem. 

\begin{theorem} \cite{Gajer} \cite{Carr} \cite{Walsh1} Let $X, X', \phi, T_{\phi}, c, c'$ and $p,q$ be as above. Let $K\rightarrow \Riem^{+}(X), k\mapsto g(k)$ be a compact family of psc-metrics. There is a  compact family of psc-metrics $K\rightarrow \Riem^{+}(T_{\phi}), k\mapsto \bar{g}(k)$ so that for each $k\in K$, $\bar{g}(k)|_{X}=g(k)$.
\end{theorem}

\noindent It will be necessary to discuss this result in a little more detail later on. For now, we turn our attention to the other important strengthening of the Gromov-Lawson construction, due to Chernysh.

\subsection{Chernysh's Theorem}
We return to the smooth manifold $X^{n}$ and embedding $\phi:S^{p}\times D^{q+1}\hookrightarrow X^{n}$ with $p+q+1=n$. Typically we assume that $q\geq 2$. In this case we make the stronger assumption that both $p,q\geq 2$. Let $X'$ denote the manifold obtained from $X$ by surgery with respect to $\phi$. As we discussed earlier, we can equivalently regard $X$ as arising from $X'$ by surgery on an embedding $\phi': D^{p+1}\times S^{q}\rh X'$. Here $\phi_{\rho}'$, $N_{\rho}'$ play the analogous roles to $\phi_{\rho}$ and $N_{\rho}$, for $\rho\in(0,1]$. In any case, $X$ and $X'$ are mutually obtainable from each other by surgeries in codimension at least three. Under these hypotheses, Chernysh's theorem (Theorem \ref{Chernysh}) gives that the spaces $\Riem^{+}(X)$ and $\Riem^{+}(X')$ are homotopy equivalent. As we will make significant use of some of the underlying techniques, it is worth giving a brief outline of the proof. 
During this outline, we will define some spaces, adapted from some of the spaces involved in the proof of Chernysh's Theorem, which we will make use of later. As mentioned in the introduction, a highly detailed account of this proof (based on Chernysh's original work) is given by Ebert and Frenck in \cite{EF}.

The main idea is as follows. Instead of working with $\Riem^{+}(X)$ and $\Riem^{+}(X')$, we consider certain subspaces $\Riem_{\std}^{+}(X)\subset \Riem^{+}(X)$ and $\Riem_{\std}^{+}(X')\subset \Riem^{+}(X')$. Roughly speaking, these are spaces in which metrics take a standard product form, $ds_{p}^{2}+g_{\tor}^{q+1}$, near the embedded spheres $S^{p}:=\phi(S^{p}\times D^{q+1}(0))\subset N\subset X$ and $S^{q}:=\phi'(D^{p+1}(0)\times S^{q})\subset N'\subset X'$. From this standard structure it will be immediately clear that these spaces are in fact homeomorphic. Thus, it will be enough to show that $\Riem_{\std}^{+}(X)$ and $\Riem^{+}(X)$ are themselves homotopy equivalent spaces. An identical argument will take care of the analogous spaces for $X'$, proving the theorem.
From work of Palais in \cite{Palais}, we know that the spaces $\Riem_{\std}^{+}(X)$ and $\Riem^{+}(X)$ are dominated by CW-complexes. Thus, by a famous theorem of Whitehead (Theorem 4.5 of  \cite{Hatcher}), it is in fact enough to show that the inclusion $\Riem_{\std}^{+}(X)\subset \Riem^{+}(X)$ is a weak homotopy equivalence. Proving this result is effectively what the theorem is about.
\begin{remark}\label{EFboundary}
In \cite{Che}, Chernysh considers the inclusion $\Riem_{\std}^{+}(X)\subset \Riem^{+}(X)$ in a more general setting than we do here. The embedding $S^{p}\times D^{q+1}\hookrightarrow X$, is replaced by an embedding $Y^{p}\times D^{q+1}\hookrightarrow X^{n}$, for some arbitrary compact manifold $Y$. Obviously, $p,q$ are assumed to satisfy $p+q+1=n$, with $p,q\geq 2$. Letting $g_Y$ be some arbitrary Riemannian metric on $Y$, the space of standard psc-metrics on $Y$ to is defined to consist of psc-metrics on $X$ which take the form of a product $g_Y+ g_{\tor}^{q+1}$ near $Y$. The proof in the more general case is the same. In \cite{EF}, Ebert and Frenck go even further, assuming that $X$ and $Y$ may be manifolds with boundary and with $Y$ satisfying a compatibility condition with respect to the boundary $\p X$. We will discuss this non-empty boundary case in more detail at the end of the section. 
\end{remark}

In order to define the subspaces of standard psc-metrics $\Riem_{\std}^{+}(X)$ and $\Riem_{\std}^{+}(X')$, there are a couple of slightly larger subspaces we need to define. 
Recall from the previous section: the space $\Riem_{O(q+1)}^{+}(D^{q+1})$ of rotationally symmetric psc-metrics on the disk and the subspaces $\Riem_{\cT}^{+}(D^{q+1})$ and $\Riem_{\cA\cT}^{+}(D^{q+1})$ of torpedo and almost torpedo metrics. These spaces include as:
$$\Riem_{\cT}^{+}(D^{q+1})\subset \Riem_{\cA\cT}^{+}(D^{q+1})\subset \Riem_{O(q+1)}^{+}(D^{q+1}).$$
We will make considerable use of these spaces throughout the remainder of this section.
We begin by defining the space $\Riem_{\cT}^{+}(X)$ as:
$$\Riem_{\cT}^{+}(X):=\{g\in \Riem^{+}(X): \phi_{\frac{1}{2}}^{*}g=ds_{p}^{2}+h_{q+1} \text{ on } S^{p}\times D^{q+1} \text{ where }h_{q+1}\in\Riem_{\cT}^{+}(D^{q+1})\}.$$
%and
%$$\Riem_{\cT}^{+}(X'):=\{g'\in \Riem^{+}(X'): (\phi_{r_{\std}}')^{*}g'=h_{p+1}+ds_{q}^{2} \text{ on } D^{p+1}\times S^{q}\text{ where }h_{p+1}\in\Riem_{\cT}^{+}(D^{p+1})\}\}.$$
Thus, $\Riem_{\cT}^{+}(X)$ is the space of psc-metrics on $X$ which restrict on $N_{\frac{1}{2}}$ as a product of a unit round sphere with an arbitrary torpedo metric.
Next, we define $\Riem_{\cT(1)}^{+}(X)$ to be the subspace of $\Riem_{\cT}^{+}(X)$ where the torpedo metric on the disk factor has fixed radius $1$. Thus, we have:
$$\Riem_{\cT(1)}^{+}(X):=\{g\in \Riem^{+}(X): \phi_{\frac{1}{2}}^{*}g=ds_{p}^{2}+h_{q+1} \text{ on } S^{p}\times D^{q+1}\text{ where } h_{q+1}\in \Riem_{\cT(1)}(D^{q+1})  \}.$$
%and
%$$\Riem_{\cT(1)}^{+}(X'):=\{g'\in \Riem^{+}(X'): (\phi_{r_{\std}}')^{*}g'=h_{p+1}+ds_{q}^{2} \text{ on } D^{p+1}\times S^{q} \text{ where } h_{p+1}\in \Riem_{\cT(1)}(D^{p+1})\}.$$
The spaces $\Riem_{\cT}^{+}(X')$ and $\Riem_{\cT(1)}^{+}(X')$ are analogously defined.
Note that we place no conditions on the neck-length of any of these torpedo metrics. The following proposition is hardly surprising and so we provide only the idea behind the proof. 
\begin{proposition}\label{torpone}
The inclusion: $$\Riem_{\cT(1)}^{+}(X)\subset \Riem_{\cT}^{+}(X)$$ is a weak homotopy equivalence.
\end{proposition}
\noindent The basic idea, in the spirit of Lemma \ref{isotopyimpliesconc}, is to take an arbitrary compact family of psc-metrics and gradually stretch out the torpedo parts slowly adjusting the radius to the desired size. Provided this is done slowly enough, positive scalar curvature can be maintained. Finally, torpedos which already have radius $1$ will simply have their necks stretched leaving the radius unaltered.
 
Thus, $\Riem_{\cT(1)}^{+}(X)$ is a space of metrics on $X$ which pull back to a certain (small) subspace of product metrics in $\Riem^{+}(S^{p}\times D^{q+1})$. However, by  the method used in Lemma \ref{torpcontract} (continuously retracting the space of radius $1$ torpedo metrics down to a single preferred radius $1$ torpedo metric), it is easy to see that this collection of torpedo metrics is a contractible subspace in $\Riem^{+}(S^{p}\times D^{q+1})$. Moreover, one can continuously retract this space down to a lone torpedo metric of radius $1$, without damaging the $C^{\infty}$-data along the boundary. Thus, up to homotopy, $\Riem_{\cT(1)}^{+}(X)$ (and in the analogous case $\Riem_{\cT(1)}^{+}(X')$) can be regarded as a space of psc-metrics on $X$ all of which which pull back via $\phi_{\frac{1}{2}}$ to a specific standard product metric. With this in mind, we fix radius $1$ torpedo metrics $g_{\tor}^{q+1}(1)_{\lambda}\in \Riem_{\cT(1)}(D^{q+1})$ and $g_{\tor}^{p+1}(1)_{\lambda}\in \Riem_{\cT(1)}(D^{p+1})$ of various neck-lengths $\lambda>0$, then define the spaces:
$$\Riem_{\std(\lambda)}^{+}(X):=\{g\in \Riem^{+}(X): \phi_{\frac{1}{2}}^{*}g=ds_{p}^{2}+g_{\tor}^{q+1}(1)_{\lambda} \text{ on } S^{p}\times D^{q+1}\}$$
and
$$\Riem_{\std(\lambda)}^{+}(X'):=\{g'\in \Riem^{+}(X'): (\phi_{\frac{1}{2}}')^{*}g'=g_{\tor}^{p+1}(1)_{\lambda}+ds_{q}^{2} \text{ on } D^{p+1}\times S^{q} \}.$$
For the most part, the neck-length of the torpedo factor on the standard part of the metric will not be important. We will pay slight attention to it in later sections however, where it will be useful to allow it to vary. With that in mind we define the following spaces which we will reintroduce later on:
$$\Riem_{\std(-)}^{+}(X):=\bigcup_{\lambda>0}\Riem_{\std(\lambda)}^{+}(X) \text{ and } \Riem_{\std(-)}^{+}(X):=\bigcup_{\lambda>0}\Riem_{\std(\lambda)}^{+}(X).$$
The following proposition is obvious.
\begin{proposition}
For any $\lambda>0$, the inclusion: $$\Riem_{\std(\lambda)}^{+}(X)\subset \Riem_{\std(-)}^{+}(X)$$ is a homotopy equivalence.
\end{proposition}

For now we will consider the case when $\lambda=1$. Thus, we define:
$$ \Riem_{\std}^{+}(X):=\Riem_{\std(1)}^{+}(X)\text{ and } \Riem_{\std}^{+}(X'):=\Riem_{\std(1)}^{+}(X').$$
Based on the discussion above and Proposition \ref{torpone} we have the following.
\begin{proposition}\label{stdinclude}
The inclusions: $$\Riem_{\std}^{+}(X)\subset \Riem_{\cT(1)}^{+}(X)\subset \Riem_{\cT}^{+}(X)$$ are weak homotopy equivalences.
\end{proposition}
\noindent The following proposition is now immediate.
\begin{proposition}\label{easyhomeoclosed}
The map defined by removing the standard piece $(N_{\frac{1}{2}}\cong S^{p}\times D^{q+1}, ds_{p}^{2}+g_{\tor}^{q+1})$ from a metric $g\in \Riem_{\std}^{+}(X)$ and replacing it with $(N_{\frac{1}{2}}'\cong D^{p+1}\times S^{q}, g_{\tor}^{p+1}+ds_{q}^{2})$ to obtain a psc-metric $g'\in\Riem^{+}(X')$ (as suggested in  Fig. \ref{GLsurgery}) is a homeomorphism: $$ \Riem_{\std}^{+}(X)\cong \Riem_{\std}^{+}(X').$$
\end{proposition}

As discussed earlier, to prove Chernysh's theorem it remains to show that the inclusion $\Riem_{\std}^{+}(X)\subset \Riem^{+}(X)$ is a weak homotopy equivalence. This is really what the theorem is about and takes the form of Theorem 3.1 in Ebert and Frenck's account in \cite{EF}. As usual, the idea is to exhibit a suitable homotopy of an arbitrary compact family $K\rightarrow \Riem^{+}(X)$. In particular, it must be shown that elements in the image of this map which lie in the subspace $\Riem_{\std}^{+}(X)$ remain there throughout the homotopy. Theorem \ref{GLcompact} certainly allows a homotopy of any such map to one whose image lies in  $\Riem_{\std}^{+}(X)$. Unfortunately, along the way, psc-metrics which are already standard may be temporarily moved out of $\Riem_{\std}^{+}(X)$. 

As the damage to these metrics is not too severe (the Gromov-Lawson construction displaying a great deal of symmetry) this problem is solved by replacing $\Riem_{\std}^{+}(X)$ with a certain larger space of ``almost standard" psc-metrics. This space is denoted $\Riem_{\Astd}^{+}(X)$ and satisfies:
$$ \Riem_{\std}^{+}(X)\subset \Riem_{\Astd}^{+}(X)\subset \Riem^{+}(X).$$ The theorem is proved by showing that each of the above inclusions is a weak homotopy equivalence. The idea is that $\Riem_{\Astd}^{+}(X)$ should capture all adjustments made to a standard psc-metric by the Gromov-Lawson construction. More precisely, suppose $g\in\Riem^{+}(X)$. The Gromov-Lawson construction describes an isotopy from the psc-metric $g$ to a psc-metric $g_{\std}\in\Riem_{\std}^{+}(X)$. Technically, there are many isotopies when one considers the various choices which can be made in performing the construction. In any case, let $\GL:[0,1]\rightarrow\Riem^{+}(X)$ denote such an isotopy. Thus, $\GL(0)=g$ and $\GL(1)=g_{\std}$. Now suppose the starting psc-metric $g$ is already an element of $\Riem_{\std}^{+}(X)$. The space $\Riem_{\Astd}^{+}(X)$ must be defined so that for any such isotopy, $\GL$, and any $\tau\in [0,1]$, the psc-metric $\GL(\tau)$ is an element of 
$\Riem_{\Astd}^{+}(X)$. 

With this in mind, we define the intermediary space of {\em almost standard metrics} on $X$, $\Riem_{\Astd}^{+}(X)$, as:
$$\Riem_{\Astd}^{+}(X):=\{g\in \Riem^{+}(X):\phi_{\frac{1}{2}}^{*}g=ds_{p}^{2}+h_{q+1} \text{ where } h_{q+1}\in \Riem_{\cA\cT}^{+}(D^{q+1})\}. $$ 
\begin{remark}\label{EFcomment}
 Chernysh in \cite{Che} (and in turn Ebert and Frenck in \cite{EF}) work with a slightly larger intermediary space denoted $\Riem_{\mathrm{rot}}^{+}(X)$ of metrics which restrict on the standard region to a product 
 $ds_{p}^{2}+h_{q+1}$ where $h_{q+1}$ is an $O(q+1)$-rotationally symmetric psc-metric on the disk $D^{n}$. Technically $\Riem_{\Astd}^{+}(X)\subset \Riem_{\mathrm{rot}}^{+}(X)$. This difference is of no consequence when it comes to completing the proof.
\end{remark}
\noindent To complete the proof of Chernysh's theorem it remains to show that each of the inclusions:
$$\Riem_{\std}^{+}(X)\subset\Riem_{\Astd}^{+}(X)\subset \Riem^{+}(X),$$
is a weak homotopy equivalence. The second inclusion is less difficult and is dealt with below.
 
\begin{lemma}\label{AstdX}
The inclusion: $$\Riem_{\Astd}^{+}(X)\subset \Riem^{+}(X)$$ is a weak homotopy equivalence.
%For any $k$, $\pi_{k}(\Riem^{+}(X),  \Riem_{\Astd}^{+}(X))=0.$
\end{lemma}
\begin{proof}
Application of steps (1), (2), (3) and then (5) of Gromov-Lawson construction (Theorem \ref{GLthm}) and via Theorem \ref{GLcompact} for compact families of metrics, continuously moves any compact family of metrics in $\Riem^{+}(X)$ to one which lies entirely in $\Riem_{\Astd}^{+}(X)$. (We ignore step (4) as there is no need to adjust the final torpedo radius to make it equal to $1$ here.) It remains to show that metrics which lie in $\Riem_{\Astd}^{+}(X)$ remain in that space throughout the deformation. 

Reviewing steps (1),(2) and (3) and (5) it follows that a metric $g$ in $\Riem_{\Astd}^{+}(X)$ is only affected by steps (2) and (5). Furthermore, step (5) only has the effect of restricting the warping function of an almost torpedo metric, an almost torepdo function of the form $\omega:[0,b]\rightarrow [0,\infty)$, to various subintervals of the form $[0,b']$ where $b'\in(0,b]$. Such a restriction is still an almost torpedo function. Thus, we need only worry about step (2). 

Let us now examine step (2) a little more closely, assuming again that $g\in \Riem_{\Astd}^{+}(X)$.
The adjustment made by step (2) affects only the metric on the disk factor, $D^{q+1}$, and takes place in the radial direction. 
 Consider $\gamma:[0,b]\rightarrow [0,\infty)\times [0,\infty)$, a Gromov-Lawson curve of the type described above with $\gamma(l)=(\gamma_{t}(l), \gamma_r(l))$. %We can always choose $\gamma$ so that, for any $g$ in the compact family, only $g|_{N_{\frac{1}{2}}}$ is affected by this construction. 
We will compute the hypersurface metric obtained by applying $\gamma$ to $g$, inside the ambient manifold $([0, \infty)\times N_{{\frac{1}{2}}}, dt^{2}+g|_{N_{{\frac{1}{2}}}})$. The ambient metric takes the form $dt^{2}+dr^{2}+ds_{p}^{2}+\omega(r)^{2}ds_{q}^{2}$, for some almost torpedo function $\omega$, and so the hypersurface metric (which we denote $g_{\gamma}$) is computed in these coordinates as:
\begin{equation*}
\begin{split}
g_{\gamma}=&d(\gamma_{t}(l))^{2}+d(\gamma_{r}(l))^{2}+ds_{p}^{2}+\omega(\gamma_{r}(l))^{2}ds_{q}^{2}\\
=&dl^{2}+ds_{p}^{2}+\omega\circ\gamma_{r}(l)ds_{q}^{2}, 
\end{split}
\end{equation*}
since $\gamma$ has unit speed. A straightforward calculation, making use of the properties of the smooth function $\gamma_{r}$ which are outlined in step (6) of the above description of the Gromov-Lawson construction, shows that the composition $\omega\circ\gamma_{r}$ is also an almost torpedo function. This completes the proof.
\end{proof}

The most difficult part of the whole process concerns showing that the first of the two inclusions above: $\Riem_{\std}^{+}(X)\subset\Riem_{\Astd}^{+}(X)$, is a weak homotopy equivalence. We state the result here in the form of Lemma \ref{stdAstd} below, which when combined with Lemma \ref{AstdX} above completes the proof of Chernysh's theorem. The basic idea is to operate at the level of warping functions and provide a homotopy from what is essentially a ``wild" compact family of almost torpedo functions to a family of torpedo functions. Crucially, already torpedo functions must remain so throughout the process. The proof involves some very delicate analysis, as families of warping functions are ``bent" and ``stretched" into a more manageable form all the while satisfying the curvature constraint. As it is quite long and technical, we refer the reader to the thorough account in \cite{EF}. 
\begin{lemma}\label{stdAstd}
The inclusion: $$\Riem_{\std}^{+}(X)\subset\Riem_{\Astd}^{+}(X)$$ is a weak homotopy equivalence.
\end{lemma}

\subsection{Chernysh's Theorem for Manifolds with Boundary}
The last topic of this section concerns a more general version of Chernysh's theorem for manifolds with boundary. Indeed, as we mentioned in Remark \ref{EFboundary}, this is the situation Ebert and Frenck consider in \cite{EF}. Recalling the discussion at the beginning of section 2, let us assume the manifold $X$ above is the boundary of a smooth compact $(n+1)$-dimensional manifold $W^{n+1}$; thus $\p W=X$. We prescribe a collar neighbourhood embedding $c:X\times [0,2)\rh W$ (which of course implies $c(X\times \{0\})=\p W$) and consider the corresponding space $\Riem^{+}(W, \p W)$. Recall that any metric $\bar{g}\in \Riem^{+}(W, \p W)$ is a psc-metric on $W$ satisfying $c^{*}\bar{g}=g+dt^{2}$ on $X\times I$ for some arbitrary $g\in \Riem^{+}(X)$.

We next consider the disk $D^{p+1}$. Viewing each element $x\in D^{p+1}$ in terms of its standard polar coordinate description, $x=x(r,\theta)$,  we specify the collar neighbourhood embedding: $c_D:[0,1]\times S^{p}\rh D^{p+1}$
by $c_D(r, \theta)=x(1-\frac{r}{2}, \theta)$. Thus, $c_{D}(\{0\}\times S^{p})=\p D^{p+1}$. Suppose now that we have an embedding $\bar{\phi}:D^{p+1}\times D^{q+1}\rh W$ satisfying the following conditions:
\begin{enumerate}
\item[(i.)] $\bar{\phi}|_{S^{p}\times D^{q+1}}=\phi$, where $\phi:S^{p}\times D^{q+1}\hookrightarrow X=\p W$ is the embedding above satisfying $p+q+1=n$ and $p,q\geq 2$.
\item[(ii.)] $\bar{\phi}(c_D(r,\theta), y)=c(r, \phi(\theta, y))$, for all $(r, \theta, y)\in I\times S^{p}\times D^{q+1}$.
\end{enumerate}
The embedding $\bar{\phi}$ is considered to be {\em compatible} with the collar $c$; see the left hand image in Fig. \ref{Chebd} for a schematic illustration. Recalling that $N:=\phi(S^{p}\times D^{q+1})$, we set $\bar{N}:=\bar{\phi}(D^{p+1}\times D^{q+1})$. Having previously defined the neighbourhood $N_{\rho}:=\phi_{\rho}(S^{p}\times D^{q+1})$, we now define $\bar{N}_{\rho}:=\bar{\phi}_{\rho}(D^{p+1}\times D^{q+1})$ , where $\rho\in(0,1]$ and $\bar{\phi}_{\rho}$ is defined analogously to $\phi_{\rho}$.  Fixing torpedo metrics $g_{\tor}^{p+1}$ on $D^{p+1}$ and $g_{\tor}^{q+1}$ on $D^{q+1}$, we can now define the space $\Riem_{\std}^{+}(W, \p W)$ as:
$$ \Riem_{\std}^{+}(W, \p W):=\{\bar{g}\in \Riem^{+}(W, \p W): \bar\phi_{\frac{1}{2}}^{*}(\bar{g})=g_{\tor}^{p+1}+g_{\tor}^{q+1} \text{ on } D^{p+1}\times D^{q+1}\}.$$
The right hand image in Fig. \ref{Chebd} provides a schematic picture of a typical metric in $\Riem_{\std}^{+}(W, \p W)$.
  
\begin{figure}[!htbp]
\vspace{1cm}
\hspace{0cm}
\begin{picture}(0,0)
\includegraphics{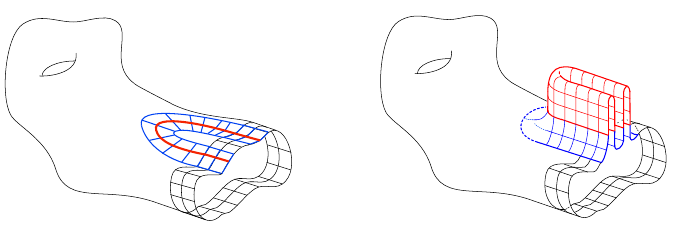}%
\end{picture}
\setlength{\unitlength}{3947sp}%
\begingroup\makeatletter\ifx\SetFigFont\undefined%
\gdef\SetFigFont#1#2#3#4#5{%
  \reset@font\fontsize{#1}{#2pt}%
  \fontfamily{#3}\fontseries{#4}\fontshape{#5}%
  \selectfont}%
\fi\endgroup%
\begin{picture}(5079,1559)(1902,-7227)
\put(2800,-6500){\makebox(0,0)[lb]{\smash{{\SetFigFont{10}{8}{\rmdefault}{\mddefault}{\updefault}{\color[rgb]{0,0,1}${\bar{N}}$}%
}}}}
\put(3800,-7150){\makebox(0,0)[lb]{\smash{{\SetFigFont{10}{8}{\rmdefault}{\mddefault}{\updefault}{\color[rgb]{0,0,0}$c(X\times I)$}
}}}}
\put(6900,-6000){\makebox(0,0)[lb]{\smash{{\SetFigFont{10}{8}{\rmdefault}{\mddefault}{\updefault}{\color[rgb]{1,0,0}$g_{\tor}^{p+1}+g_{\tor}^{q+1}$}%
}}}}
\end{picture}%
\caption{The embedding $\bar{\phi}$  (left) and an arbitrary metric in $\Riem_{\std}^{+}(W, \p W)$ (right) }
\label{Chebd}
\end{figure}   
We further define the analogous spaces:
$$\Riem_{\cT}^{+}(W, \p W):=\{\bar{g}\in \Riem^{+}(W, \p W): \bar{\phi}_{\frac{1}{2}}^{*}\bar{g}=g_{\tor}^{p+1}+h_{q+1} \text{ on } D^{p+1}\times D^{q+1}, h_{q+1}\in\Riem_{\cT}^{+}(D^{q+1})\},$$
$$\Riem_{\cT(1)}^{+}(W, \p W):=\{\bar{g}\in \Riem^{+}(W, \p W): \bar{\phi}_{\frac{1}{2}}^{*}\bar{g}=g_{\tor}^{p+1}+h_{q+1} \text{ on } D^{p+1}\times D^{q+1}, h_{q+1}\in \Riem_{\cT(1)}(D^{q+1})\},$$
$$\Riem_{\Astd}^{+}(W, \p W):=\{\bar{g}\in \Riem^{+}(W, \p W):\bar\phi_{\frac{1}{2}}^{*}\bar{g}=g_{\tor}^{p+1}+h_{q+1} \text{ on } D^{p+1}\times D^{q+1}, h_{q+1}\in \Riem_{\cA\cT}^{+}(D^{q+1})\}.$$
Applying the methods used in the proof of Theorem \ref{Chernysh} above, it is a straightforward exercise to show that the inclusions:
$$
\Riem_{\std}^{+}(W, \p W)\subset\Riem_{\cT(1)}^{+}(W, \p W)\subset   \Riem_{\cT}^{+}(W, \p W)\subset \Riem_{\Astd}^{+}(W, \p W)\subset \Riem^{+}(W, \p W),
$$
are all weak homotopy equivalences. Thus, we have the following version of Chernysh's theorem for manifolds with boundary. This boundary version of Chernysh's theorem, in a somewhat more general setting, is what is proved by Ebert and Frenck in \cite{EF}. 

\begin{theorem} \cite{EF} The inclusion:
$$
\Riem_{\std}^{+}(W, \p W)\subset \Riem^{+}(W, \p W),
$$
is a weak homotopy equivalence.
\end{theorem}

We close this discussion by considering, for a fixed $h\in \Riem^{+}(X)$, the subspace $\Riem^{+}(W, \p W)_h$ of  $\Riem^{+}(W, \p W)$  consisting of psc-metrics on $W$ which restrict on the boundary $\p W$ as $h$. We reiterate that $\Riem^{+}(W, \p W)_h$ may be empty. Suppose now that the fixed psc-metric $h$ is an element of $\Riem_{\std}^{+}(X)$, the space of standardised metrics on $X$ defined above with respect to the embedding $\phi$ (the restriction of $\bar{\phi}$ to $S^{p}\times D^{q+1}$) above. For such a metric $h$, we can sensibly define the analogous space of standard metrics on $W$, $\Riem_{\std}^{+}(W, \p W)_{h}$, as:
$$ \Riem_{\std}^{+}(W, \p W)_h:= \Riem^{+}(W, \p W)_h\cap \Riem_{\std}^{+}(W, \p W)= \{\bar{g}\in \Riem^{+}(W, \p W)_h: \bar\phi_{\frac{1}{2}}^{*}(\bar{g})=g_{\tor}^{p+1}+g_{\tor}^{q+1}\}.$$ The following lemma will play an important role in later sections.
\begin{lemma}\label{stdincludeg}
For any $h\in \Riem_{\std}^{+}(X)$, the inclusion: $$\Riem_{\std}^{+}(W, \p W)_h\subset \Riem^{+}(W, \p W)_h$$ is a weak homotopy equivalence.
\end{lemma}

Before proving this lemma there is a space we need to introduce. With $h\in \Riem_{\std}^{+}(X)$ fixed as before, we define the space
$\Riem^{+}(X)_{\Astd(h)}$ as:
$$\Riem^{+}(X)_{\Astd(h)}:=\{g\in \Riem_{\Astd}^{+}(X): g|_{X\setminus N_{\frac{1}{2}}}=h|_{X\setminus N_{\frac{1}{2}}}\}.$$
Thus, we take our fixed $h\in\Riem_{\std}^{+}(X)$ and loosen the condition on $N_{\frac{1}{2}}$ so the metric need only be almost standard there. From here we define $\Riem^{+}(W, \p W)_{\Astd(h)}$ as:
$$\Riem^{+}(W, \p W)_{\Astd(h)}:=\{\bar{g}\in \Riem^{+}(W, \p W):\bar{g}|_{\p W}\in \Riem^{+}(X)_{\Astd(h)} \}. $$
Thus, elements of $\Riem^{+}(W, \p W)_{\Astd(h)}$ are metrics which on the boundary take the form of a psc-metric on $X$ which agrees with $h$ outside the neighbourhood $N_{\frac{1}{2}}$, but are only almost standard on $N_{\frac{1}{2}}$. 
\begin{proposition}\label{stdcontract} The inclusion: $$\{h\}\subset \Riem^{+}(X)_{\Astd(h)}$$ is a weak homotopy equivalence.
\end{proposition}
\begin{proof} Making no adjustment to metrics outside of $N_{\frac{1}{2}}$, this follows from Lemma \ref{stdAstd} followed by application of Lemma \ref{torpcontract}.
\end{proof}

\begin{proof}[Proof of Lemma \ref{stdincludeg}]
Consider the commutative diagram of inclusion maps below:
\begin{equation*}
\xymatrix{
&\Riem^{+}(W, \p W)_h \ar@{^{(}->}[d]\\
\Riem_{\std}^{+}(W, \p W)_h\ar@{^{(}->}[r] \ar@{^{(}->}[ru]& \Riem^{+}(W, \p W)_{\Astd(h)}. }
\end{equation*}
To prove the diagonal inclusion is a weak homotopy equivalence (thus proving the lemma), we need only show that the vertical and horizontal inclusions are weak homotopy equivalences.
Beginning with the horizontal inclusion, we see that this map is the composition of inclusions:
$$ \Riem_{\std}^{+}(W, \p W)_h \hookrightarrow \Riem_{\Astd}^{+}(W, \p W) \cap\Riem^{+}(W, \p W)_{\Astd(h)}\hookrightarrow \Riem^{+}(W, \p W)_{\Astd(h)}.$$
That the maps in this composition are weak homotopy equivalences follows by application of Lemma \ref{stdAstd} in the case of the first map and Lemma \ref{AstdX} in the case of the second. We note here that any adjustments to metrics in these spaces, by application of these lemmas, take place only inside $\bar{N}_{\frac{1}{2}}$.  Moreover, the loosening of the boundary requirement for metrics in $\Riem^{+}(W, \p W)_{\Astd(h)}$ means that compact families of metrics in this space remain there under application of the Gromov-Lawson process.

We now turn our attention to the vertical inclusion. Let $K\rightarrow \Riem^{+}(W, \p W)_{\Astd(h)}, k\mapsto \bar{g}_k$, be a compact family of metrics. For each $k\in K$, consider now the metric $\bar{g}_{k}|_{c(X\times I)}$, the restriction of $\bar{g}_{k}$ to the region $c(X\times I)$ on $W$. By hypothesis, this metric pulls back via the collar embedding $c$ to a product: $g_{k}+dt^{2}$ on $X\times I$, where $g_{k}:=\res(\bar{g}_k)\in \Riem^{+}(X)_{\Astd(h)}$, the restriction of $\bar{g}_k$ to the boundary.

By Proposition \ref{stdcontract}, there is a homotopy $F: K\times I\rightarrow \Riem^{+}(X)_{\Astd(h)}$ satisfying:
\begin{enumerate}
\item[(i.)] $F(0,k)=g_k$ for all $k\in K$ 
\item[(ii.)] $F(1,k)=h$ for all $k\in K$,
\item[(iii.)] $F(\tau,k)=h$ for all $k\in K$ satisfying $F(0, k)=h$ and all $\tau\in I$.
\end{enumerate}
By Lemma \ref{welldefconc}, there is a constant $L_{K}\geq 1$ so that the map $K\rightarrow \Riem^{+}(X\times[0,L_{K}+2], X\times\{0\}\sqcup X\times\{L_{K}+2\} )$ defined by:
$$k\longmapsto F({{\nu_{L_{K}}(t)}, k})+dt^{2},$$
defines a continuous family of concordances on $X\times [0,L_{K}+2]$. 
Moreover, by Corollary \ref{conchomot} there is a homotopy between this map and the trivial concordance map: $$ k\mapsto g_{k}+dt^{2}. $$ Let us denote this homotopy, $H$. Importantly, at each stage in this homotopy, the restriction of the concordance to $X\times [0,1]$ takes the form $g_k+dt^{2}$. 

Returning to the initial compact family $K\rightarrow \Riem^{+}(W, \p W)_{\Astd(h)}$, we begin by continuously adjusting these metrics on the collar. Consider the newly rescaled collar, $\bar{c}:X\times [0, L_{K}+2]\rh W$ given by precomposing the restricted collar embedding $c|_{X\times I}$ with the rescaling map:
\begin{equation*}
\begin{split}
X\times [0, L_{K}+2]&\longrightarrow X\times [0,1]\\
t&\longmapsto 1-\frac{t}{L_{K}+2}.
\end{split}
\end{equation*}
We note that this maps also flips the direction of the $t$-coordinate. We now continuously stretch the collar parts of all metrics in the family so that for each $k\in K$, $\bar{c}^{*}{\bar{g}_k}=g_{k}+dt^{2}$ on $X\times [0, L_{K}+2]$. (We retain the name $\bar{g}_k$ for this stretched collar metric.)
Next we continuously adjust the restricted original collar $c|_{X\times I}:X\times I\rh W$ (retaining the original name) to one which satisfies: 
$$ c|_{X\times I}(x,t)=\bar{c}(x, L_{K}+2-t).$$
This is depicted schematically for a lone metric $\bar{g}_k$ in the second image of Fig. \ref{collarstretch}.
An application of the homotopy $H$ above above completes the proof; see third image of Fig. \ref{collarstretch}.  Note that from the previous  paragraph, $H$ fixes each metric in the family as $g_k+dt^{2}$ on the region $\bar{c}(X\times I)$, where the new extended collar meets the rest of $W$ thus facilitating a smooth transition.

\begin{figure}[!htbp]
\vspace{0cm}
\hspace{-3.5cm}
\begin{picture}(0,0)
\includegraphics{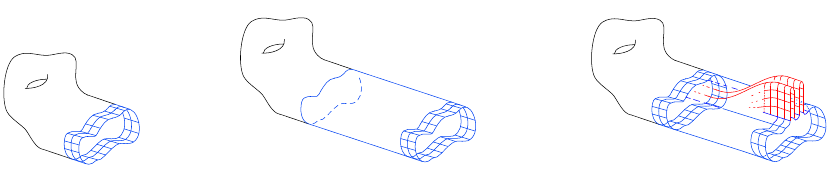}%
\end{picture}
\setlength{\unitlength}{3947sp}%
\begingroup\makeatletter\ifx\SetFigFont\undefined%
\gdef\SetFigFont#1#2#3#4#5{%
  \reset@font\fontsize{#1}{#2pt}%
  \fontfamily{#3}\fontseries{#4}\fontshape{#5}%
  \selectfont}%
\fi\endgroup%
\begin{picture}(5079,1559)(1902,-7227)
\put(2500,-6400){\makebox(0,0)[lb]{\smash{{\SetFigFont{10}{8}{\rmdefault}{\mddefault}{\updefault}{\color[rgb]{0,0,0}$(W, \bar{g}_k)$}%
}}}}
\put(3000,-7050){\makebox(0,0)[lb]{\smash{{\SetFigFont{10}{8}{\rmdefault}{\mddefault}{\updefault}{\color[rgb]{0,0,1}$c(X\times I)$}
}}}}
\put(5000,-6500){\makebox(0,0)[lb]{\smash{{\SetFigFont{10}{8}{\rmdefault}{\mddefault}{\updefault}{\color[rgb]{0,0,1}$\bar{c}(X\times[0,\lambda_{K}+2])$}%
}}}}

\put(5700,-7000){\makebox(0,0)[lb]{\smash{{\SetFigFont{10}{8}{\rmdefault}{\mddefault}{\updefault}{\color[rgb]{0,0,1}$c(X\times I)$}%
}}}}
\end{picture}%
\caption{An arbitrary metric in the family $\bar{g}_k, k\in K$(left), adjusted by metrically stretching the collar (middle) and the element of  $\Riem_{\std}^{+}(W, \p W)_h$ obtained from this by completing the homotopy $H$ (right) }
\label{collarstretch}
\end{figure}   
\end{proof}

We finish with a slight generalisation of Lemma \ref{stdincludeg} which will be of use later. First, recall that for each $\lambda>0$, $\Riem_{\std(\lambda)}^{+}(X)$ denotes the space of psc-metrics on $X$ which on $N_{\frac{1}{2}}$ have a standard form with torpedo factor of neck-length $\lambda$. Suppose we replace the metric $h\in \Riem_{\std}^{+}(X)$ in Lemma \ref{stdincludeg} with $h(\lambda)\in \Riem_{\std(\lambda)}^{+}(X)$. Importantly, $h(1)=h$ and for all $\lambda$:
$$h|_{X\setminus{N_{\frac{1}{2}}}}=h(\lambda)|_{X\setminus{N_{\frac{1}{2}}}}\text{ while } \phi_{\frac{1}{2}}^{*}h(\lambda)=ds_{p}^{2}+g_{\tor}^{q+1}(1)_{\lambda}.$$
It is clear that the theorem goes through just as well in this case. That is, the inclusion $\Riem_{\std}^{+}(W, \p W)_{h(\lambda)}\subset \Riem^{+}(W, \p W)_{h(\lambda)}$ is a weak homotopy equivalence.
We go only slightly further. Defining:
$$\Riem_{\std}^{+}(W, \p W)_{h(-)}:=\bigcup_{\lambda}\Riem_{\std}^{+}(W, \p W)_{h(\lambda)}\text{ and }\Riem^{+}(W, \p W)_{h(-)}:=\bigcup_{\lambda}\Riem^{+}(W, \p W)_{h(\lambda)},$$ we have the following commutative diagram of inclusion maps.
\begin{equation}
\xymatrix{
& \Riem_{\std}^{+}(W, \p W)_{h(\lambda)}\ar@{^{(}->}[d]^{}\ar@{^{(}->}[r]^{}& \Riem^{+}(W, \p W)_{h(\lambda)}\ar@{^{(}->}[d]^{}\\
&\Riem_{\std}^{+}(W, \p W)_{h(-)} \ar@{^{(}->}[r]^{} & \Riem^{+}(W, \p W)_{h(-)}  }
\end{equation}
Showing weak homotopy equivalence for the vertical maps in this diagram is a straighforward application of Lemma \ref{isotopyimpliesconc} and Corollary \ref{conchomot}. The top horizontal map is a weak homotopy equivalence by Lemma \ref{stdincludeg}. Thus, we obtain the following lemma.
\begin{lemma}\label{stdincludeggen}
The inclusion: $$\Riem_{\std}^{+}(W, \p W)_{h(-)}\subset \Riem^{+}(W, \p W)_{h(-)}$$ is a weak homotopy equivalence.
\end{lemma}

\section{Variations on the Torpedo Metric}\label{vartorp}
Before proceeding any further, there are some variations on the torpedo metric which we must discuss. Much of this is based on section $2$ of \cite{Walsh2} where we perform various types of cutting, pasting, stretching and bending of the original torpedo metric. In particular, we will define the so-called ``boot metric" which will be important when proving our main result, Theorem A. Throughout this chapter we assume that $n\geq 3$.

\subsection{Toes}
We begin by denoting as $S_{-}^{n}$ and $S_{+}^{n}$ the respective hemispheres  $\{x\in\mathbb{R}^{n+1}:|x|=1, x_{n+1}\leq 0\}$ and $\{x\in\mathbb{R}^{n+1}:|x|=1, x_{n+1}\geq 0\}$. %These hemispheres will be identified with the disk $D^{n}$ via the obvious map which sends geodesic rays emanating from the points $(0,0,\cdots, 0,\pm 1)$ to the corresponding ray on the flat disk $D^{n}(\frac{\pi}{2})$ followed by the above rescaling map. 
We then denote by $D_{-}^{n}$ and $D_{+}^{n}$, the hemi-disks $\{x\in{D}^{n}: x_{n}\leq 0\}$ and $\{x\in{D}^{n}: x_{n}\geq 0\}$. These hemi-disks are, of course, smooth manifolds with corners, where:
$$
\partial(D_{\pm}^{n})=S_{\pm}^{n-1}\cup D^{n-1}\quad {\mathrm{and}} \quad
\partial^{2}(D_{\pm}^{n}) = \partial S_{\pm}^{n-1} = \partial D^{n-1} = S_{\pm}^{n-1}\cap D^{n-1}=S^{n-2}.
$$

The restriction of a $\delta-\lambda$-torpedo metric on $D^{n}$ (recall this is a torpedo metric with radius $\delta>0$ and neck-length $\lambda\geq 0$),
$g_{\tor}^{n}(\delta)_{\lambda}$, to each of the  hemi-disks gives metrics:
$$
g_{\tor\pm}^{n}(\delta)_{\lambda}:=g_{\tor}^{n}(\delta)_{\lambda}|_{D_{\pm}^{n}}.
$$
These metrics restrict on the boundary components as:
$$
g_{\tor\pm}^{n}(\delta)_{\lambda}|_{S_{\pm}^{n-1}}=\delta^{2}ds_{n-1}^{2}|_{S_{\pm}^{n-1}}  \quad {\mathrm{and}} \quad g_{\tor\pm}^{n}(\delta)_{\lambda}|_{{D}^{n-1}} = g_{\tor}^{n-1}(\delta)_{\lambda}.
$$ 
Typically, we will require that $\lambda=1$ and also often $\delta=1$. We adopt the following abbreviations: $$ g_{\tor\pm}^{n}(\delta):=g_{\tor\pm}^{n}(\delta)_{1}\quad \text{ and }\quad g_{\tor\pm}^{n}:=g_{\tor\pm}^{n}(1).$$
We turn our attention now to the cylinder $D^{n}\times I$. We may glue this to the hemi-disk $D_{+}^{n+1}$ (or $D_{-}^{n+1}$), in the obvious way, to obtain (after suitable smoothing) the manifold with corners $D_{\stret}^{n+1}$ depicted below in Fig. \ref{Torpcorners}. Roughly speaking:
$$
D_{\stret}^{n+1}\approx D_{+}^{n+1}\cup (D^{n}\times I).
$$
We would like to equip it with a metric $\hat{g}_{\tor}^{n+1}(\delta)_{\lambda_1, \lambda_2}$ (depicted in Fig. \ref{Torpcorners}) which essentially satisfies: 
$$
\hat{g}_{\tor}^{n+1}(\delta)_{\lambda_1, \lambda_2}|_{D_{+}^{n+1}}=g_{\tor+}^{n+1}(\delta)_{\lambda_1}\quad \mathrm{and} \quad \hat{g}_{\tor}^{n+1}(\delta)_{\lambda_1, \lambda_2}|_{D^{n}\times I} \text{ is isometric to }g_{\tor}^{n}(\delta)_{\lambda_1}+dt^{2},
$$
where $t\in[0,\lambda_2]$. 
As it stands, simply gluing these metric pieces together will not result in a smooth metric. However, it is possible to do this smoothly producing a metric which, away from the attachment submanifold is precisely the metric described above. This is done in great detail in Lemma 2.1 of \cite{Walsh2} and so we will not discuss it any further here. Henceforth, $\hat{g}_{\tor}^{n+1}(\delta)_{\lambda_1, \lambda_2}$ denotes the metric obtained by the above attachment after suitable smoothing; see the rightmost image in Fig. \ref{Torpcorners}. This metric is known as a {\em $\delta$-toe metric} for reasons which will become clear shortly. The parameter $\lambda_1$ will frequently be referred to as the ``straight height" of the toe while $\lambda_2$ will be regarded as the ``base-length".

The following proposition follows easily from the curvature calculations done in section 2 of \cite{Walsh2}.
\begin{proposition}
Let $n\geq 3$. For any $\delta>0$ and $\lambda, \lambda_1, \lambda_2\geq 0$, each of the metrics $g_{\tor\pm}^{n}(\delta)_{\lambda}$ and $\hat{g}_{\tor}^{n+1}(\delta)_{\lambda_1, \lambda_2}$ has positive scalar curvature. Moreover, the scalar curvature of any such metric can be bounded below by an arbitrarily large positive constant by choosing $\delta>0$ sufficiently small.
\end{proposition}

As we have used a good deal of notation here, let us try to clean things up a little. Regarding the metric $\hat{g}_{\tor}^{n+1}(\delta)_{\lambda_1, \lambda_2}$, we will usually require that the straight height $\lambda_1=1$. This is to ensure compatibility during certain gluings. Specifiying the base-length parameter $\lambda_2$ is less important and it could be left free to vary. However, to simplify things, we set it at $1$ also. We write simply $\hat{g}_{\tor}^{n+1}(\delta)$ to denote the $\delta$-toe metric defined:
$$ \hat{g}_{\tor}^{n+1}(\delta):=\hat{g}_{\tor}^{n+1}(\delta)_{1, 1}.$$
Finally, in the case when $\delta=1$ we write:
$$ \hat{g}_{\tor}^{n+1}:=\hat{g}_{\tor}^{n+1}(1).$$ 

\begin{figure}[!htbp]
\vspace{3cm}
\hspace{2cm}
\begin{picture}(0,0)
\includegraphics{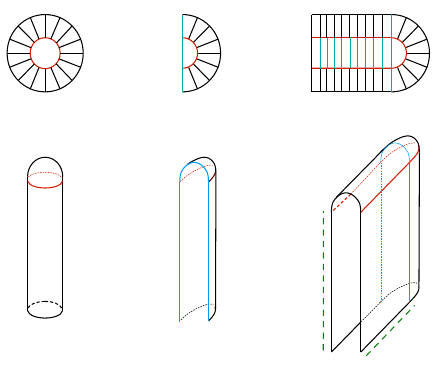}%
\end{picture}
\setlength{\unitlength}{3947sp}%\includegraphics{RelGLpic/
\begingroup\makeatletter\ifx\SetFigFont\undefined%
\gdef\SetFigFont#1#2#3#4#5{%
  \reset@font\fontsize{#1}{#2pt}%
  \fontfamily{#3}\fontseries{#4}\fontshape{#5}%
  \selectfont}%
\fi\endgroup%
\begin{picture}(5079,1559)(1902,-7227)
%\put(2760,-7000){\makebox(0,0)[lb]{\smash{{\SetFigFont{10}{8}{\rmdefault}{\mddefault}{\updefault}{\color[rgb]{0,0.6,0}$w$}%
%}}}}
\put(4200,-6500){\makebox(0,0)[lb]{\smash{{\SetFigFont{10}{8}{\rmdefault}{\mddefault}{\updefault}{\color[rgb]{0,0.6,0}$\lambda_1$}%
}}}}
\put(5000,-7000){\makebox(0,0)[lb]{\smash{{\SetFigFont{10}{8}{\rmdefault}{\mddefault}{\updefault}{\color[rgb]{0,0.6,0}$\lambda_2$}%
}}}}
\end{picture}%
\caption{The metrics $g_{\tor}^{n+1}(\delta)_{\lambda}$, $g_{\tor+}^{n+1}(\delta)_{\lambda}$ and $\hat{g}_{\tor}^{n+1}(\delta)_{\lambda_1, \lambda_2}$ (bottom) on the manifolds $D^{n+1}$, $D_{+}^{n+1}$ and $D_{\stret}^{n+1}$ (top) }
\label{Torpcorners}
\end{figure}

In the original Gromov-Lawson construction, a geometric surgery is performed on a Riemannian manifold of positive scalar curvature by replacing a standardised part of the metric, of the form $ds_{p}^{2}+g_{\tor}^{q+1}$, with a new standard piece $g_{\tor}^{p+1}+ds_{q+1}^{2}$. As we will see shortly, in the case of manifolds with boundary, analogous roles will be played by the standard metrics $ds_{p}^{2}+\hat{g}_{\tor}^{q+2}$ and $g_{\tor}^{p+1}+g_{\tor}^{q+1}$; see Fig. \ref{bootprodl}. The common neck-length $1$ on the second factors is necessary for these pieces to appropriately line up.

\begin{figure}[htb!]
\vspace{-1cm}
\hspace{-6cm}
\begin{picture}(0,0)%
\includegraphics{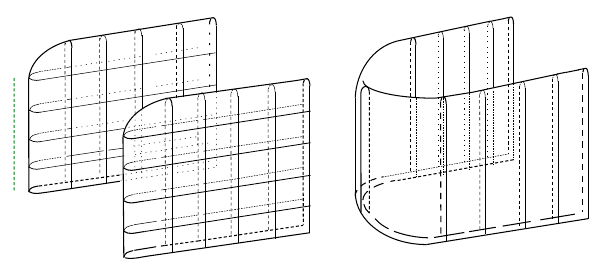}%
\end{picture}%
\setlength{\unitlength}{3947sp}%
\begingroup\makeatletter\ifx\SetFigFont\undefined%
\gdef\SetFigFont#1#2#3#4#5{%
  \reset@font\fontsize{#1}{#2pt}%
  \fontfamily{#3}\fontseries{#4}\fontshape{#5}%
  \selectfont}%
\fi\endgroup%
\begin{picture}(2079,2559)(1902,-5227)
\put(1900,-4200){\makebox(0,0)[lb]{\smash{{\SetFigFont{10}{8}{\rmdefault}{\mddefault}{\updefault}{\color[rgb]{0,0.6,0}${1}$}
}}}}
\end{picture}%
\caption{The metric $ds_{p}^{2}+\hat{g}_{\tor}^{q+2} $ (left) and $g_{\tor}^{p+1}+g_{\tor}^{q+1}$ (right)}
\label{bootprodl}
\end{figure} 

\subsection{Boots}%\label{bendingboot}

We begin with a cylinder metric $(D^{n}\times [0, L+2], g_{\tor}^{n}(\delta)_\lambda+dt^{2})$, for some $ \delta>0, \lambda\geq 0$ and $L\geq 0$. %The role of $\lambda$ is relatively unimportant and so the reader may as well assume $\lambda=1$. 
It is convenient to think of this metric as obtained by way of an embedding into $\mathbb{R}^{n+2}=\mathbb{R}\times\mathbb{R}^{n}\times\mathbb{R}$, equipped with coordinates $(x_0, (x_1,\cdots, x_n), x_{n+1})$, as follows. Recall that the torpedo metric $g_{\tor}^{n}(\delta)_\lambda$ takes the form $dr^{2}+\eta_{\delta, \lambda}(r)^{2}ds_{n-1}^{2}$ on $D^{n}(\delta\frac{\pi}{2}+\lambda)$ for an appropriate warping function $\eta_{\delta, \lambda}:[0, \delta\frac{\pi}{2}+\lambda]\rightarrow [0,\infty)$. As discussed at the beginning of section \ref{disksection}, such a metric is obtained via an embedding: 
\begin{equation*}
\begin{array}{rl}
%\begin{split}
(0, \delta\frac{\pi}{2}+\lambda)\times S^{n-1}&\longrightarrow\mathbb{R}\times\mathbb{R}^{n}\\
(r,\theta)&\longmapsto(\alpha_{\delta, \lambda}(r), \eta_{\delta, \lambda}(r).\theta),
\end{array}
%\end{split}
\end{equation*}
where $\alpha_{\delta, \lambda}:[0, \delta\frac{\pi}{2}+\lambda]\rightarrow [0,\infty)$ is a smooth function which along with $ \eta_{\delta, \lambda}$ satisfy the conditions of $\alpha$ and $\beta$ respectively in \ref{beta0}. In particular, the metric $g_{\tor}^{n}(\delta)_\lambda$ is obtained by pulling back the Euclidean metric on $\mathbb{R}^{n+1}$ via this embedding. With this in mind we now define a cylindrical version of this embedding as:
\begin{equation*}
\begin{array}{rl}
%\begin{split}
\cyl: (0, \delta\frac{\pi}{2}+\lambda)\times S^{n-1}\times[0,L+2]&\longrightarrow\mathbb{R}\times\mathbb{R}^{n}\times\mathbb{R}\\
(r,\theta, t)&\longmapsto(\alpha_{\delta, \lambda}(r), \eta_{\delta, \lambda}(r).\theta, t).
\end{array}
%\end{split}
\end{equation*}
The image of this map is schematically depicted in the left image of Fig. \ref{bendingcylinder}.
We would like to adjust this metric by bending it around an angle $\gamma\in[0,\frac{\pi}{2}]$ in the $(x_0-x_{n+1})$-plane as shown in Fig. \ref{bendingcylinder} below. The bottom of the cylinder (near the origin) is pulled in the positive direction along the $x_0$ and $x_{n+1}$-axes while the top remains stationary. Importantly, we wish to maintain positivity of the scalar curvature of the metric at every stage while making no change to the metric near the the ends.  Regarding the latter problem, we specify collars $D^{n}\times [0, 1]$ and $D^{n}\times [L+1, L+2]$ wherein the metric should remain unaltered. We will make all adjustments on the region $D^{n}\times [1, L+1]$. %From the point of view of maintaining positive scalar curvature, this is clearly the more challenging direction in which to bend the cylinder. A little later, we will need to consider bending such a cylinder in the other direction. As we will discuss, this can be done so as to maintain positive scalar curvature without any great difficulty. In the above case however, 
Maintaining positive scalar curvature is possible provided the bend is done gradually over a wide ``bending arc".

\begin{figure}[htb!]
\vspace{4cm}
\hspace{-8cm}
\begin{picture}(0,0)%
\includegraphics{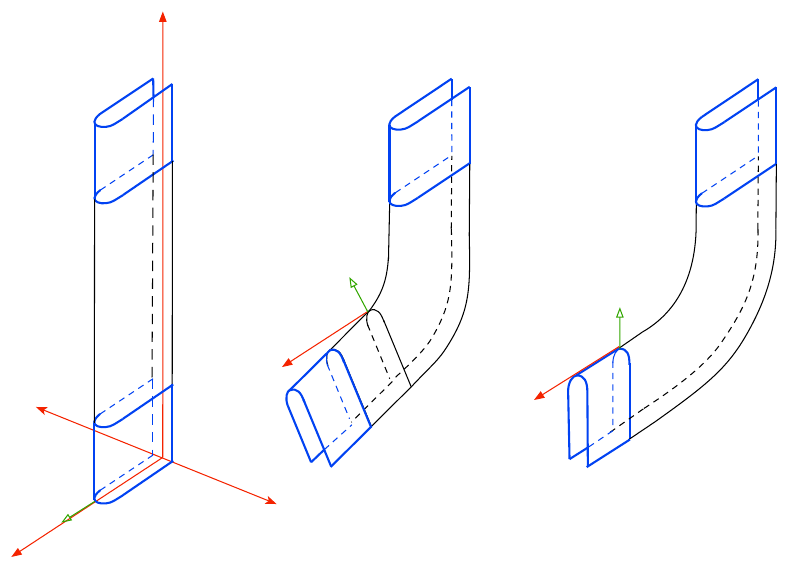}%
\end{picture}%
\setlength{\unitlength}{3947sp}%
\begingroup\makeatletter\ifx\SetFigFont\undefined%
\gdef\SetFigFont#1#2#3#4#5{%
  \reset@font\fontsize{#1}{#2pt}%
  \fontfamily{#3}\fontseries{#4}\fontshape{#5}%
  \selectfont}%
\fi\endgroup%
\begin{picture}(2079,2559)(1902,-5227)
\put(3310,-1700){\makebox(0,0)[lb]{\smash{{\SetFigFont{10}{8}{\rmdefault}{\mddefault}{\updefault}{\color[rgb]{0,0,1}$D^{n}\times [L+1, L+2]$}
}}}}
\put(1450,-3000){\makebox(0,0)[lb]{\smash{{\SetFigFont{10}{8}{\rmdefault}{\mddefault}{\updefault}{\color[rgb]{0,0,0}$D^{n}\times [1, L+1]$}
}}}}
\put(3310,-4200){\makebox(0,0)[lb]{\smash{{\SetFigFont{10}{8}{\rmdefault}{\mddefault}{\updefault}{\color[rgb]{0,0,1}$D^{n}\times [0, 1]$}
}}}}
\put(2150,-4700){\makebox(0,0)[lb]{\smash{{\SetFigFont{10}{8}{\rmdefault}{\mddefault}{\updefault}{\color[rgb]{0,0.6,0}$\gamma=0$}
}}}}
\put(1200,-3800){\makebox(0,0)[lb]{\smash{{\SetFigFont{10}{8}{\rmdefault}{\mddefault}{\updefault}{\color[rgb]{1,0,0}$(x_1, \cdots, x_n)$-plane}
}}}}
\put(1470,-5300){\makebox(0,0)[lb]{\smash{{\SetFigFont{10}{8}{\rmdefault}{\mddefault}{\updefault}{\color[rgb]{1,0,0}$x_0$-axis}
}}}}
\put(2430,-1000){\makebox(0,0)[lb]{\smash{{\SetFigFont{10}{8}{\rmdefault}{\mddefault}{\updefault}{\color[rgb]{1,0,0}$x_{n+1}$-axis}
}}}}
\put(3960,-3140){\makebox(0,0)[lb]{\smash{{\SetFigFont{10}{8}{\rmdefault}{\mddefault}{\updefault}{\color[rgb]{0,0.6,0}$0<\gamma<\frac{\pi}{2}$}
}}}}
\put(6360,-3300){\makebox(0,0)[lb]{\smash{{\SetFigFont{10}{8}{\rmdefault}{\mddefault}{\updefault}{\color[rgb]{0,0.6,0}$\gamma=\frac{\pi}{2}$}
}}}}
%\put(5500,-4100){\makebox(0,0)[lb]{\smash{{\SetFigFont{10}{8}{\rmdefault}{\mddefault}{\updefault}{\color[rgb]{0,0,1}${\hat{g}_{\tor}^{n+1}}(\delta)$}
%}}}}
\end{picture}%
\caption{Bending the cylinder metric $g_{\tor}^{n}(\delta)_\lambda+dt^{2}$}
\label{bendingcylinder}
\end{figure} 

We will begin by considering what happens when we bend a cylinder metric with positive scalar curvature around a circular arc without worrying about maintaining a product form near the boundary. Once we can do this in a way that preserves positive scalar curvature, we will then adjust this construction to obtain the desired boundary product structure. We will first perform the bending on a cylinder $S^{n}\times [0,\gamma]$ for where $\gamma\in(0,\frac{\pi}{2}]$, knowing that this takes care of the case of $D^{n}\times [0,\gamma]$ by restriction. We will also consider the more general case of a psc-metric $g_{\beta}^{n}$ on $S^{n}$ where $\beta:[0,b]\rightarrow [0,\infty)$ satisfies the conditions of \ref{beta0} and \ref{beta}.  
In order to ensure a wide circular arc when we bend, we replace the embedding $\cyl$ by:
\begin{equation*}
\begin{array}{rl}
%\begin{split}
\cyl_{\beta, \gamma, \Lambda}: (0,b)\times S^{n-1}\times[0,\gamma]& \longrightarrow\mathbb{R}\times\mathbb{R}^{n}\times [0,\infty)\\
(r,\theta, t)&\longmapsto( \Lambda+\alpha(r), \beta(r).\theta, t),
\end{array}
%\end{split}
\end{equation*}
where $\gamma\in(0, \frac{\pi}{2}]$ and  $\Lambda>0$ are constants while $\alpha$ and $\beta$ are related as in \ref{beta0}. We assume at least that $\Lambda>|\alpha(b)|$ so that the image of this embedding, depicted in the left image of Fig. \ref{cylinderandbend} below, lies inside $(0, \infty)\times\mathbb{R}^{n}\times[0,\infty)$. 
We now consider the map:
\begin{equation*}
\begin{split}
\bend:(0, \infty)\times\mathbb{R}^{n}\times [0,\infty)&\longrightarrow\mathbb{R}\times\mathbb{R}^{n}\times\mathbb{R}\\
(x_0, (x_{1},\cdots, x_n), x_{n+1})&\longmapsto (x_0\cos(x_{n+1}), (x_1,\cdots, x_n), x_0\sin(x_{n+1})).
\end{split}
\end{equation*}
The following proposition is immediate.
\begin{proposition}
For any $\epsilon>0$, the restriction of the map $\bend$ to $(\epsilon, \infty)\times\mathbb{R}^{n}\times [0,\infty)$ is a smooth embedding.
\end{proposition}
\noindent We denote by $g_{\beta, \gamma, \Lambda}^{n+1}$, the metric defined by pulling back, via the composition map $\bend\circ\cyl_{\beta, \gamma, \Lambda}$, the standard Euclidean metric on $\mathbb{R}\times\mathbb{R}^{n}\times\mathbb{R}$. Thus, setting $g_{\Euc}^{n+2}=\sum_{i=1}^{n+2}dx_{i}^{2}$, we have:
$$g_{\beta, \gamma, \Lambda}^{n+1}:=(\bend\circ\cyl_{\beta, \gamma, \Lambda})^{*}g_{\Euc}^{n+2},$$
inducing a metric on $S^{n-1}\times[0,\gamma]$. The cases when $\gamma\in(0,\frac{\pi}{2})$ and $\gamma=\frac{\pi}{2}$ are respectively depicted as the images of appropriate embeddings in the middle and right of Fig. \ref{cylinderandbend}. 
\begin{figure}[!htbp]
\vspace{0mm}
\hspace{-12cm}
\begin{picture}(0,0)%
\includegraphics{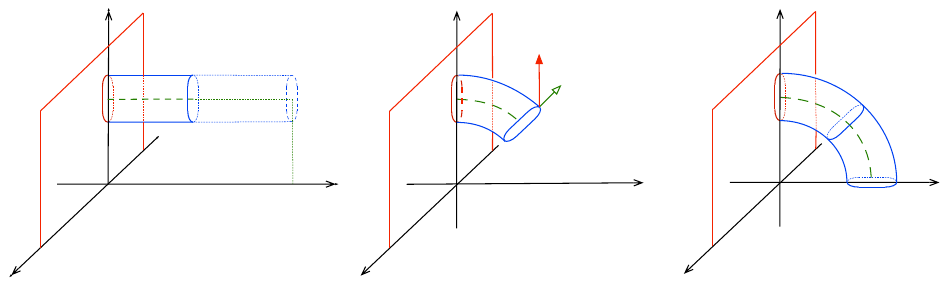}%
\end{picture}%
\setlength{\unitlength}{3947sp}%
\begingroup\makeatletter\ifx\SetFigFont\undefined%
\gdef\SetFigFont#1#2#3#4#5{%
  \reset@font\fontsize{#1}{#2pt}%
  \fontfamily{#3}\fontseries{#4}\fontshape{#5}%
  \selectfont}%
\fi\endgroup%
\begin{picture}(2079,2559)(1902,-5227)
\put(3600,-4100){\makebox(0,0)[lb]{\smash{{\SetFigFont{10}{8}{\rmdefault}{\mddefault}{\updefault}{\color[rgb]{0,0.6,0}$\Lambda+\alpha(b)$}%
}}}}
\put(6280,-3500){\makebox(0,0)[lb]{\smash{{\SetFigFont{10}{8}{\rmdefault}{\mddefault}{\updefault}{\color[rgb]{0,0.6,0}$\gamma$}%
}}}}
%\put(4350,-4500){\makebox(0,0)[lb]{\smash{{\SetFigFont{10}{8}{\rmdefault}{\mddefault}{\updefault}{\color[rgb]{0,0,0}$c\frac{\pi}{2}$}%
%}}}}
\put(3100,-3000){\makebox(0,0)[lb]{\smash{{\SetFigFont{10}{8}{\rmdefault}{\mddefault}{\updefault}{\color[rgb]{0.8,0,0}$\mathbb{R}^{n}\times(0,\infty)\times \{0\}$}%
}}}}
\put(2150,-5200){\makebox(0,0)[lb]{\smash{{\SetFigFont{10}{8}{\rmdefault}{\mddefault}{\updefault}{\color[rgb]{0,0,0}$\mathbb{R}^{n}$}%
}}}}
\put(2200,-3000){\makebox(0,0)[lb]{\smash{{\SetFigFont{10}{8}{\rmdefault}{\mddefault}{\updefault}{\color[rgb]{0,0,0}$x_0$-axis}%
}}}}
\put(4000,-4550){\makebox(0,0)[lb]{\smash{{\SetFigFont{10}{8}{\rmdefault}{\mddefault}{\updefault}{\color[rgb]{0,0,0}$x_{n+1}$-axis}%
}}}}
\end{picture}%
\caption{The image of the map $\cyl_{\beta, \gamma, \Lambda}$ (left) and images of proposed family of ``bent cylinder" embeddings (middle and right)} 
\label{cylinderandbend}
\end{figure} 

\begin{lemma}\label{easybend}
Let $n\geq 3$ and let $\beta:[0,b]\rightarrow [0,\infty)$ be a smooth function satisfying the conditions of \ref{beta0} and \ref{beta}. There exists $\Lambda>0$ so that for all $\gamma\in(0,\frac{\pi}{2}]$, the metric $g_{\beta, \gamma,  \Lambda}^{n+1}$ on $S^{n}\times [0,\gamma]$ has positive scalar curvature. %In the event that condition (iii.) of \ref{beta0} is replaced by the condition that $\beta(b)>0$, there is a constant $\Lambda>0$ so that resulting metric $g_{\beta, \Lambda, \gamma}$ on $D^{n}\times \mathbb{R}$ has positive scalar curvature for all $\lambda\in[0,\frac{\pi}{2}]$.
\end{lemma}
\begin{proof}
We begin by computing $\bend^{*}(g_{\Euc}^{n+2})$: 
\begin{equation*}\label{metriccalc}
\begin{split}
\bend^{*}(g_{\Euc}^{n+2})&=\bend^{*}(\sum_{i=1}^{n+2}dx_{i}^{2})\\
&=\sum_{i=1}^{n}dx_{i}^{2}+d(x_0\cos{x_{n+1}})^{2}+ d(x_0\sin{x_{n+1}})^{2}\\
&=\sum_{i=1}^{n}dx_{i}^{2}+(\cos{x_{n+1}}^{2}+\sin{x_{n+1}}^{2})dx_{0}^{2}\\
&\hspace{0.4cm}+2x_0\left[ \cos{x_{n+1}}\dot{\cos}{x_{n+1}}+\sin{x_{n+1}}\dot{\sin}{x_{n+1}} \right] dx_{0}dx_{n+1}+x_0^{2}dx_{n+1}^{2}\\
&=\sum_{i=1}^{n}dx_{i}^{2}+dx_{0}^{2}+x_0^{2}dx_{n+1}^{2}.
\end{split}
\end{equation*}
From here we compute $g_{\beta, \gamma, \Lambda}$ as:
\begin{equation}\label{metriccalc2}
\begin{split}
g_{\beta, \gamma, \Lambda}^{n+1}=&\cyl_{\beta,\gamma, \Lambda}^{*}(\bend^{*}(g_{\Euc}^{n+2}))\\
&=\cyl_{\beta,\gamma, \Lambda}^{*}\left[\sum_{i=1}^{n}dx_{i}^{2}+dx_{0}^{2}+x_0^{2}dx_{n+1}^{2}\right]\\
&=\cyl_{\beta, \gamma, \Lambda}^{*}\left[\sum_{i=1}^{n}dx_{i}^{2}\right]+\dot{\alpha}(r)^{2}dr^{2}+(\Lambda+\alpha(r))^{2}dt^{2},
\end{split}
\end{equation}
where $\alpha:[0,b]\rightarrow [0,\infty)$ is the smooth function determined by $\beta$ and satisfying the appropriate conditions of \ref{beta0}.
We will temporarily examine the term:
$$
\cyl_{\beta,\gamma, \Lambda}^{*}\left[\sum_{i=1}^{n}dx_{i}^{2}\right]=\sum_{i=1}^{n}\delta_{ij}(\theta_{i}\dot{\beta}(r)dr+\beta(r).d\theta_i)(\theta_{j}\dot{\beta}(r)dr+\beta(r).d\theta_j).
$$
This term
simplifies as follows.
\begin{equation*}
\begin{split}
&\sum_{i=1}^{n}\delta_{ij}(\theta_{i}\dot{\beta}(r)dr+\beta(r).d\theta_i)(\theta_{j}\dot{\beta}(r)dr+\beta(r).d\theta_j)\\
=&\dot{\beta}(r)^{2}\sum_{i=1}^{n}\theta_{i}^{2}dr^{2}+2\sum_{i=1}^{n}\beta(r)\dot{\beta}(r)\theta_{i}d\theta_idr+\beta(r)^{2}\sum_{i=1}^{n}d\theta_{i}^{2}\\
=&\dot{\beta}(r)^{2}dr^{2}+\beta(r)^{2}\sum_{i=1}^{n}d\theta_{i}^{2}\\
=&\dot{\beta}(r)^{2}dr^{2}+\beta(r)^{2}ds_{n-1}^{2}.
\end{split}
\end{equation*}
The third line above follows from the fact that
$\sum_{i=1}^{n}\theta_{i}^{2}=1$ and so $\sum_{i=1}^{n}\theta_{i}d\theta_{i}=0$.
%
%\begin{equation*}
%\sum_{i=1}^{n}\theta_{i}^{2}=1, \ \ \ \mbox{and so} \ \ \ 
%\sum_{i=1}^{n}\theta_{i}d\theta_{i}=0.
%\end{equation*}
%and so
%\begin{equation*}
%\sum_{i=1}^{n}\theta_{i}d\theta_{i}=0.
%\end{equation*}
The fourth line then follows from the fact that
$\sum_{i=1}^{n}d\theta_{i}^{2}$ restricted to $S^{n-1}$ is precisely
the standard round metric of radius $1$. Returning to
(\ref{metriccalc2}), we now have that:
\begin{equation*}
\begin{split}
g_{\beta, \gamma, \Lambda}^{n+1}:=\cyl_{\beta, \gamma, \Lambda}^{*}(\bend_{\gamma}^{*}(g_{\Euc}^{n+2}))&=\dot{\beta}(r)^{2}dr^{2}+\beta(r)^{2}ds_{n-1}^{2}+\dot{\alpha}(r)^{2}dr^{2}+(\Lambda+\alpha(r))^{2}dt^{2}\\
&=[\dot{\beta}(r)^{2}+\dot{\alpha}(r)^{2}]dr^{2}+\beta(r)^{2}ds_{n-1}^{2}+(\Lambda+\alpha(r))^{2}dt^{2}\\
&=dr^{2}+\beta(r)^{2}ds_{n-1}^{2}+(\Lambda+\alpha(r))^{2}dt^{2}.
\end{split}
\end{equation*}
The scalar curvature of this metric, denoted $s_{\beta, \gamma, \Lambda}$, is given by the formula:
\begin{equation*}%\label{scaldeftorp2}
s_{\beta, \gamma, \Lambda}= {(n-1)(n-2)}\Big{[}\frac{1-{\dot{\beta}}^{2}}{\beta^{2}}\Big{]}-{2(n-1)}\frac{\ddot{\beta}}{\beta}-\frac{2(n+1)}{\beta}\Big{[}\frac{\dot{\alpha}.\dot{\beta}}{\Lambda+\alpha}\Big{]}- 2\frac{\ddot{\alpha}}{\Lambda+\alpha}.
\end{equation*}
The conditions imposed on $\alpha$ and $\beta$ mean the first, second and fourth terms in this expression are all non-negative and indeed always sum to something positive. By choosing sufficiently large $\Lambda$, we can minimise any negativity arising from the third term, proving the lemma.
\end{proof}

We are now ready to construct the desired metric bending. We assume that $\beta:[0,b]\rightarrow[0,\infty)$ is as in Lemma \ref{easybend} with corresponding smooth function $\alpha$. Let $$w_{\Lambda, \gamma}:[0,b]\times[-3, \gamma+3]\longrightarrow [1, \infty)$$ denote a family of smooth functions where $\gamma\in[0,\frac{\pi}{2}]$ is a constant and which satisfy the following properties:
\begin{enumerate}
\item[(i.)] When $\gamma=0$, $w_{\Lambda, 0}(r,t)=1$ for all $(r,t)\in [0,b]\times[-3, \gamma+3]$.
\item[(ii.)] There is a constant $\gamma_0\in(0, \frac{\pi}{2})$ so that whenever $\gamma\in[\gamma_0, \frac{\pi}{2}]$:
$$
w_{\Lambda, \gamma}(r,t) =
  \begin{cases}		1 & \text{if $-3\leq t\leq -2$,} \\
                                \Lambda  & \text{if $-\frac{3}{2}\leq t\leq -\frac{1}{2}$,} \\
                                 \Lambda+\alpha(r) & \text{if $t\in[0,\gamma]$,} \\
   \Lambda & \text{if $\gamma+\frac{1}{2} \leq t\leq \gamma+\frac{3}{2}$,}\\
   1 & \text{if $\gamma+2\leq t\leq \gamma+3$.} 
  \end{cases}
$$
\item[(iii.)]  $\frac{\partial{w_{\Lambda, \gamma}}}{\partial r}(r,t)=0$ when $t\in[-3,-\frac{1}{2}]\cup [\gamma+\frac{1}{2}, 3]$. 
\item[(iv.)] $\left|\frac{\partial^{(k)}{w_{\Lambda, \gamma}}}{\partial r^{(k)}}(r,t)\right|\leq \left| \frac{\partial^{(k)}{\alpha}}{\partial r^{(k)}}(r)\right|$ for all $k\in\{0,1,2,\cdots\}$.
\item[(v.)] $\frac{\partial^{2}{w_{\Lambda, \gamma}}}{\partial r^{2}}\leq 0$.
\end{enumerate}
To aid the reader, we provide in Fig. \ref{hardbendadjust} a schematic description of such a function, in the case when $\gamma>\gamma_0$, on its rectangular domain $[0,b]\times[-3, \gamma+3]$. The interval $[0,b]$ is represented on the vertical here. Finally, we specify the family of metrics $\{ g_{\gamma-\bend}^{n+1}(\beta, \Lambda): \gamma\in[0,\frac{\pi}{2}]\}$ defined by the formula:
$$
g_{\gamma-\bend}^{n+1}(\beta, \Lambda):=dr^{2}+w_{\theta}(r,t)^{2}dt^{2}+\eta_{\delta}(r)^{2}ds_{n-1}^{2},
$$
where $r\in [0,b]$, $t\in [-3, \gamma+3]$ and $\gamma\in[0,\frac{\pi}{2}]$.
\begin{figure}[!htbp]
\vspace{-1.5cm}
\hspace{-9cm}
\begin{picture}(0,0)%
\includegraphics{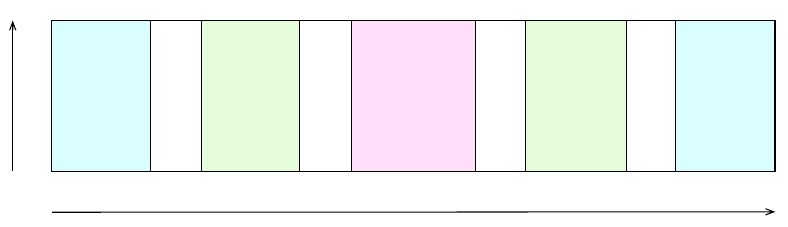}%
\end{picture}%
\setlength{\unitlength}{3947sp}%
\begingroup\makeatletter\ifx\SetFigFont\undefined%
\gdef\SetFigFont#1#2#3#4#5{%
  \reset@font\fontsize{#1}{#2pt}%
  \fontfamily{#3}\fontseries{#4}\fontshape{#5}%
  \selectfont}%
\fi\endgroup%
\begin{picture}(2079,2559)(1902,-5227)
\put(2700,-4200){\makebox(0,0)[lb]{\smash{{\SetFigFont{10}{8}{\rmdefault}{\mddefault}{\updefault}{\color[rgb]{0,0,0}$1$}%
}}}}

\put(3900,-4200){\makebox(0,0)[lb]{\smash{{\SetFigFont{10}{8}{\rmdefault}{\mddefault}{\updefault}{\color[rgb]{0,0,0}$\Lambda$}%
}}}}
\put(4900,-4200){\makebox(0,0)[lb]{\smash{{\SetFigFont{10}{8}{\rmdefault}{\mddefault}{\updefault}{\color[rgb]{0,0,0}$\Lambda+\alpha(r)$}%
}}}}
\put(6500,-4200){\makebox(0,0)[lb]{\smash{{\SetFigFont{10}{8}{\rmdefault}{\mddefault}{\updefault}{\color[rgb]{0,0,0}$\Lambda$}%
}}}}
\put(7700,-4200){\makebox(0,0)[lb]{\smash{{\SetFigFont{10}{8}{\rmdefault}{\mddefault}{\updefault}{\color[rgb]{0,0,0}$1$}%
}}}}
\put(2200,-4950){\makebox(0,0)[lb]{\smash{{\SetFigFont{10}{8}{\rmdefault}{\mddefault}{\updefault}{\color[rgb]{0,0,0}$-3$}%
}}}}
\put(2900,-4950){\makebox(0,0)[lb]{\smash{{\SetFigFont{10}{8}{\rmdefault}{\mddefault}{\updefault}{\color[rgb]{0,0,0}$-2$}%
}}}}
\put(3400,-4950){\makebox(0,0)[lb]{\smash{{\SetFigFont{10}{8}{\rmdefault}{\mddefault}{\updefault}{\color[rgb]{0,0,0}$-\frac{3}{2}$}%
}}}}
\put(4100,-4950){\makebox(0,0)[lb]{\smash{{\SetFigFont{10}{8}{\rmdefault}{\mddefault}{\updefault}{\color[rgb]{0,0,0}$-\frac{1}{2}$}%
}}}}
\put(4700,-4950){\makebox(0,0)[lb]{\smash{{\SetFigFont{10}{8}{\rmdefault}{\mddefault}{\updefault}{\color[rgb]{0,0,0}$0$}%
}}}}
\put(5650,-4950){\makebox(0,0)[lb]{\smash{{\SetFigFont{10}{8}{\rmdefault}{\mddefault}{\updefault}{\color[rgb]{0,0,0}$\gamma$}%
}}}}
\put(6000,-4950){\makebox(0,0)[lb]{\smash{{\SetFigFont{10}{8}{\rmdefault}{\mddefault}{\updefault}{\color[rgb]{0,0,0}$\gamma+\frac{1}{2}$}%
}}}}
\put(6700,-4950){\makebox(0,0)[lb]{\smash{{\SetFigFont{10}{8}{\rmdefault}{\mddefault}{\updefault}{\color[rgb]{0,0,0}$\gamma+\frac{3}{2}$}%
}}}}
\put(7200,-4950){\makebox(0,0)[lb]{\smash{{\SetFigFont{10}{8}{\rmdefault}{\mddefault}{\updefault}{\color[rgb]{0,0,0}$\gamma+2$}%
}}}}
\put(7950,-4950){\makebox(0,0)[lb]{\smash{{\SetFigFont{10}{8}{\rmdefault}{\mddefault}{\updefault}{\color[rgb]{0,0,0}$\gamma+3$}%
}}}}
\put(4800,-5270){\makebox(0,0)[lb]{\smash{{\SetFigFont{10}{8}{\rmdefault}{\mddefault}{\updefault}{\color[rgb]{0,0,0}$t$}%
}}}}
\put(2200,-4750){\makebox(0,0)[lb]{\smash{{\SetFigFont{10}{8}{\rmdefault}{\mddefault}{\updefault}{\color[rgb]{0,0,0}$0$}%
}}}}
\put(2200,-3650){\makebox(0,0)[lb]{\smash{{\SetFigFont{10}{8}{\rmdefault}{\mddefault}{\updefault}{\color[rgb]{0,0,0}$b$}%
}}}}
\put(1900,-4150){\makebox(0,0)[lb]{\smash{{\SetFigFont{10}{8}{\rmdefault}{\mddefault}{\updefault}{\color[rgb]{0,0,0}$r$}%
}}}}
\end{picture}%
\caption{The function $\omega_{\gamma, \Lambda}$} 
\label{hardbendadjust}
\end{figure} 
%\label{scaldeftorp2}
\begin{lemma}\label{bendisot}
Let $n\geq 3$ and let $\beta:[0,b]\rightarrow[0,\infty)$ be as in Lemma \ref{easybend}. Then there exists $\Lambda_{\beta}>0$ so that for all $\Lambda\geq \Lambda_{\beta}$ the map $\gamma\mapsto g_{\gamma-\bend}^{n+1}(\beta, \Lambda)$, where $\gamma\in[0,\frac{\pi}{2}]$, is an isotopy through metrics of positive scalar curvature. 
\end{lemma}

\begin{proof}

The scalar curvature of the metric $g_{\gamma-\bend}^{n+1}(\beta, \Lambda)$ is given by the formula:
\begin{equation*}%\label{scaldeftorp2}
s_{\gamma-\bend, \beta, \Lambda}= {(n-1)(n-2)}\Big{[}\frac{1-{\dot{\beta}}^{2}}{\beta^{2}}\Big{]}-{2(n-1)}\frac{\ddot{\beta}}{\beta}-\frac{2(n+1)}{\beta}\Big{[}\frac{\frac{\partial{\omega_{\gamma, \Lambda}}}{\partial r}.\dot{\beta}}{\omega_{\gamma, \Lambda}}\Big{]}- 2\frac{\ddot{\alpha}}{\omega_{\gamma, \Lambda}}.
\end{equation*}
As with the analogous calculation in the proof of Lemma \ref{easybend}, the sum of the first, second and fourth terms in this expression is universally positive. Moreover the third term is zero when $\gamma=0$. This allows us to specify a small angle $\gamma_0>0$ over which to make an initial bend. More precisely, there is a $\gamma_0>0$ for which this metric has positive scalar curvature for all $\gamma\in[0,\gamma_0]$ and all $\Lambda\geq 0$. We now consider the case where $\gamma\in[\gamma_0, \frac{\pi}{2}]$. We note that $\frac{\partial{\omega_{\gamma, \Lambda}}}{\partial r}(r,t)=0$ unless $t\in[-\frac{1}{2}, \gamma+\frac{1}{2}]$. In this case the denominator $\omega_{\gamma, \Lambda}\geq \Lambda$. Hence, by choosing sufficiently large $\Lambda_{\beta}>0$, we can minimise the effect of the third term and maintain positivity throughout for all $\Lambda\geq \Lambda_{\beta}$.
\end{proof}

We will be particularly interested in the case when the smooth function $\beta$ above is the torpedo function $\eta_{\delta, \lambda_1}:[0, \delta\frac{\pi}{2}]\rightarrow[0,\infty)$ for some $\delta>0$. Noting that the choice of $\lambda_1$ in this function has no effect on curvature in the bending construction above, the following corollary is just a special case of Lemma \ref{bendisot}.
\begin{corollary}\label{bendbound} Let $n\geq 3$. Given $\delta>0$, there is a constant $\Lambda_{\delta}>0$ so that for all $\Lambda\geq \Lambda_{\delta}$ and all $\gamma\in[0,\frac{\pi}{2}]$, the metric $g_{\gamma-\bend}^{n+1}(\eta_{\delta, \lambda_1}, \Lambda)$ has positive scalar curvature. 
\end{corollary}

We are finally in a position to define the ``boot" metric. We will build this by attaching together various parts as shown below in Fig. \ref{bootdetail2}. We begin with the so-called toe of the boot. This is the previously defined metric, $\hat{g}_{\tor}^{n+1}(\delta)_{\lambda_1, \lambda_2}$, defined on the region $D_{\stret}^{n+1}$. We will denote this first region $R_1$. Recall that the $\lambda_1$ parameter denotes the length of the vertical straight part of the toe while $\lambda_2$ denotes the length of the horizontal part. %In our case we will assume that $\lambda_1=1$ (the vertical height) and allow $\lambda_2\geq 0$ to be arbitrary, thus denoting the metric $\hat{g}_{\tor}^{n+1}(\delta)_{1, \lambda_2}$. 
We will for now make no restrictions on the size of $\delta$. Returning to the bending construction above, we replace $\beta$ by the smooth torpedo function $\eta_{\delta, \lambda_1}:[0, \delta\frac{\pi}{2}]\rightarrow[0,\infty)$ for some $\delta>0$. 
%Choosing some $\Lambda>\Lambda_{\delta}$, the metric $g_{\frac{\pi}{2}-\bend}^{n+1}(\eta_{\delta, \lambda_1}, \Lambda)$ has positive scalar curvature. 
Recall that this metric is the result of bending a cylinder metric of the form $g_{\tor}^{n}(\delta)_{\lambda_1}+dt^{2}$ around an angle of $\frac{\pi}{2}$. The neck length $\lambda_1$ coincides with the straight height of the toe metric on $R_1$. Recall that the metric $g_{\frac{\pi}{2}-\bend}^{n+1}(\eta_{\delta, \lambda_1}, \Lambda)$ is defined on $D^{n}\times [0, \frac{\pi}{2}+3]$ and so we regard this space as the second region, $R_2$.

%Making use of the identification map $\xi_{1+\frac{\pi}{2}}:[0,1]\rightarrow [0, \frac{\pi}{2}+3]$ defined at the end of section \ref{isoconcfam}, allows us to pull back $g_{\frac{\pi}{2}-\bend}^{n+1}(\eta_{\delta, \lambda_1, \Lambda})$ (defined as a metric on $D^{n}\times [0, \frac{\pi}{2}+3]$) to $D^{n}\times I$ which we denote $R_2$, the second region. 

We then take another cylinder $D^{n}\times [0,\lambda_3]$ for some arbitrary $\lambda_3\geq 1$, which we denote $R_3$, and equip it with the product metric $g_{\tor}^{n}(\delta)_{\lambda_1}+dt^{2}$.
At this stage, we can smoothly glue together $(R_1, \hat{g}_{\tor}^{n+1}(\delta)_{\lambda_1, \lambda_2})$ and $(R_2, g_{\frac{\pi}{2}-\bend}(\eta_{\delta, \lambda_1, \Lambda})$ by attaching along $D^{n}\times\{0\}\subset R_2$ in the obvious way. Similarly, we can attach $(R_3, g_{\tor}^{n}(\delta)+dt^{2})$ to $(R_2, g_{\frac{\pi}{2}-\bend}(\eta_{\delta, \lambda_1, \Lambda})$, this time gluing along  $D^{n}\times\{1\}\subset R_2$.   
Finally, we wish to attach a ``corner" region, $R_4$, which extends the boundary of $R_2$ so that the union $R_1\cup R_2\cup R_3\cup R_4$ is diffeomorphic to the cylinder $D^{n}\times I$. This region is a product of $S^{n-1}$ with a piece of $2$-dimensional Euclidean space, which takes the form shown in Fig. \ref{bootdetail2} below. The metrics on each of these regions glues together in the obvious way to form a smooth metric: 
$$\hat{g}_{\tor}^{n+1}(\delta)_{\lambda_1, \lambda_1'}\cup g_{\frac{\pi}{2}-\bend}(\eta_{\delta, \lambda_1, \Lambda})\cup (g_{\tor}^{n}(\delta)_{\lambda_1}+dt^{2})\cup (\delta^{2}ds_{n-1}^{2}+dr^{2}+dt^{2}),$$
as suggested in Fig. \ref{bootdetail2}. Such a metric is known as a {\em $\delta$-boot} metric. 

\begin{figure}[!htbp]
\vspace{4.5cm}
\hspace{-3cm}
\begin{picture}(0,0)
\includegraphics{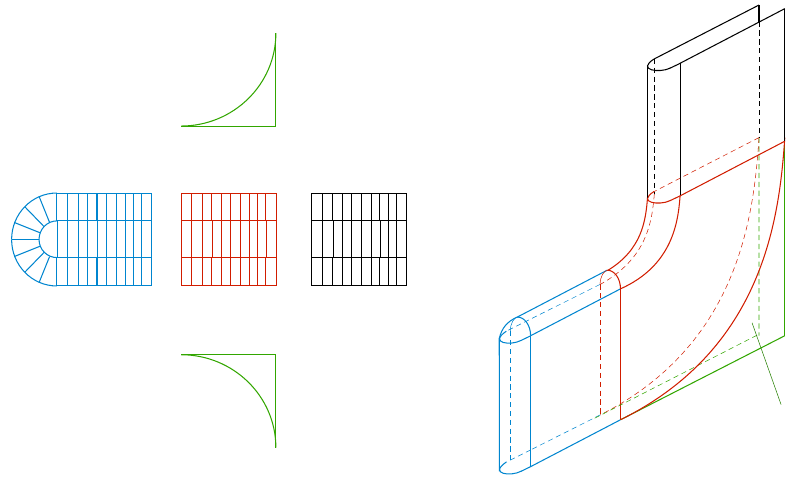}%
\end{picture}
\setlength{\unitlength}{3947sp}%\includegraphics{RelGLpic/
\begingroup\makeatletter\ifx\SetFigFont\undefined%
\gdef\SetFigFont#1#2#3#4#5{%
  \reset@font\fontsize{#1}{#2pt}%
  \fontfamily{#3}\fontseries{#4}\fontshape{#5}%
  \selectfont}%
\fi\endgroup%
\begin{picture}(5079,1559)(1902,-7227)
\put(2000,-5800){\makebox(0,0)[lb]{\smash{{\SetFigFont{10}{8}{\rmdefault}{\mddefault}{\updefault}{\color[rgb]{0,0.6,0.8}$R_1=D_{\stret}^{n+1}$}%
}}}}
\put(3150,-4800){\makebox(0,0)[lb]{\smash{{\SetFigFont{10}{8}{\rmdefault}{\mddefault}{\updefault}{\color[rgb]{0.8,0,0}$R_2\cong D^{n}\times I$}%
}}}}
\put(4200,-5800){\makebox(0,0)[lb]{\smash{{\SetFigFont{10}{8}{\rmdefault}{\mddefault}{\updefault}{\color[rgb]{0,0,0}$R_3\cong D^{n}\times I$}%
}}}}
\put(2650,-6850){\makebox(0,0)[lb]{\smash{{\SetFigFont{10}{8}{\rmdefault}{\mddefault}{\updefault}{\color[rgb]{0,0.6,0}$R_4\cong S^{n-1}\times D^{2}$}
}}}}
\put(5150,-6850){\makebox(0,0)[lb]{\smash{{\SetFigFont{10}{8}{\rmdefault}{\mddefault}{\updefault}{\color[rgb]{0,0.6,0.8}$\hat{g}_{\tor}^{n+1}(\delta)$}
}}}}
\put(6000,-5300){\makebox(0,0)[lb]{\smash{{\SetFigFont{10}{8}{\rmdefault}{\mddefault}{\updefault}{\color[rgb]{0.8,0,0}$g_{\bend\tor}^{n}(\delta)$}%
}}}}
\put(5900,-4300){\makebox(0,0)[lb]{\smash{{\SetFigFont{10}{8}{\rmdefault}{\mddefault}{\updefault}{\color[rgb]{0,0,0}$g_{\tor}^{n}(\delta)+dt^{2}$}%
}}}}
\put(7350,-6650){\makebox(0,0)[lb]{\smash{{\SetFigFont{10}{8}{\rmdefault}{\mddefault}{\updefault}{\color[rgb]{0,0.6,0}$\delta^{2}ds_{n-1}^{2}+dr^{2}+dt^{2}$}
}}}}
\end{picture}%
\caption{The various components of the boot metric}
\label{bootdetail2}
\end{figure}

Obviously, there are various parameter choices involved here such as the lengths: $\lambda_1, \lambda_2$ and $\lambda_3$. Indeed, there are four straight edge pieces each of which has a length which can be adjusted. To make things more precise, we describe these lengths with a vector $\bar{l}=(l_1, l_2, l_3, l_4)\in [0,\infty)^{4}$. The first component, $l_1=\lambda_1$ denotes the height of the straight part of the toe, as above. Then, moving anti-clockwise around the boot the other components respectively describe the lengths of the various pieces, as shown in Fig. \ref{bootdetail21}. The corresponding metric is denoted: $g_{\boot}^{n+1}(\delta)_{\Lambda, \bar{l}}$. It is convenient to regard the union $R_1\cup R_2\cup R_3\cup R_4$ on which this metric lies as $D^{n}\times [0,l_3]$ where $l_3$ is the height of the boot. 
Note that $g_{\boot}^{n+1}(\delta)_{\Lambda, \bar{l}}$ has product structures:
$$g_{\tor}^{n}(\delta)_{l_4}+dt^{2} \text{ when $t$ is near $l_3$}   \quad \text{ and } \quad g_{\tor}^{n}(\delta)_{l_2}+dt^{2} \text{ when $t$ is near $0$.} $$

\begin{figure}[!htbp]
\vspace{5.5cm}
\hspace{3cm}
\begin{picture}(0,0)
\includegraphics{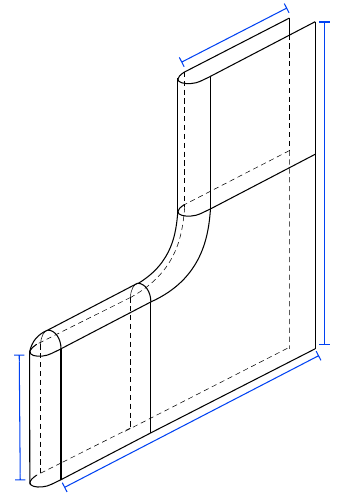}%
\end{picture}
\setlength{\unitlength}{3947sp}%\includegraphics{RelGLpic/
\begingroup\makeatletter\ifx\SetFigFont\undefined%
\gdef\SetFigFont#1#2#3#4#5{%
  \reset@font\fontsize{#1}{#2pt}%
  \fontfamily{#3}\fontseries{#4}\fontshape{#5}%
  \selectfont}%
\fi\endgroup%
\begin{picture}(5079,1559)(1902,-7227)
\put(1800,-6650){\makebox(0,0)[lb]{\smash{{\SetFigFont{10}{8}{\rmdefault}{\mddefault}{\updefault}{\color[rgb]{0,0,0.9}$l_1$}%
}}}}
\put(3600,-3500){\makebox(0,0)[lb]{\smash{{\SetFigFont{10}{8}{\rmdefault}{\mddefault}{\updefault}{\color[rgb]{0,0,1}$l_4$}%
}}}}
\put(4500,-4800){\makebox(0,0)[lb]{\smash{{\SetFigFont{10}{8}{\rmdefault}{\mddefault}{\updefault}{\color[rgb]{0,0,1}$l_3$}%
}}}}
\put(3400,-6800){\makebox(0,0)[lb]{\smash{{\SetFigFont{10}{8}{\rmdefault}{\mddefault}{\updefault}{\color[rgb]{0,0,1}$l_2$}
}}}}
\end{picture}%
\caption{The boot metric $g_{\boot}^{n+1}(\delta)_{\Lambda, \bar{l}}$ with $\bar{l}=(l_1, l_2, l_3, l_4)$}
\label{bootdetail21}
\end{figure}

When it comes to maintaining positive scalar curvature, there are of course constraints on the choices of $\Lambda$ and $\bar{l}$. From Corollary \ref{bendbound}, we know that for each $\delta>0$ there is a constant $\Lambda_\delta>0$ so that positive scalar curvature on the region $R_2$ is preserved provided the bending parameter $\Lambda$ satisfies $\Lambda\geq \Lambda_{\delta}$. The components $l_1$ and $l_4$ can be any non-zero constants. However the base length, $l_2$, and height of the boot, $l_3$, are constrained below by some positive constants determined by the bending parameter $\Lambda\geq \Lambda_{\delta}$) as well as $l_1$ and $l_4$. With this in mind, we specify the space $C_{\boot,+}^{n+1}\subset(0,\infty)\times[0,\infty)\times[0,\infty)^{4}$ which consists of all triples $(\delta, \Lambda, \bar{l})$ so that the $L\geq L_{\delta}$ and $\bar{l}$ has components such that the metric
$g_{\boot}^{n+1}(\delta)_{\Lambda, \bar{l}}$ has positive scalar curvature. Elements of $C_{\boot,+}^{n+1}$ are deemed to be {\em psc-boot triples}.

Let us return to the standard cylinder $g_{\tor}^{n}(\delta)_{\lambda}+dt^{2}$ on $D^{n}\times I$. Lemma \ref{bendisot} allow us to construct an isotopy through positive scalar curvature metrics on $D^{n}\times I$ which turns the above cylinder metric into a boot metric. This is described in the following Lemma. Before stating it, we recall from the end of section \ref{isoconcfam}, the family of diffeomorphisms $\xi_{L}:[0,1]\rightarrow [0, L]$. Throughout this section, we will make use of this family to identify metrics on an arbitrary cylinder $D^{n}\times [0,L]$ with metrics on the standard cylinder $D^{n}\times I$ via the obvious pullback. Thus, we will often refer to metrics on $D^{n}\times I$ as metrics on $D^{n}\times L$ for some $L$ with the understanding that these metrics are related via this identification.

\begin{lemma}\label{bootisot}
 Let $n\geq 3$ and let $(\delta, \Lambda, \bar{l})\in C_{\boot,+}^{n+1}$ be a psc-boot triple. For any $L>0$, let $g_0$ be a psc-metric on $D^{n}\times I$ which takes the form of the cylinder metric $g_{\tor}^{n}(\delta)_{\lambda}+dt^{2}$ on $D^{n}\times [0,L]$. 
 There is an isotopy of psc-metrics on $D^{n}\times I$, $\tau\mapsto g_{\boot}^{n+1}(\delta)_{\Lambda, \bar{l}}(\tau)$, $\tau\in I$ satisfying: 
 \begin{enumerate}
\item[(i.)] $g_{\boot}^{n+1}(\delta)_{\Lambda, \bar{l}}(0)=g_0,$
\item[(ii.)] $g_{\boot}^{n+1}(\delta)_{\Lambda, \bar{l}}(1)=g_{\boot}^{n+1}(\delta)_{\Lambda, \bar{l}}$.
\item[(iii.)] For all $\tau\in [0,1]$, 
$$
g_{\boot}^{n+1}(\delta)_{\Lambda, \bar{l}}(\tau) =
  \begin{cases}
   g_{\tor}^{n}(\delta)_{\lambda_2(\tau)}+dt^{2}& \text{ for some $\lambda_2(\tau)\in[\lambda, l_2]$ when $t$ is near $0$,} \\
   g_{\tor}^{n}(\delta)_{\lambda_4(\tau)}+dt^{2} & \text{ for some $\lambda_4(\tau)\in[\lambda, l_4]$ when $t$ is near $1$.}                                                                                 
  \end{cases}
$$
\end{enumerate}
\end{lemma}
\begin{proof}
This is an elementary application of Lemma \ref{bendisot}. Recall that the vector $\bar{l}=(l_1, l_2, l_3, l_4)$ denotes lengths of the various straight-edge pieces of the desired boot metric. In particular, $l_1$ corresponds to the straight height of the toe, $l_2$ the base, $l_3$ the total boot height and $l_4$ the straight length at the top. We begin by continuously stretching stretching the cylinder metric so that it takes the form $g_{\tor}^{n}(\delta)_{\lambda}+dt^{2}$ on the cylinder $X\times [0, l_3]$. We denote by $\lambda_{i}(\tau)$, the stretching parameter which varies linearly between $\lambda_i(0)$ and $\lambda_{i}(1)=l_i$ for each $i\in\{1,2,3,4\}$. Thus, $\lambda_3(0)=1$ and $\lambda_{3}(1)=l_3$, while $\lambda_1(\tau)\in[0,\lambda_3(\tau)]$ and $\lambda_{1}(1)=l_1$.
The parameters $\lambda_2(\tau)$ and $\lambda_{4}(\tau)$ satisfy $\lambda\leq \lambda_{4}(\tau)\leq\lambda_{2}\leq l_2$ and $\lambda\leq \lambda_{4}(\tau)\leq l_2$ for all $\tau\in[0,1]$.

By pushing out two bends (controlled by angle $\gamma(\tau)\in[0,\frac{\pi}{2}]$ where $\tau\in I$) as depicted in Fig. \ref{bootdetail} below, we can push out a ``toe" from the original cylinder metric and obtain the desired ``boot". The lower bend, which begins at $t=l_1$ (with respect to the $[0,l_3]$ coordinate interval), is in a direction conducive to preserving positive scalar curvature and so we do not make use of Lemma \ref{bendisot} here. Indeed details of such bends are taken care of in Lemma 2.1 of \cite{Walsh2}. The upper bend is more delicate and does require the gradual bending over a wide bending arc as in Lemma \ref{bendisot}. However, as $\Lambda$ is part of a psc-boot triple, Corollary \ref{bendbound} guarantees that this bending preserves positive scalar curvature. Finally, the remaining parameters of $\bar{l}$, $l_2$ and $l_4$ can be achieved by obvious stretching. 
\end{proof}
\begin{figure}[!htbp]
\vspace{4.5cm}
\hspace{-2.5cm}
\begin{picture}(0,0)
\includegraphics{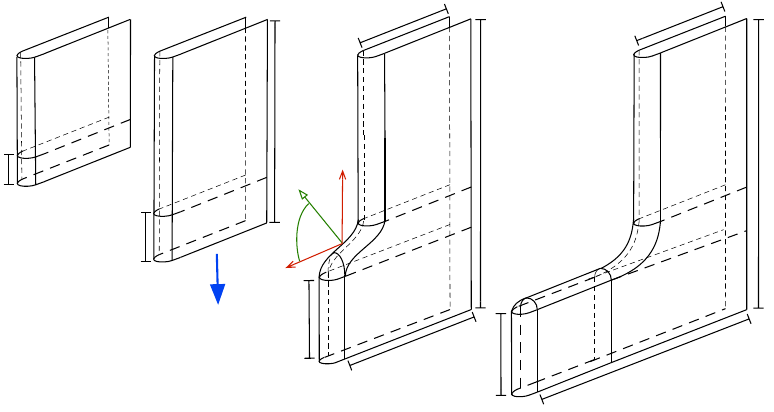}%
\end{picture}
\setlength{\unitlength}{3947sp}%\includegraphics{RelGLpic/
\begingroup\makeatletter\ifx\SetFigFont\undefined%
\gdef\SetFigFont#1#2#3#4#5{%
  \reset@font\fontsize{#1}{#2pt}%
  \fontfamily{#3}\fontseries{#4}\fontshape{#5}%
  \selectfont}%
\fi\endgroup%
\begin{picture}(5079,1559)(1902,-7227)
\put(1500,-5360){\makebox(0,0)[lb]{\smash{{\SetFigFont{10}{8}{\rmdefault}{\mddefault}{\updefault}{\color[rgb]{0,0,0}$\lambda_{1}(0)$}%
}}}}

\put(1850,-5600){\makebox(0,0)[lb]{\smash{{\SetFigFont{10}{8}{\rmdefault}{\mddefault}{\updefault}{\color[rgb]{0,0,0}$g_{\tor}^{n}(\delta)_{\lambda}+dt^{2}$}%
}}}}
\put(2600,-5900){\makebox(0,0)[lb]{\smash{{\SetFigFont{10}{8}{\rmdefault}{\mddefault}{\updefault}{\color[rgb]{0,0,0}$\lambda_{1}(\tau)$}%
}}}}

\put(4050,-5000){\makebox(0,0)[lb]{\smash{{\SetFigFont{10}{8}{\rmdefault}{\mddefault}{\updefault}{\color[rgb]{0,0,0}$\lambda_3(\tau)$}%
}}}}
\put(4230,-5850){\makebox(0,0)[lb]{\smash{{\SetFigFont{10}{8}{\rmdefault}{\mddefault}{\updefault}{\color[rgb]{0,0.6,0}{\tiny ${\gamma(\tau)}$}}%
}}}}
\put(6150,-5000){\makebox(0,0)[lb]{\smash{{\SetFigFont{10}{8}{\rmdefault}{\mddefault}{\updefault}{\color[rgb]{0,0,0}$g_{\boot}^{n+1}(\delta)_{\Lambda, \bar{l}}$}%
}}}}
\put(5750,-5300){\makebox(0,0)[lb]{\smash{{\SetFigFont{10}{8}{\rmdefault}{\mddefault}{\updefault}{\color[rgb]{0,0,0}$l_3$}%
}}}}
\put(5040,-6850){\makebox(0,0)[lb]{\smash{{\SetFigFont{10}{8}{\rmdefault}{\mddefault}{\updefault}{\color[rgb]{0,0,0}$\lambda_2(\tau)$}
}}}}
\put(4700,-4050){\makebox(0,0)[lb]{\smash{{\SetFigFont{10}{8}{\rmdefault}{\mddefault}{\updefault}{\color[rgb]{0,0,0}$\lambda_4(\tau)$}
}}}}
\put(4150,-6500){\makebox(0,0)[lb]{\smash{{\SetFigFont{10}{8}{\rmdefault}{\mddefault}{\updefault}{\color[rgb]{0,0,0}$l_1$}
}}}}

\put(5700,-6800){\makebox(0,0)[lb]{\smash{{\SetFigFont{10}{8}{\rmdefault}{\mddefault}{\updefault}{\color[rgb]{0,0,0}$l_1$}
}}}}
\put(7050,-6930){\makebox(0,0)[lb]{\smash{{\SetFigFont{10}{8}{\rmdefault}{\mddefault}{\updefault}{\color[rgb]{0,0,0}$l_2$}
}}}}
%\put(7650,-5300){\makebox(0,0)[lb]{\smash{{\SetFigFont{10}{8}{\rmdefault}{\mddefault}{\updefault}{\color[rgb]{0,0,0}$t_\delta$}
%}}}}
\put(7100,-4050){\makebox(0,0)[lb]{\smash{{\SetFigFont{10}{8}{\rmdefault}{\mddefault}{\updefault}{\color[rgb]{0,0,0}$l_4$}%
}}}}
\put(8000,-5300){\makebox(0,0)[lb]{\smash{{\SetFigFont{10}{8}{\rmdefault}{\mddefault}{\updefault}{\color[rgb]{0,0,0}$l_3$}
}}}}
\end{picture}%
\caption{The isotopy between the torpedo cylinder $g_{\tor}^{n}(\delta)_{\lambda}+dt^{2}$ (far left) and the boot $g_{\boot}^{n+1}(\delta)_{\Lambda, \bar{l}}$ (far right) with the middle two images depicting various intermediary metrics $g_{\boot}^{n+1}(\delta)_{\Lambda, \bar{l}}(\tau)$ for some $\tau\in(0,1)$. }
\label{bootdetail}
\end{figure}

There is a certain space of psc-metrics that is worth specifying at this point. For a given $\delta>0$ and a psc-boot triple $(\delta, \Lambda, \bar{l})\in C_{\boot}^{n+1}$, consider the family of metrics $g_{\boot}^{n+1}(\delta)_{\Lambda, \bar{l}}(\tau)$, $\tau\in I$ as described in Lemma \ref{bootisot} above. We would like to describe a space consisting of the union of all such families for a fixed $\delta$. More, precisely we define:
$$\Riem_{\delta-\boot}^{+}(D^{n}\times I):=\{g\in\Riem^{+}(D^{n}\times I): g=g_{\boot}^{n+1}(\delta)_{\Lambda, \bar{l}}(\tau), \tau\in[0,1],(\delta, \Lambda, \bar{l})\in C_{\boot}^{n+1}\}. $$
The following proposition follows easily from Lemma \ref{bootisot}.
\begin{proposition}
Let $n\geq 3$. The space $\Riem_{\delta-\boot}^{+}(D^{n}\times I)$ is contractible.
\end{proposition}
\begin{proof}

From Lemma \ref{bootisot}, we can perform a deformation retract of $\Riem_{\delta-\boot}^{+}(D^{n}\times I)$ to the subspace consisting of all metrics which take the form of a cylinder $g_{\tor}^{n}(\delta)_\lambda+dt^{2}$ on $D^{n}\times [0,L]$ for some $\lambda, L>0$. An easy further deformation retract, continuously scaling both $\lambda$ and $L$ to 1 completes the proof.
\end{proof}

\subsection{Step metrics}\label{steps}
Before concluding this section, there is a class of metrics we wish to define which generalises the boot metrics discussed above and which will be of immense use later on. In the proof of Theorem A, we will encounter a problem of the sort encountered in the proof of Theorem \ref{Chernysh}, where a deformation we carry out to standardise certain metrics temporarily moves already standard metrics out of the standard space. In Theorem \ref{Chernysh}, we overcome this problem by defining a suitable intermediary space of almost standard metrics. In proving Theorem A, we will have to do something similar. The reader should view the space of ``step metrics", which we define below, as the principle tool in constructing an analogue of the space of ``almost standard metrics." This will become clearer when it comes to proving Theorem A.

Suppose we begin with a cylinder metric $g_{\tor}^{n-1}(\delta)_{\lambda}+dt^{2}$ on $D^{n}\times I$. Given a psc-boot triple $(\delta, \Lambda, \bar{l})$, Lemma \ref{bootisot} gives us a family of psc-metrics $g_{\boot}^{n+1}(\delta)_{\Lambda, \bar{l}}(\tau), \tau\in[0,1]$. Denoting this psc-boot triple $(\delta, \Lambda^{1}, \bar{l}^{1})$, we choose some $\tau^{1}\in[0,1]$ and consider the metric: $g_{\boot}^{n+1}(\delta)_{\Lambda^{1}, \bar{l}^{1}}(\tau^{1})$. We denote the straight edge length parameters by $\lambda_{i}^{1}$, where $i\in\{1,2,3,4\}$. Recall that making use of the identification, $\xi_{\lambda_3^{1}(\tau^{1})}:[0,1]\rightarrow [0, \lambda_3^{1}(\tau^{1})]$ described at the end of section \ref{isoconcfam}, allows us to regard this as a metric on $D^{n}\times [0, \lambda_3^{1}(\tau^{1})]$. This metric takes the form of a product $g_{\tor}^{n}(\delta)_{\lambda_{2}^{1}(\tau^{1})}$ on the sub-cylinder $D^{n}\times [0, \lambda_1^{1}(\tau^{1})]\subset D^{n}\times [0, \lambda_3^{1}(\tau^{1})]$. Restricting attention to the sub-cylinder metric $(D^{n}\times [0, \lambda_1^{1}(\tau^{1})], g_{\tor}^{n-1}(\delta)_{\lambda_{2}^{1}(\tau^{1})}+dt^{2})$ we choose another psc-boot triple $(\delta, \Lambda^{2}, \bar{l}^{2})$ and reapply the isotopy from Lemma \ref{bootisot} up to some stage $\tau^{2}\in I$. In this case the straight edge length parameters will be denoted $\lambda_i^{2}$ and we will insist that $\lambda_{4}^{2}(\tau^{2})=\lambda_{2}^{1}(\tau)$ for all $\tau^{2}\in I$. Thus, as we adjust the original metric below $t=\lambda_1^{1}(\tau^{1})$, we guarantee smooth transition to the unadjusted region above. We denote the resulting metric: $ g_{\boot}^{n+1}(\delta)_{(\Lambda^{1}, \Lambda^{2}), (\bar{l}^{1}, \bar{l}^{2})}(\tau^{1}, \tau^{2})$, depicted in Fig. \ref{bootdetail3} below.

\begin{figure}[!htbp]
\vspace{5.5cm}
\hspace{0cm}
\begin{picture}(0,0)
\includegraphics{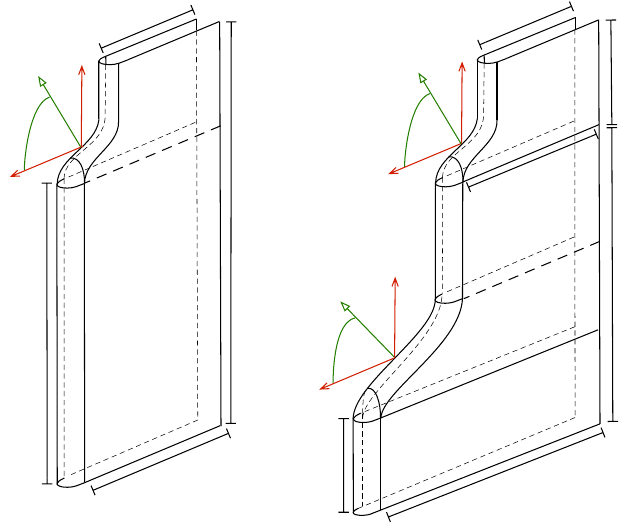}%
\end{picture}
\setlength{\unitlength}{3947sp}%\includegraphics{RelGLpic/
\begingroup\makeatletter\ifx\SetFigFont\undefined%
\gdef\SetFigFont#1#2#3#4#5{%
  \reset@font\fontsize{#1}{#2pt}%
  \fontfamily{#3}\fontseries{#4}\fontshape{#5}%
  \selectfont}%
\fi\endgroup%
\begin{picture}(5079,1559)(1902,-7227)
\put(2100,-4170){\makebox(0,0)[lb]{\smash{{\SetFigFont{10}{8}{\rmdefault}{\mddefault}{\updefault}{\color[rgb]{0,0.6,0}{\tiny $\gamma(\tau^{1})$}}%
}}}}
\put(5125,-4150){\makebox(0,0)[lb]{\smash{{\SetFigFont{10}{8}{\rmdefault}{\mddefault}{\updefault}{\color[rgb]{0,0.6,0}{\tiny $\gamma(\tau^{1})$}}%
}}}}
\put(1700,-5600){\makebox(0,0)[lb]{\smash{{\SetFigFont{10}{8}{\rmdefault}{\mddefault}{\updefault}{\color[rgb]{0,0,0}{$\lambda_{1}^{1}(\tau^{1})$}}%
}}}}
\put(3100,-6850){\makebox(0,0)[lb]{\smash{{\SetFigFont{10}{8}{\rmdefault}{\mddefault}{\updefault}{\color[rgb]{0,0,0}$\lambda_{2}^{1}(\tau^{1})$}%
}}}}
\put(3750,-4850){\makebox(0,0)[lb]{\smash{{\SetFigFont{10}{8}{\rmdefault}{\mddefault}{\updefault}{\color[rgb]{0,0,0}$\lambda_{3}^{1}(\tau^{1})$}%
}}}}
\put(2760,-3120){\makebox(0,0)[lb]{\smash{{\SetFigFont{10}{8}{\rmdefault}{\mddefault}{\updefault}{\color[rgb]{0,0,0}$\lambda_{4}^{1}(\tau^{1})$}%
}}}}

\put(4550,-5850){\makebox(0,0)[lb]{\smash{{\SetFigFont{10}{8}{\rmdefault}{\mddefault}{\updefault}{\color[rgb]{0,0.6,0}{\tiny $\gamma(\tau^{2})$}}%
}}}}

\put(4100,-6750){\makebox(0,0)[lb]{\smash{{\SetFigFont{10}{8}{\rmdefault}{\mddefault}{\updefault}{\color[rgb]{0,0,0}$\lambda_{1}^{2}(\tau^{2})$}%
}}}}
\put(5800,-6950){\makebox(0,0)[lb]{\smash{{\SetFigFont{10}{8}{\rmdefault}{\mddefault}{\updefault}{\color[rgb]{0,0,0}$\lambda_{2}^{2}(\tau^{2})$}%
}}}}
\put(6800,-5300){\makebox(0,0)[lb]{\smash{{\SetFigFont{10}{8}{\rmdefault}{\mddefault}{\updefault}{\color[rgb]{0,0,0}$\lambda_3^{2}(\tau^{2})$}
}}}}
\put(5550,-4650){\makebox(0,0)[lb]{\smash{{\SetFigFont{10}{8}{\rmdefault}{\mddefault}{\updefault}{\color[rgb]{0,0,0}{\tiny $\lambda_{4}^{2}(\tau^{2})=\lambda_{2}(\tau^{1})$}}%
}}}}

\put(5700,-3120){\makebox(0,0)[lb]{\smash{{\SetFigFont{10}{8}{\rmdefault}{\mddefault}{\updefault}{\color[rgb]{0,0,0}$\lambda_{4}^{1}(\tau^{1})$}%
}}}}
\put(6800,-3600){\makebox(0,0)[lb]{\smash{{\SetFigFont{10}{8}{\rmdefault}{\mddefault}{\updefault}{\color[rgb]{0,0,0}$\lambda_{3}^{1}(\tau^{1})-\lambda_1^{1}(\tau^{1})$}%
}}}}

\end{picture}%
\caption{The metrics $g_{\boot}^{n+1}(\delta)_{\Lambda^{1}, \bar{l}^{1}}(\tau^{1})$ (left) and $g_{\boot}^{n+1}(\delta)_{(\Lambda^{1}, \Lambda^{2}), (\bar{l}^{1}, \bar{l}^{2})}(\tau^{1}, \tau^{2})$ (right)}
\label{bootdetail3}
\end{figure}

Continuing in this way some $m$ number of times we obtain a psc-metric denoted: $$g_{\boot}^{n+1}(\delta)_{(\Lambda^{1},\cdots \Lambda^{m}), (\bar{l}^{1},\cdots, \bar{l}^{m})}(\tau^{1},  \cdots, \tau^{m}).$$ 
This is a psc-metric on the $D^{n}\times I$ which is depicted below in Fig. \ref{bootdetail4}. Note that to ensure smooth transition at each stage, we insist that:
$$ \lambda_4^{(j)}(\tau^{(j)})=\lambda_{2}^{(j-1)}(\tau^{(j-1)}),$$
for all $j\in\{2, \cdots, m\}$.
Such a metric is known as a {\em $\delta-m$-step metric} or a more generally a $\delta$-step metric (ignoring the specific number of iterations in its construction). We denote by $L_{3}$ the full vertical height of this metric and note that it is given by the formula:
$$L_{3}=\lambda_{1}^{m}(\tau^{m})+\sum_{j=1}^{m-1}\lambda_{3}^{j}(\tau^{j})-\lambda_1^{j}(\tau^{j}).$$
For consistency, we will also use the notation $L_1=\lambda_{1}^{m}(\tau^{m})$, $L_2=\lambda_{2}^{m}(\tau^{m})$ and $L_{4}=\lambda_{4}^{1}(\tau^{1})$ for the height of the ``frontmost toe", the full base length and the top horizontal length. 

We denote by $\Riem_{\delta-\step}^{+}(D^{n}\times I)$, the subspace of $\Riem^{+}(D^{n}\times I)$ consisting of all such metrics over all $m\in\{1,2,\cdots\}$. Note that as any particular iteration may involve only the trivial bend, this space consists of all finite length cylinder metrics. More precisely, this is the space $\Riem_{\delta-\tor\cyl}^{+}(D^{n}\times I)$ defined:
$$ \Riem_{\delta-\tor\cyl}^{+}(D^{n}\times I):=\{g\in \Riem^{+}(D^{n}\times I): g=g_{\tor}^{n}(\delta)_{\lambda}+dt^{2} \text{ on $D^{n}\times [0,L]$, $L>0$, $\lambda\geq 0$}\}.$$ We conclude this section with a useful observation, Lemma \ref{stepprop} below.
\begin{figure}[!htbp]
\vspace{7.5cm}
\hspace{0cm}
\begin{picture}(0,0)
\includegraphics{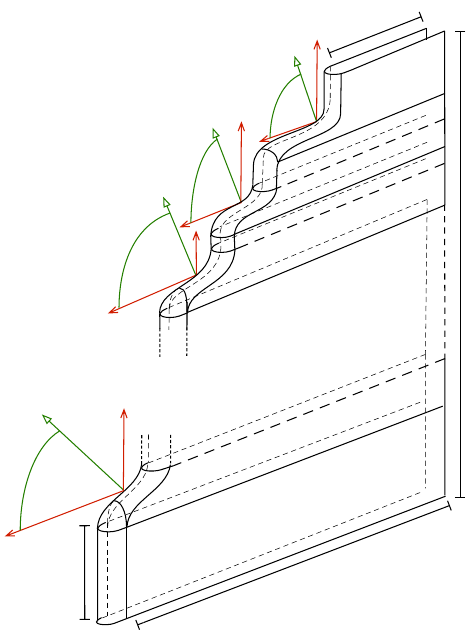}%
\end{picture}
\setlength{\unitlength}{3947sp}%\includegraphics{RelGLpic/
\begingroup\makeatletter\ifx\SetFigFont\undefined%
\gdef\SetFigFont#1#2#3#4#5{%
  \reset@font\fontsize{#1}{#2pt}%
  \fontfamily{#3}\fontseries{#4}\fontshape{#5}%
  \selectfont}%
\fi\endgroup%
\begin{picture}(5079,1559)(1902,-7227)
\put(2880,-4290){\makebox(0,0)[lb]{\smash{{\SetFigFont{10}{8}{\rmdefault}{\mddefault}{\updefault}{\color[rgb]{0,0.6,0}{$\gamma(\tau^{3})$}}%
}}}}
\put(3380,-3800){\makebox(0,0)[lb]{\smash{{\SetFigFont{10}{8}{\rmdefault}{\mddefault}{\updefault}{\color[rgb]{0,0.6,0}{\tiny$\gamma(\tau^{2})$}}%
}}}}
\put(2200,-6000){\makebox(0,0)[lb]{\smash{{\SetFigFont{10}{8}{\rmdefault}{\mddefault}{\updefault}{\color[rgb]{0,0.6,0}$\gamma(\tau^{m})$}%
}}}}
\put(4020,-3120){\makebox(0,0)[lb]{\smash{{\SetFigFont{10}{8}{\rmdefault}{\mddefault}{\updefault}{\color[rgb]{0,0.6,0}{\tiny $\gamma(\tau^{1})$}}%
}}}}

\put(2300,-6750){\makebox(0,0)[lb]{\smash{{\SetFigFont{10}{8}{\rmdefault}{\mddefault}{\updefault}{\color[rgb]{0,0,0}$L_1$}%
}}}}
\put(4200,-6900){\makebox(0,0)[lb]{\smash{{\SetFigFont{10}{8}{\rmdefault}{\mddefault}{\updefault}{\color[rgb]{0,0,0}$L_2$}%
}}}}

\put(5550,-4350){\makebox(0,0)[lb]{\smash{{\SetFigFont{10}{8}{\rmdefault}{\mddefault}{\updefault}{\color[rgb]{0,0,0}$
L_{3}$}
}}}}

\put(4680,-2370){\makebox(0,0)[lb]{\smash{{\SetFigFont{10}{8}{\rmdefault}{\mddefault}{\updefault}{\color[rgb]{0,0,0}$L_4$}%
}}}}
\end{picture}%
\caption{The metric $g_{\boot}^{n+1}(\delta)_{(\Lambda^{1},\cdots \Lambda^{m}), (\bar{l}^{1},\cdots, \bar{l}^{m})}(\tau^{1},  \cdots, \tau^{m})$}
\label{bootdetail4}
\end{figure}

\begin{lemma}\label{stepprop}
Let $n\geq 3$. The inclusion: $$i:\Riem_{\delta-\tor\cyl}^{+}(D^{n}\times I)\hookrightarrow \Riem_{\delta-\step}^{+}(D^{n}\times I)$$ is a weak homotopy equivalence.
\end{lemma}
\begin{proof}
Suppose $K\rightarrow \Riem_{\delta-\step}^{+}(D^{n}\times I), k\mapsto g_k$ is a compact family of step metrics. 
Given any such metric $g_k$ with straight edge lengths $L_1(k), L_2(k), L_3(k)$ and $L_4(k)$, there is an obvious isotopy which moves this metric to the cylinder $g_{\tor}^{n}(\delta)_{L_4(k)}+dt^{2}$ along the interval $[0,L_3(k)]$. Starting on the long end (with neck length $L_2(k)$) we utilise Lemma \ref{bootisot} to unbend back to the height of the next cylinder piece immediately above. Continuing in this way $m$ times, allows us to move $g_k$ through psc-metrics to the desired cylinder metric; see Fig. \ref{partialstepmetric}.

\begin{figure}[!htbp]
\vspace{4cm}
\hspace{-6.5cm}
\begin{picture}(0,0)%
\includegraphics{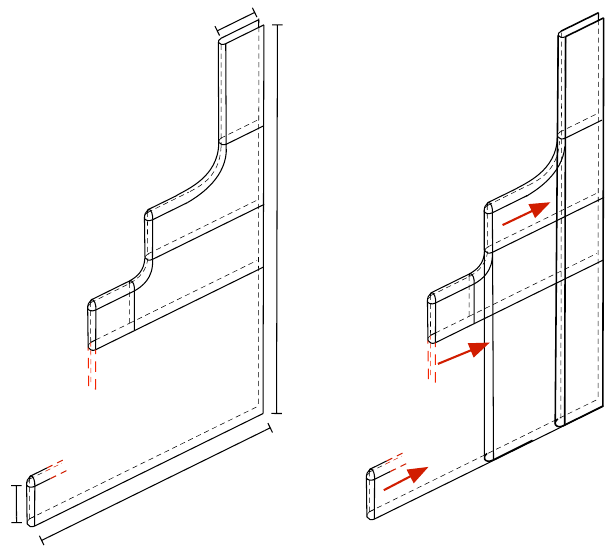}%
\end{picture}%
\setlength{\unitlength}{3947sp}%
\begingroup\makeatletter\ifx\SetFigFont\undefined%
\gdef\SetFigFont#1#2#3#4#5{%
  \reset@font\fontsize{#1}{#2pt}%
  \fontfamily{#3}\fontseries{#4}\fontshape{#5}%
  \selectfont}%
\fi\endgroup%
\begin{picture}(2079,2559)(1902,-5227)
\put(1550,-4850){\makebox(0,0)[lb]{\smash{{\SetFigFont{10}{8}{\rmdefault}{\mddefault}{\updefault}{\color[rgb]{0,0,0}$
L_{1}(k)$}
}}}}

\put(3150,-4800){\makebox(0,0)[lb]{\smash{{\SetFigFont{10}{8}{\rmdefault}{\mddefault}{\updefault}{\color[rgb]{0,0,0}$
L_{2}(k)$}
}}}}

\put(4200,-2600){\makebox(0,0)[lb]{\smash{{\SetFigFont{10}{8}{\rmdefault}{\mddefault}{\updefault}{\color[rgb]{0,0,0}$L_3(k)$}%
}}}}

\put(3600,-800){\makebox(0,0)[lb]{\smash{{\SetFigFont{10}{8}{\rmdefault}{\mddefault}{\updefault}{\color[rgb]{0,0,0}$L_4(k)$}%
}}}}

\end{picture}%
\caption{The step metric $g_k$ (left) and the isotopy back to the cylinder metric (right)} 
\label{partialstepmetric}
\end{figure}

By compactness, we can regard all metrics in this family as $\delta-m$-step metrics for some fixed $m$ (allowing for trivial bends). Applying the isotopy to the entire family above induces a homotopy on the map $K\rightarrow \Riem_{\delta-\step}^{+}(D^{n}\times I)$ to a map with image in $\Riem_{\tor\cyl}^{+}(D^{n-1}\times I)$. Importantly, at each stage in this homotopy, cylinder metrics are never bent but are at worst shrunk down within the space of cylinder metrics.
\end{proof}

\section{The Proofs of Theorems A and B}\label{ThmAproof}
We begin by recalling some notation and technical conditions concerning the geometric set-up for Theorem A.
Let $W$ and $X$ be as described in the introduction. Thus, $W$ is a smooth compact manifold with boundary $\partial W=X$, a closed manifold. Throughout, we assume that $W$ has dimension $n+1$ and that $n\geq 3$. Furthermore, there is a collar ${c}:X\times [0,2)\rightarrow W$ around $X$. Recall from the introduction that the space $\Riem^{+}(W, \p W)$ is the space of all psc-metrics on $W$ whose restriction to $c(X\times I)$ pulls back to a product metric on $X\times I$. Furthermore for each $g\in \Riem^{+}(X)$, $\Riem^{+}(W, \p W)_g$ denotes the space of all psc-metrics on $W$ which extend $g$. This latter space may be empty.

An important special case is that of a smooth compact $(n+1)$-dimensional manifold $Z$ with boundary $\partial Z$, a disjoint union of closed manifolds $X_0$ and $X_1$. In this case we specify a pair of disjoint collars, $c_i:X_i\times [0,2)\rightarrow Z$ around $X_i$ for each of $i=0,1$. Finally, $\Riem^{+}(Z, \p Z)_{g_0, g_1}$ is the space of psc-metrics $g$ on $W$ so that $c_{i}^{*}g_{i}=g_{i}+dt^{2}$ when restricted to $X_{i}\times I$ for each of $i=0,1$. In Theorem A, we will be interested in gluing elements $h\in\Riem^{+}(Z, \p Z)_{g_0, g_1}$ to elements $\bar{g}\in\Riem^{+}(W, \p W=X_0)_{g}$. If the metric $g=g_0$, this is easy. If not, but if $g$ and $g_0$ are concordant, then an intermediary concordance can be used. As part of the preparation for this, we will need to be able to make the following technical adjustment.  

\subsection{Adjustments on the collar}\label{coladj}
Suppose $g_0, g_1\in\Riem^{+}(X)$ are a pair of concordant psc-metrics. Thus, there is a psc-metric $\bar{g}_{\con}$ on the cylinder $X\times [0, L+2]$, for some $L>0$, satisfying:
$$\bar{g}_{\con} = g_1+dt^{2} \text{ on $X\times [0,1]$ and } \bar{g}_{\con}=g_0+dt^{2} \text{ on $X\times [L+1,L+2]$}.$$
(Note here the slightly unorthodox direction of this concordance.)
We would like to specify a map $\Riem^{+}(W, \p W)_{g_0}\rightarrow \Riem^{+}(W, \p W)_{g_1}$ which, roughly speaking, sends each metric $h\in \Riem^{+}(W, \p W)_{g_0}$ to the metric $h\cup \bar{g}_{\con}$ obtained by the obvious gluing. There is a slight problem. Technically, the metric resulting from this attachment is a metric on the manifold $W\cup (X\times [0, L+2])$, not $W$. We get around this problem as follows. We attach to $W$, the cylinder $X\times [0,L+2]$ by identifying $X\times \{L+2\}$ with $\partial W=X$ in the obvious way. We then specify a diffeomorphism $F:W\longrightarrow W\cup (X\times [0,L+2])$ which satisfies the following.
\begin{enumerate}
\item[(i.)] The restriction $F|_{W\setminus c(X\times [0,2))}$ is the identity map from $W\setminus c(X\times [0,2))$ to the space $[W\cup (X\times [0,L+2])]\setminus [c(X\times [0,2))\cup (X\times [0,L+2])]$.
\item[(ii.)] The composition $F\circ c|_{X\times [0,1]}$ is the identity map from $X\times [0,1]$ to $X\times [0,1]\subset W\cup (X\times [0,L+2])$.
\end{enumerate}
In particular, this means that $F^{-1}( X\times [1,L+2]\cup c(X\times [0,2)))=c(X\times [1,2))$ as depicted in Fig. \ref{colstretch}. 
\begin{figure}[!htbp]
\vspace{-1cm}
\hspace{-2cm}
\begin{picture}(0,0)%
\includegraphics{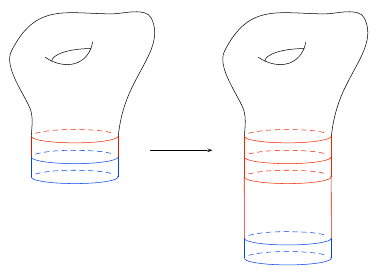}%
\end{picture}%
\setlength{\unitlength}{3947sp}%
\begingroup\makeatletter\ifx\SetFigFont\undefined%
\gdef\SetFigFont#1#2#3#4#5{%
  \reset@font\fontsize{#1}{#2pt}%
  \fontfamily{#3}\fontseries{#4}\fontshape{#5}%
  \selectfont}%
\fi\endgroup%
\begin{picture}(2079,2559)(1902,-5227)
\put(4600,-4500){\makebox(0,0)[lb]{\smash{{\SetFigFont{10}{8}{\rmdefault}{\mddefault}{\updefault}{\color[rgb]{0,0,0}$\textcolor{red}{c([0,2))\cup(X\times[1,L+2])}$}
}}}}
\put(4600,-5100){\makebox(0,0)[lb]{\smash{{\SetFigFont{10}{8}{\rmdefault}{\mddefault}{\updefault}{\color[rgb]{0,0,0}$\textcolor{blue}{X\times[0,1]}$}
}}}}
\put(1550,-4230){\makebox(0,0)[lb]{\smash{{\SetFigFont{10}{8}{\rmdefault}{\mddefault}{\updefault}{\color[rgb]{0,0,0}$\textcolor{red}{c([1,2))}$}
}}}}
\put(1550,-4430){\makebox(0,0)[lb]{\smash{{\SetFigFont{10}{8}{\rmdefault}{\mddefault}{\updefault}{\color[rgb]{0,0,0}$\textcolor{blue}{c([0,1])}$}
}}}}
\end{picture}%
\caption{The map $F:W\longrightarrow W\cup (X\times [0,L+2])$}
\label{colstretch}
\end{figure} 
\noindent This allows us to define a map: 
\begin{equation*}
\begin{split}
\mu_{\bar{g}_{\con}}:\Riem^{+}(W, \p W)_{g_0}&\longrightarrow \Riem^{+}(W, \p W)_{g_1}\\
h&\longmapsto F^{*}(h\cup \bar{g}_{\con}),
\end{split}
\end{equation*}
where $h\cup \bar{g}_{\con}$ is the psc-metric obtained on $W\cup (X\times [0,L+2])$ by the obvious gluing. The second condition on $F$ above means that the pullback metric satisfies the appropriate collar condition on $c[X\times [0,1]]$, i.e. $$c^{*}F^{*}(h\cup \bar{g}_{\con})|_{X\times [0,1]}=g_{1}+dt^{2}.$$ Thus, the map is well-defined. Over the next two sections we will need to use this ``collar adjustment" construction during the proof of Theorem A. This will involve the construction of maps by adding a fixed concordance of the type $\bar{g}_{\con}$ to psc-metrics $h$ on $W$. To avoid an overwhelming amount of notation we will simply describe the resulting metric as $h\cup \bar{g}_{\con}$, assuming that the necessary adjustments have been taken care of.

%$\Riem^{+}(W, \p W)_{g_0,-}$, $\Riem^{+}(W, \p W)_{-, g_1}$ and $\Riem^{+}(W, \p W)_{g_0, g_1}$ are  the subspaces of $\Riem^{+}(W, \p W)$ consisting of psc-metrics on $W$ extending $g_0\in\Riem^{+}(X_0)$, $g_1\in\Riem^{+}(X_1)$ and $g_0\sqcup g_1\in\Riem^{+}(X_0\sqcup X_1)$ respectively.

\subsection{Extending the Gromov-Lawson construction over the trace of a surgery} We begin with the set-up described in section \ref{GLconstruction}. Recall that $\phi:S^{p}\times D^{q+1}\hookrightarrow X$ is an embedding, where $\dim X=n=p+q+1$ and $q\geq 2$. As this set-up will be used to prove Theorem A, we further assume that $p\geq 2$. Recall, we have a collection of rescaling maps $\sigma_\rho:S^{p}\times D^{q+1}\longrightarrow S^{p}\times D^{q+1}$ defined by $\sigma_{\rho}(x,y)=(x, \rho y)$, where $\rho\in(0,1]$. We then set $\phi_{\rho}:=\phi\circ\sigma_\rho$ and $N_{\rho}:=\phi_{\rho}(S^{p}\times D^{q+1})$, abbreviating $N:=N_{1}$. 

As discussed in section \ref{GLreview}, the Gromov-Lawson Surgery Theorem provides a technique for replacing $g\in\Riem^{+}(X)$ with a new psc-metric $g_{\std}\in\Riem^{+}(X)$. We may assume that the metric $g_{\std}$ satisfies the condition that $\phi_{\frac{1}{2}}^{*}g_{\std}=ds_{p}^{2}+g_{\tor}^{q+1}$ while outside $N$, $g_{\std}=g$. We recall that this metric is now surgery ready. By removing the standard piece $ds_{p}^{2}+g_{\tor}^{q+1}$ on $N_{\frac{1}{2}}$ and attaching the metric $g_{\tor}^{p+1}+ds_{q}^{2}$, we obtain a metric $g'\in \Riem^{+}(X')$, where $X'$ is the manifold obtained from $X$ by surgery on the embedding $\phi$.

Let us now consider the trace $T_{\phi}$ of the surgery on $
\phi$. Recall that this is the smooth manifold with boundary obtained by attaching to $X\times I$, the disk product $D^{p+1}\times D^{q+1}$, via the embedding $\phi$. Thus $\partial T_{\phi}=X\sqcup X'$ as depicted earlier in Fig. \ref{surgerytracedef}. In \cite{Walsh1} we describe in detail a procedure for extending a psc-metric $g$ over $T_{\phi}$ to obtain an element $\bar{g}\in\Riem^{+}(T_{\phi})_{g,g'}$. For details, the reader is referred to Theorem 2.2 in \cite{Walsh1}. Roughly, the metric $\bar{g}$ is constructed as follows.
\begin{enumerate}
\item[1.] Using Lemma \ref{welldefconc}, equip $X\times I$ with the concordance $\bar{g}_{\con}$, arising from the Gromov-Lawson isotopy between $g$ and $g_{\std}$ and described in Theorem \ref{GLcompact}. The metric $\bar{g}_{\con}$ is constructed to be a product $g+dt^{2}$ near $X\times\{1\}$ and $g_{\std}+dt^{2}$ near $X\times\{0\}$. This is depicted in the left image in Fig. \ref{tracemetricold}. 
\item[2.] Attach to $N_{\frac{1}{2}}\times\{0\}\subset X\times\{0\}$ a piece which is almost $(D^{p+1}\times D^{q+1}, g_{\tor}^{p+1}+g_{\tor}^{q+1})$ but which contains an extra smoothing region to avoid corners. This is depicted in the right of Fig. \ref{tracemetricold}. 
\item[3.] Of course, on this extra smoothing region, the metric is not a product. However, in the proof of Theorem 2.2 of \cite{Walsh1}, we show how to adjust the metric on this region so as to obtain one which is a product near the boundary.
\end{enumerate}

\begin{figure}[!htbp]
\vspace{-2cm}
\hspace{-9cm}
\begin{picture}(0,0)%
\includegraphics{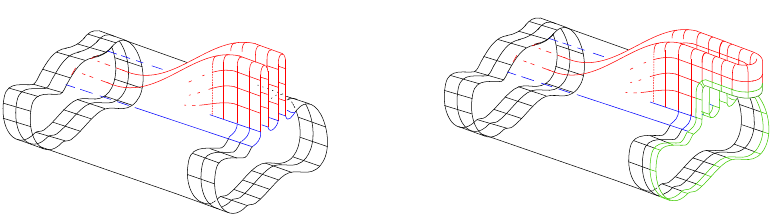}%
\end{picture}%
\setlength{\unitlength}{3947sp}%
\begingroup\makeatletter\ifx\SetFigFont\undefined%
\gdef\SetFigFont#1#2#3#4#5{%
  \reset@font\fontsize{#1}{#2pt}%
  \fontfamily{#3}\fontseries{#4}\fontshape{#5}%
  \selectfont}%
\fi\endgroup%
\begin{picture}(2079,2559)(1902,-5227)
\put(3300,-5300){\makebox(0,0)[lb]{\smash{{\SetFigFont{10}{8}{\rmdefault}{\mddefault}{\updefault}{\color[rgb]{0,0,0}$\textcolor{black}{(X\times\{0\}, g_{\std}+dt^{2})}$}
}}}}
\put(1300,-3600){\makebox(0,0)[lb]{\smash{{\SetFigFont{10}{8}{\rmdefault}{\mddefault}{\updefault}{\color[rgb]{0,0,0}$\textcolor{black}{(X\times\{1\}, g+dt^{2})}$}
}}}}

\end{picture}%
\caption{The concordance $\bar{g}_{\con}$ of $g$ and $g_{\std}$ (left) and the original Gromov-Lawson trace construction before adjustment for product structure near boundary (right)}
\label{tracemetricold}
\end{figure} 

\begin{remark}
In Theorem 2.2 of \cite{Walsh1}, we actually consider the more general case of a {\em Gromov-Lawson cobordism}, which consists of a union of Gromov-Lawson traces determined by an appropriate Morse function. In this paper, we need consider only elementary cobordisms.
\end{remark}

In hindsight, this method can be made a little neater. Before attaching $D^{p+1}\times D^{q+1}$, we make a further preparation. 
The initial concordance, $\bar{g}_{\con}$, ends as a product $g_{\std}+dt^{2}$. Consider for a moment the cylinder $X\times I$ equipped only with this product, $g_{\std}+dt^{2}$. Thus, on $N_{\frac{1}{2}}\times [0,1]$ this metric takes the form $ds_{p}^{2}+g_{\tor}^{q+1}+dt^{2}$. Applying the boot metric isotopy described in Lemma \ref{bootisot} we can replace the cylinder metric $(N_{\frac{1}{2}}\times [0,1], ds_{p}^{2}+g_{\tor}^{q+1}+dt^{2} )$ with a product of boot metrics $(N_{\frac{1}{2}}\times [0,1], ds_{p}^{2}+g_{\boot}^{q+2}(\delta)_{\Lambda, \bar{l}})$, for some admissible psc-boot triple $(\delta, \Lambda, \bar{l})$. Moreover this can be done without altering the metric near $\p N_{\frac{1}{2}}\times I$ so as to ensure the metric smoothly transitions as $g_{\std}+dt^{2}$ over the remaining $(X\setminus N_{\frac{1}{2}})\times I$. To ensure compatibility, we further specify that $\delta=1$ and that the vector $\bar{l}=(l_1, l_2, l_3, l_4)$ controlling the various straight edge lengths of the boot satisfy: $l_1=l_4=1$. We have less control over $\Lambda$ and the remaining edge lengths $l_2$ and $l_3$, but that does not matter. We simply choose quantities which work. 

The resulting metric on $X\times I$, which takes the form $ds_{p}^{2}+g_{\boot}^{q+2}(1)_{\Lambda, \bar{l}}$ with $l_1=l_4=1$ on $N_{\frac{1}{2}}\times [0,1]$, is denoted $\bar{g}_{\std\boot}$. Using the usual identification, it is isometric to a metric on $X\times [0,L+2]$ which takes the form: $$ g_{\std}+dt^{2} \text{ on } X\times [L+1, L+2] \text{ and } ds_{p}^{2}+{g}_{\tor}^{q+1}(1)_{l_2}+dt^{2}\text{ on $N_{\frac{1}{2}}\times[0,1]$}.$$ Importantly, this metric has a product of ``toes": $ds_{p}^{2}+\hat{g}_{\tor}^{q+2}(1)_{1, l_2}$ on a sub-neighbourhood of $N_{\frac{1}{2}}\times[0,1]$. This sub-neighbourhood, which is diffeomorphic to $S^{p}\times D^{q+2}_{\stret}$ is denoted $\bar{N}_{\toe}$. Finally, we attach this metric to the earlier concordance, $\bar{g}_{\con}$ of $g$ and $g_{\std}$, in the obvious way. Essentially, we ``put boots on" the original concordance. The resulting metric is denoted $\bar{g}_{\pre}$ and is depicted in the left image of Fig. \ref{bettertracemetric}. Thus, $$\bar{g}_{\pre}:=\bar{g}_{\con}\cup\bar{g}_{\std\boot}.$$ Performing surgery whilst preserving the product structure on the newly constructed metric, $\bar{g}_{\pre}$, is now trivial. We simply remove $(\bar{N}_{\toe}, ds_{p}^{2}+\hat{g}_{\tor}^{q+2}(1)_{1, l_2})$ and attach the a metric $(\bar{N}_{\toe}', g_{\tor}^{p+1}\times g_{\tor}^{q+1})$. Here $\bar{N}_{\toe}'\cong D^{p+1}\times D^{q+1}$. The resulting psc-metric, denoted $\bar{g}$, is called a {\em Gromov-Lawson trace} and is depicted in the lower right of Fig. \ref{bettertracemetric}. 
\begin{figure}[!htbp]
\vspace{-1cm}
\hspace{-9.5cm}
\begin{picture}(0,0)%
\includegraphics{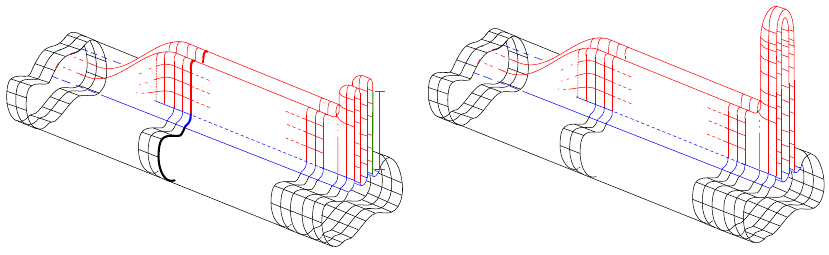}%
\end{picture}%
\setlength{\unitlength}{3947sp}%
\begingroup\makeatletter\ifx\SetFigFont\undefined%
\gdef\SetFigFont#1#2#3#4#5{%
  \reset@font\fontsize{#1}{#2pt}%
  \fontfamily{#3}\fontseries{#4}\fontshape{#5}%
  \selectfont}%
\fi\endgroup%
\begin{picture}(2079,2559)(1902,-5227)
\put(5000,-4250){\makebox(0,0)[lb]{\smash{{\SetFigFont{10}{8}{\rmdefault}{\mddefault}{\updefault}{\color[rgb]{0,0,0}$\textcolor{green}{l_2}$}
}}}}
\put(3300,-4900){\makebox(0,0)[lb]{\smash{{\SetFigFont{10}{8}{\rmdefault}{\mddefault}{\updefault}{\color[rgb]{0,0,0}$\textcolor{black}{\bar{g}_{\std\boot}}$}
}}}}
\put(2300,-4500){\makebox(0,0)[lb]{\smash{{\SetFigFont{10}{8}{\rmdefault}{\mddefault}{\updefault}{\color[rgb]{0,0,0}$\textcolor{black}{\bar{g}_{\con}}$}
}}}}
\end{picture}%
\caption{The metric $\bar{g}_{\pre}$ (left) and the neater Gromov-Lawson trace, $\bar{g}$ (right) }
\label{bettertracemetric}
\end{figure} 

Returning to the space $\Riem^{+}(W, \p W)_{g}$, it will be useful for us to think of the first part of this process as a map:
\begin{equation*}
\begin{split}
\mu_{\bar{g}_{\pre}}:\Riem^{+}(W, \p W)_{g}&\longrightarrow\Riem^{+}(W, \p W)_{g_{\std}(l_2)}\\
h&\longmapsto h\cup \bar{g}_{{\pre}},
\end{split}
\end{equation*}
where $g_{\std}(l_2)=ds_{p}^{2}+{g}_{\tor}^{q+1}(1)_{l_2}$. Here we make use of the collar adjustment technique described in section \ref{coladj}. We note that the addition of boots may stretch the torpedo neck length of the standard metric on the boundary. Hence we replace $g_{\std}$ with $g_{\std}(l_2)$. 
The image of this map, $\image{(\mu_{\bar{g}_{\pre}})}$ will be denoted $\Riem_{\boot}^{+}(W, \p W)_{g_{\std}(l_2)}$.

Before moving on to the proof of Theorem A, there is slight variation on the above construction that is worth considering.
Firstly, it will simplify matters to make available the following notation. Let us denote by $\Riem^{+}(W, \p W)_{g_{\std}(-)}$, the union of spaces $\Riem^{+}(W, \p W)_{g_{\std}(l)}$ over all $l>0$. Thus,
$$\Riem^{+}(W, \p W)_{g_{\std}(-)}:=\bigcup_{l>0} \Riem^{+}(W, \p W)_{g_{\std}(l)}.$$
Similarly, we define: $$\Riem_{\boot}^{+}(W, \p W)_{g_{\std}(-)}:=\bigcup_{l_2}\Riem_{\boot}^{+}(W, \p W)_{g_{\std}(l_2)},$$ over all {\em admissible} $l_2>0$ (i.e. $l_2$ such that the vector $\bar{l}$ forms part of a psc-boot triple $(1, \Lambda, \bar{l})$). 
\begin{proposition}\label{bootnecklengths} For any $l>0$ and any admissible $l_2>0$, the inclusions: $$\Riem^{+}(W, \p W)_{g_{\std}(l)}\subset \Riem^{+}(W, \p W)_{g_{\std}(-)}\text{ and }\Riem_{\boot}^{+}(W, \p W)_{g_{\std}(l_2)}\subset \Riem_{\boot}^{+}(W, \p W)_{g_{\std}(-)}$$ are weak homotopy equivalences.
\end{proposition}
\begin{proof}
This is a straightforward application of the techniques of Lemma \ref{bootisot}.
\end{proof}

Consider once again the construction of the metric $\bar{g}_{\std\boot}.$ Instead of replacing the cylindrical piece $(N_{\frac{1}{2}}\times[0,1], g_{\std}+dt^{2})$ with $(N_{\frac{1}{2}}\times[0,1], ds_{p}^{2}+g_{\boot}^{q+2}(1)_{\Lambda, \bar{l}})$ where $l_1=l_4=1$, we instead replace it with a metric $(N_{\frac{1}{2}}\times[0,1], ds_{p}^{2}+g_{\step}^{q+2}(1)(\bar{L}))$. Here the metric $g_{\step}^{q+2}(1)(\bar{L})$ is an element of the space of $1-\step$-metrics $\Riem_{1-\step}^{+}(D^{q+1}\times I)$ with straight edge pieces specified by the vector $\bar{L}=(L_1, L_2, L_3, L_4)$ as in section \ref{steps}. Note that we insist that $\bar{L}$ satisfies $L_1=L_4=1$ to ensure smooth attachment. There are of course a multitude of metrics in $\Riem_{1-\step}^{+}(D^{q+1}\times I)$ which satisfy this condition. (Note also that this space contains as a subspace all of the boot metrics of radius $1$.) We now consider the space of all metrics on $X\times I$ obtainable from such an attachment. This space, denoted $\Riem_{\step}^{+}(X\times I)_{{g_\std}(-)}$ is defined:
\begin{equation*}
\begin{split}
\Riem_{\step}^{+}(X\times I)_{{g_\std}(-)}:=\{&\bar{g}_{\step}\in \Riem^{+}(X\times I)_{g_{\std}(-)}:\bar{g}_{\step}|_{N_{\frac{1}{2}}}=ds_{p}^{2}+g_{\step}^{q+2}(1)(\bar{L}),\\ 
&g_{\step}^{q+2}(1)(\bar{L})\in \Riem_{1-\step}^{+}(D^{q+1}\times I)\text{ and } L_1=L_4=1\}. 
\end{split}
\end{equation*}

\noindent Finally, we denote by $\Riem_{\step}^{+}(X\times I)_{g_\std}(L_2)$, for any $L_2$, the subspace of $\Riem_{\step}^{+}(X\times I)_{{g_\std}(-)}$ consisting of metrics which take the form $g_{\std}(L_2)$ on the boundary. Note that once we restrict the torpedo neck length of the the boundary metric to some $L_2$, we limit the complexity of the corresponding step metric. That said, as the space $\Riem_{\step}^{+}(X\times I)_{g_\std}(-)$ includes even the trivial step metrics (the cylinders $g_{\std}(L_2)+dt^{2}$ for all $L_2>0$), the space $\Riem_{\step}^{+}(X\times I)_{g_\std}(L_2)$ is always non-empty.
\begin{proposition}\label{stepspecific}
For any $L_2>0$, the inclusion: $$\Riem_{\step}^{+}(X\times I)_{g_\std}(L_2)\subset \Riem_{\step}^{+}(X\times I)_{g_\std}(-)$$ is a weak homotopy equivalence.
\end{proposition}
\begin{proof}
This is a simple application of Lemma \ref{stepprop}.
\end{proof}

Let us return again to the space $\Riem^{+}(W, \p W)_{g}$. We now adjust the map $\mu_{\bar{g}_{\pre}}$ as follows. Instead of attaching to an element $h\in \Riem^{+}(W, \p W)_{g}$, the concordance $\bar{g}_{\pre}=\bar{g}_{\con}\cup\bar{g}_{\std\boot}$, we instead attach $\bar{g}_{\con}\cup \bar{g}_{\step}$ for some $\bar{g}_{\step}\in \Riem_{\step}^{+}(X\times I)_{g_\std}(-)$. We denote such a map $\mu_{\bar{g}_{\step}}$. To be clear, this map takes the form:
\begin{equation*}
\begin{split}
\mu_{\bar{g}_{\step}}:\Riem^{+}(W, \p W)_{g}&\longrightarrow\Riem^{+}(W, \p W)_{g_{\std}(-)}\\
h&\longmapsto h\cup \bar{g}_{\con}\cup \bar{g}_{\step},
\end{split}
\end{equation*}
Obviously, there many such maps, one for each $\bar{g}_{\step}\in \Riem_{\step}^{+}(X\times I)_{{g_\std}(-)}$.
We now denote by $\Riem_{\step}^{+}(W, \p W)_{{g_\std}(L_2)}$ and $\Riem_{\step}^{+}(W, \p W)_{{g_\std}(-)}$ the space defined:
$$ \Riem_{\step}^{+}(W, \p W)_{{g_{\std}}(L_2)}:=\bigcup_{\bar{g}_{\step}}\image(\mu_{\bar{g}_{\step}}) \text{ where $\bar{g}_{\step}\in \Riem_{\step}^{+}(X\times I)_{{g_{\std}}(L_2)}$}$$
and
$$ \Riem_{\step}^{+}(W, \p W)_{{g_{\std}}(-)}:=\bigcup_{\bar{g}_{\step}}\image(\mu_{\bar{g}_{\step}}) \text{ where $\bar{g}_{\step}\in \Riem_{\step}^{+}(X\times I)_{g_\std}(-)$.} $$ 
To aid the reader, we provide schematic images of elements from the spaces $\Riem_{\boot}^{+}(W, \p W)_{{g_{\std}}(-)}$, $\Riem_{\step}^{+}(W, \p W)_{{g_{\std}}(-)}$ and $\Riem^{+}(W, \p W)_{{g_{\std}}(-)}$ in Fig. \ref{bootstepping}

\begin{figure}[htb!]
\vspace{4cm}
\hspace{-11cm}
\begin{picture}(0,0)%
\includegraphics{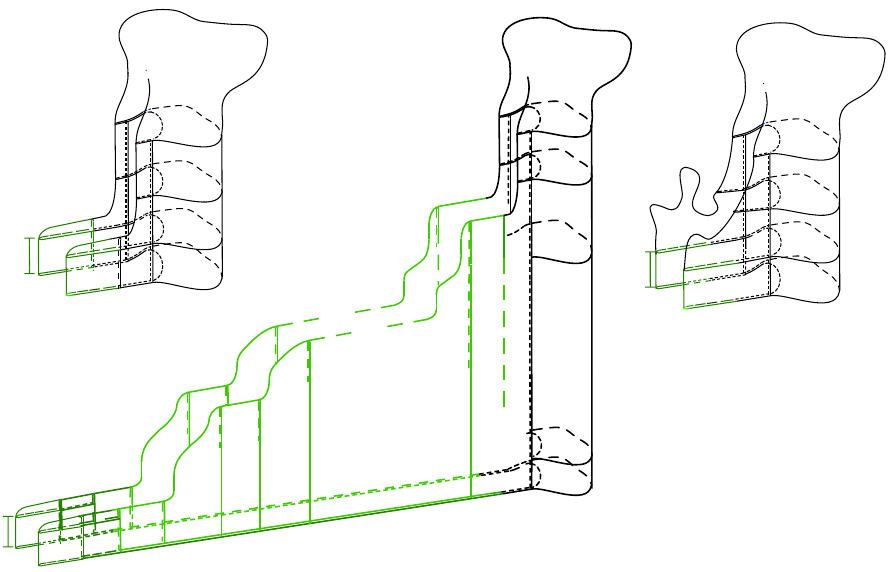}%
\end{picture}%
\setlength{\unitlength}{3947sp}%
\begingroup\makeatletter\ifx\SetFigFont\undefined%
\gdef\SetFigFont#1#2#3#4#5{%
  \reset@font\fontsize{#1}{#2pt}%
  \fontfamily{#3}\fontseries{#4}\fontshape{#5}%
  \selectfont}%
\fi\endgroup%
\begin{picture}(2079,2559)(1902,-5227)
\put(2030,-2750){\makebox(0,0)[lb]{\smash{{\SetFigFont{10}{8}{\rmdefault}{\mddefault}{\updefault}{\color[rgb]{0,0.6,0}${1}$}
}}}}
\put(1850,-4940){\makebox(0,0)[lb]{\smash{{\SetFigFont{10}{8}{\rmdefault}{\mddefault}{\updefault}{\color[rgb]{0,0.6,0}${1}$}
}}}}
\put(7000,-2850){\makebox(0,0)[lb]{\smash{{\SetFigFont{10}{8}{\rmdefault}{\mddefault}{\updefault}{\color[rgb]{0,0.6,0}${1}$}
}}}}
\end{picture}%
\caption{ Elements of the spaces $\Riem_{\boot}^{+}(W, \p W)_{{g_{\std}}(-)}$ (left), $\Riem_{\step}^{+}(W, \p W)_{{g_{\std}}(-)}$ (middle) and $\Riem^{+}(W, \p W)_{{g_{\std}}(-)}$ (right)}
\label{bootstepping}
\end{figure}

We now have the following lemma. 
\begin{lemma}\label{bootstepwhe}
Consider, for some admissible $l_2>0$, the following commutative diagram of inclusions.
$$\label{diag2}
\xymatrix{
&\Riem_{\boot}^{+}(W, \p W)_{{g_{\std}}(l_2)} \ar@{^{(}->}[d]^{} \ar@{^{(}->}[r]^{}&  \Riem_{\step}^{+}(W, \p W)_{{g_{\std}}(l_2)}     \ar@{^{(}->}[r]^{}\ar@{^{(}->}[d]^{} &  \Riem^{+}(W, \p W)_{{g_{\std}}(l_2)} \ar@{^{(}->}[d]^{} \\
& \Riem_{\boot}^{+}(W, \p W)_{{g_{\std}}(-)} \ar@{^{(}->}[r]^{} &    \Riem_{\step}^{+}(W, \p W)_{{g_{\std}}(-)} \ar@{^{(}->}[r]^{}&\Riem^{+}(W, \p W)_{{g_{\std}}(-)}
}
$$
Every map in this diagram is a weak homotopy equivalence.
\end{lemma}
\begin{proof}
We have shown already in Propositions \ref{bootnecklengths} and \ref{stepspecific} that the vertical inclusions are weak homotopy equivalences. It suffices to show the same for the  horizontal maps on the bottom row. That the bottom left horizontal inclusion is a weak homotopy equivalence follows immediately from Lemma \ref{stepprop}. In the case of the bottom right inclusion, let $K\rightarrow \Riem^{+}(W, \p W)_{{g_{\std}}(-)}$ be a compact family of psc-metrics. By stretching torpedo necks if necessary, we can assume that the image of $K$ lies in $\Riem^{+}(W, \p W)_{{g_{\std}}(l_4)}$ for some $l_4>0$. Applying the isotopy from Lemma \ref{bootisot} to the cylindrical part $g_{\std}(l)+dt^{2}$ moves this family continuously into $\Riem_{\step}^{+}(W, \p W)_{{g_{\std}}(l_2)}\subset \Riem_{\step}^{+}(W, \p W)_{{g_{\std}}(-)}$ for some $l_2>0$. Moreover, any metric which initially lies in $\Riem_{\step}^{+}(W, \p W)_{{g_{\std}}(-)}$ remains there as this isotopy has at worst the effect of adding another ``step". 
\end{proof}

\subsection{The first claim of Theorem A}
We are now in a position to introduce the objects described in the statement of Theorem A.  We begin by specifying the manifold $W'=W\cup T_{\phi}$ obtained by gluing the trace $T_{\phi}$ above to $W$ in the obvious way. The manifold $W'$ has boundary $\p W'=X'$. We next specify a collar $c':X'\times [0, 2)\rh T_\phi\subset W'$ around this boundary in such a way that the metric $(c')^{*}\bar{g}$ restricts on $X'\times I$ as the product metric $g'+dt^{2}$. 
This takes care of the first claim of Theorem A, and allows us to define the map: 
\begin{equation*}
\begin{split}
\mu_{T_{\phi}, \bar{g}}:\Riem^{+}(W, \p W)_{g}&\longrightarrow\Riem^{+}(W', \p W')_{g'}\\
h&\longmapsto h\cup \bar{g},
\end{split}
\end{equation*}
where $W'=W\cup T_{\phi}$ and $h\cup \bar{g}$ are the manifold and metric obtained by the obvious gluing. Our main challenge is still ahead of us. We must show that this map is a weak homotopy equivalence.
 
\subsection{The strategy for proving Theorem A}
The strategy for proving Theroem A is quite similar to that of the proof of Theorem \ref{Chernysh}. The spaces $\Riem^{+}(W, \p W)_{g}$ and $\Riem^{+}(W', \p W')_{g'}$ are complicated objects. We wish to replace these spaces with simpler, but weakly homotopy equivalent spaces which are easily seen to be weakly homotopy equivalent to each other. Before discussing this further, recall we remarked in the introduction that we must deal with the possibilty that the space $\Riem^{+}(W, \p W)_{g}$ may be empty. For now, we will assume that $\Riem^{+}(W, \p W)_{g}$ (and hence $\Riem^{+}(W', \p W')_{g'}$) is a non-empty space. We will call this the non-empty case of the theorem. Later we will prove that $\Riem^{+}(W, \p W)_{g}\neq \emptyset$ if and only if $\Riem^{+}(W', \p W')_{g'}\neq \emptyset$.  

Notice that, from above, the map $\mu_{T_{\phi}, \bar{g}}$ decomposes as a composition of maps: 
\begin{equation*}
\begin{split}
\mu_{T_{\phi}, \bar{g}}:\Riem^{+}(W, \p W)_{g}&\longrightarrow\Riem^{+}(W, \p W)_{g_{\std}(l_2)}\longrightarrow\Riem^{+}(W', \p W')_{g'}\\
h&\longmapsto h\cup \bar{g}_{{\pre}}\longmapsto h\cup \bar{g},
\end{split}
\end{equation*}
where $g_{\std}(l_2)=ds_{p}^{2}+{g}_{\tor}^{q+1}(1)_{l_2}$.
Earlier we denoted the first of these maps by $\mu_{\bar{g}_{\pre}}$.  We denote the second by $\mu_{\bar{g}}$. In the case of $\mu_{\bar{g}_{\pre}}$, we make use of the collar adjustment construction described in section \ref{coladj} but suppress the notation. We obtain the following commutative diagram.
$$
\xymatrix{
& \Riem^{+}(W, \p W)_{g}\ar@{->}[d]^{\mu_{\bar{g}_{\pre}}}\ar@{->}[r]^{\mu_{(T_{\phi},\bar{g})}}& \Riem^{+}(W', \p W')_{g'}\\
& \Riem^{+}(W, \p W)_{g_{\std}(l_2)}\ar@{->}[ru]^{\mu_{\bar{g}}} }
$$
Recall that $\Riem_{\boot}^{+}(W, \p W)_{g_{\std}(l_2)}:=\image{(\mu_{\bar{g}_{\pre}})}$. Setting $\Riem_{\Estd}^{+}(W', \p W')_{g'}:=\image{(\mu_{T_{\phi}, \bar{g}})}$, we reformulate this diagram as:
\begin{equation}\label{bdydiag}
\xymatrix{
& \Riem^{+}(W, \p W)_{g}\ar@{->}[d]^{\mu_{\bar{g}_{\pre}}}\ar@{->}[r]^{\mu_{(T_{\phi},\bar{g})}}& \Riem^{+}(W', \p W')_{g'}\\
& \Riem_{\boot}^{+}(W, \p W)_{g_{\std}(l_2)}\ar@{->}[r]^{\mu_{\bar{g}}} & \Riem_{\Estd}^{+}(W', \p W')_{g'}  \ar@{^{(}->}[u]^{}}
\end{equation}
where the hooked arrow indicates inclusion. The bottom horizontal map $\mu_{\bar{g}}$ is simply the map which removes the standard product of boots $ds_{p}^{2}+\hat{g}_{\tor}^{q+2}(1)_{1, l_2}$ and attaches $g_{\tor}^{p+1}+g_{\tor}^{q+1}$; see Fig. \ref{bootstretch61} for an illustration of what is going on. The space $\Riem_{\Estd}^{+}(W', \p W')_{g'}$ should be thought of as a space of {\em extra-standard} psc-metrics on $W'$ which extend $g'$. The motivation for this terminology will become clear later. The following lemma, an analogue of Proposition \ref{easyhomeoclosed}, is immediate. 
\begin{lemma}\label{homeobd}
The map $\mu_{\bar{g}}:  \Riem_{\boot}^{+}(W, \p W)_{g_{\std}(l_2)}\longrightarrow \Riem_{\Estd}^{+}(W', \p W')_{g'}$ is a homeomorphism.
\end{lemma}

\begin{figure}[htb!]
\vspace{1.2cm}
\hspace{-9cm}
\begin{picture}(0,0)%
\includegraphics{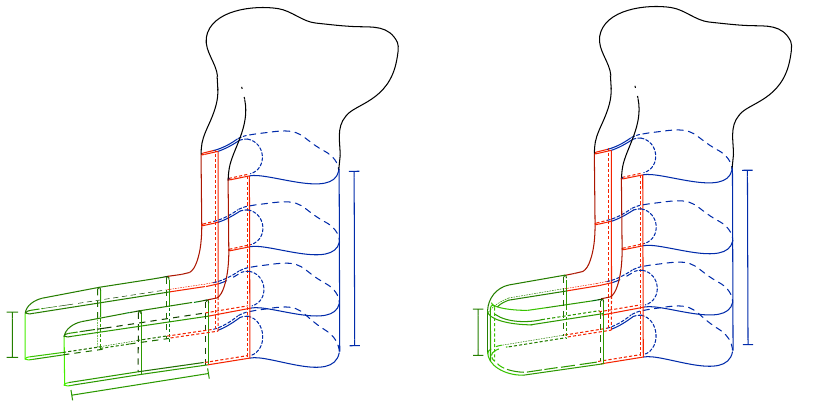}%
\end{picture}%
\setlength{\unitlength}{3947sp}%
\begingroup\makeatletter\ifx\SetFigFont\undefined%
\gdef\SetFigFont#1#2#3#4#5{%
  \reset@font\fontsize{#1}{#2pt}%
  \fontfamily{#3}\fontseries{#4}\fontshape{#5}%
  \selectfont}%
\fi\endgroup%
\begin{picture}(2079,2559)(1902,-5227)
\put(4800,-4150){\makebox(0,0)[lb]{\smash{{\SetFigFont{10}{8}{\rmdefault}{\mddefault}{\updefault}{\color[rgb]{0,0,0}$\textcolor{blue}{l_3}$}
}}}}
\put(7960,-4150){\makebox(0,0)[lb]{\smash{{\SetFigFont{10}{8}{\rmdefault}{\mddefault}{\updefault}{\color[rgb]{0,0,0}$\textcolor{blue}{l_3}$}
}}}}
\put(3000,-5170){\makebox(0,0)[lb]{\smash{{\SetFigFont{10}{8}{\rmdefault}{\mddefault}{\updefault}{\color[rgb]{0,0.6,0}${l_2}$}
}}}}
\put(1900,-4700){\makebox(0,0)[lb]{\smash{{\SetFigFont{10}{8}{\rmdefault}{\mddefault}{\updefault}{\color[rgb]{0,0.6,0}${1}$}
}}}}
\put(5610,-4700){\makebox(0,0)[lb]{\smash{{\SetFigFont{10}{8}{\rmdefault}{\mddefault}{\updefault}{\color[rgb]{0,0.6,0}${1}$}
}}}}
\end{picture}%
\caption{ An element $h\in\Riem_{\boot}^{+}(W, \p W)_{g_{\std}(l_2)}$ (left) and its image $\mu_{\bar{g}}(h)\in \Riem_{\Estd}^{+}(W', \p W')_{g'}$ (right)}
\label{bootstretch61}
\end{figure} 

\noindent To complete the proof of theorem A, we must show that the vertical maps in diagram (\ref{bdydiag}) are weak homotopy equivalences. This is reasonably straightforward for the left inclusion and we begin with that case below. As discussed in the introduction, the case of the right inclusion is more challenging. Our proof is modeled on the proof of the analogous Theorem \ref{Chernysh}. In particular, it will be done in stages involving some intermediary spaces. 

\subsection{The left side of the diagram}\label{leftside}
We start on the left hand side of diagram (\ref{bdydiag}). The map $\mu_{\bar{g}_{\pre}}$ further decomposes into two maps as follows. We denote as usual by $\bar{g}_{\con}\in \Riem^{+}(X\times I)$, the concordance between $g$ and $g_{\std}$, obtained via the Gromov-Lawson isotopy. After making use of the collar adjustment construction of section \ref{coladj}, we obtain a map:
\begin{equation*}
\begin{split}
\mu_{{\bar{g}_{\con}}}:\Riem^{+}(W, \p W)_{g}&\longrightarrow\Riem^{+}(W, \p W)_{g_{\std}}\\
h&\longmapsto h\cup \bar{g}_{\con}.
\end{split}
\end{equation*}
We describe the remaining part of the Gromov-Lawson trace construction, where we add the ``boot concordance" $\bar{g}_{\std\boot}$ as the map:
\begin{equation*}
\begin{split}
\mu_{\bar{g}_{\std\boot}}:\Riem^{+}(W, \p W)_{g_{\std}}&\longrightarrow \Riem_{\boot}^{+}(W, \p W)_{g_{\std}(l_2)}\\
h&\longmapsto h\cup \bar{g}_{\std\boot},
\end{split}
\end{equation*}
once again making use of the collar adjustment from section \ref{coladj}.
Thus, we have that:
$$
\mu_{\bar{g}_{\pre}}=\mu_{\bar{g}_{\std\boot}}\circ\mu_{{\bar{g}_{\con}}}.
$$

Importantly, the concordance $\bar{g}_{\con}$ is slicewise. By this, we mean that it takes the form (or at least pulls back by an obvious rescaling to) $g_{s}+ds^{2}$ on some $X\times [0,\bar{s}]$ where $g_s\in\Riem^{+}(X)$ for all $s\in [0, \bar{s}]$, $g_{0}=g$ and $g_{\bar{s}}=g_{\std}$. (Recall it was constructed directly from an isotopy as an appropriately scaled warped product metric.) This means that there is an obvious homotopy through concordances to the standard cylinder $g+ds^{2}$; see Corollary \ref{conchomot}. In particular, we obtain the following lemma.
\begin{lemma}\label{whe1}
The map $\mu_{\bar{g}_{\con}}$ is a weak homotopy equivalence.
\end{lemma}
\begin{proof}
We define a homotopy inverse to $\mu_{{\bar{g}_{\con}}}$, 
$$\mu_{\bar{g}_{\con}^{-1}}: \Riem^{+}(W, \p W)_{g_{\std}}\longrightarrow\Riem^{+}(W, \p W)_{g},
$$
defined by attaching the concordance $\bar{g}_{\con}$ at the opposite end. By using the slicewise nature of the concordance, it is now easy to construct a homotopy of the maps: 
$\mu_{\bar{g}_{\con}}\circ \mu_{\bar{g}_{\con}^{-1}}$ and $\mu_{\bar{g}_{\con}^{-1}}\circ \mu_{\bar{g}_{\con}}$ to the appropriate identity maps via homotopies which specify isotopies of appropriate parts of the concordance to the cylinder metric, followed by any necessary rescaling. The case of $\mu_{\bar{g}_{\con}^{-1}}\circ \mu_{\bar{g}_{\con}}$ is illustrated in Fig. \ref{conchomo}.
\end{proof}
\begin{figure}[!htbp]
\vspace{2cm}
\hspace{-10cm}
\begin{picture}(0,0)%
\includegraphics{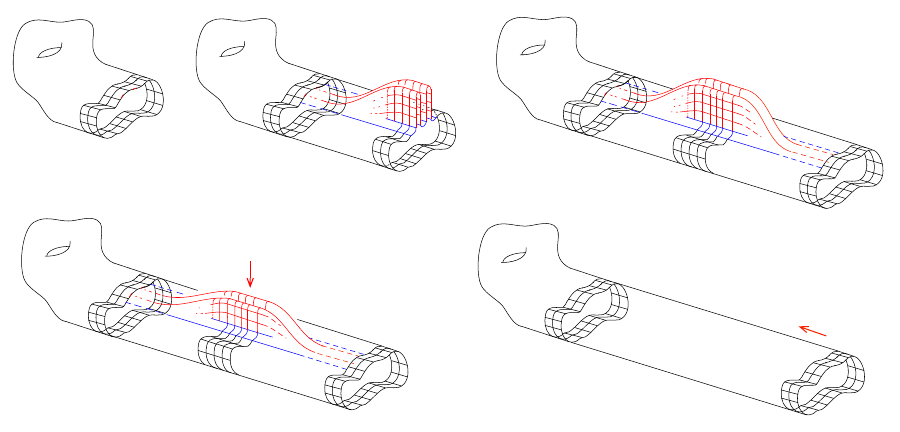}%
\end{picture}%
\setlength{\unitlength}{3947sp}%
\begingroup\makeatletter\ifx\SetFigFont\undefined%
\gdef\SetFigFont#1#2#3#4#5{%
  \reset@font\fontsize{#1}{#2pt}%
  \fontfamily{#3}\fontseries{#4}\fontshape{#5}%
  \selectfont}%
\fi\endgroup%
\begin{picture}(2079,2559)(1902,-5227)
\end{picture}%
\caption{The metric $h$ (top left), followed by the metric $\mu_{\bar{g}_{\con}}(h)$ (top middle), the metric $\mu_{\bar{g}_{\con}^{-1}}\circ \mu_{\bar{g}_{\con}}(h)$ (top right), the isotopy of the concordance back to the cylinder (bottom left) and rescaling (bottom right) back to $h$}
\label{conchomo}
\end{figure} 
\noindent The argument employed to prove Lemma \ref{whe1} works just as well in proving the following more general fact.
\begin{lemma} \label{quasihomo}
Suppose $g_1, g_2\in \Riem^{+}(X)$ are isotopic metrics. The spaces $\Riem^{+}(W, \p W)_{g_1}$ and $\Riem^{+}(W, \p W)_{g_2}$ are weakly homotopy equivalent.
\end{lemma}
\begin{remark}
Lemma \ref{quasihomo} can also be obtained as a corollary to a much stronger result due to Chernysh in \cite{Che2}. We will come back to this result a little later.
\end{remark}
%%% MENTION EBERT_FRENCK IN THE ABOVE REMARK.
\begin{lemma}\label{whe2}
The map $\mu_{\bar{g}_{\std\boot}}$ is a weak homotopy equivalence.
\end{lemma}
\begin{proof}
The proof here is completely analogous to the proof of Lemma \ref{whe1}. We define a homotopy inverse $\mu_{\bar{g}_{\std\boot}^{-1}}$ as before by turning the concordance $\bar{g}_{\std\boot}$ around and gluing at the opposite end. Lemma \ref{bootisot} allows us to construct an isotopy of this concordance back to the standard cylinder and so the proof goes through exactly as in Lemma \ref{whe1}. To aid the reader, we illustrate the case of $\mu_{\bar{g}_{\std\boot}^{-1}}\circ \mu_{\bar{g}_{\std\boot}}$ in Fig. \ref{concboot}.
\end{proof}
\begin{figure}[!htbp]
\vspace{3.5cm}
\hspace{-4cm}
\begin{picture}(0,0)%
\includegraphics{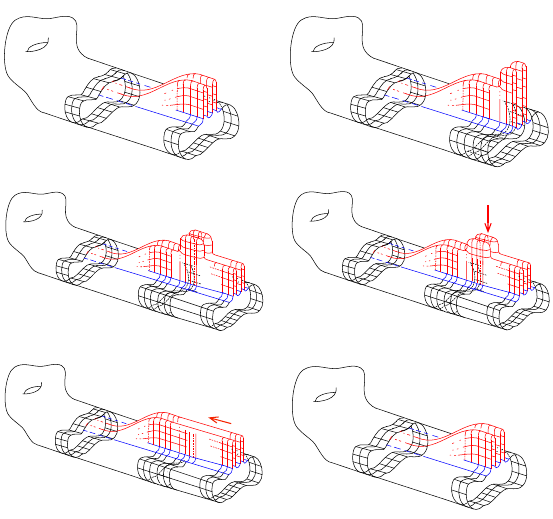}%
\end{picture}%
\setlength{\unitlength}{3947sp}%
\begingroup\makeatletter\ifx\SetFigFont\undefined%
\gdef\SetFigFont#1#2#3#4#5{%
  \reset@font\fontsize{#1}{#2pt}%
  \fontfamily{#3}\fontseries{#4}\fontshape{#5}%
  \selectfont}%
\fi\endgroup%
\begin{picture}(2079,2559)(1902,-5227)
\end{picture}%
\caption{The metric $h$ (top left), followed by the metric $\mu_{\bar{g}_{\std\boot}}(h)$ (top right), the metric  $\mu_{\bar{g}_{\std\boot}^{-1}}\circ \mu_{\bar{g}_{\std\boot}}(h)$ (middle left), the isotopy of the ``double boot" concordance back to the cylinder (middle right) and rescaling (bottom left) back to $h$ (bottom right)}
\label{concboot}
\end{figure} 

\noindent Combining Lemmas \ref{whe1} and \ref{whe2} gives the following.
\begin{lemma}\label{wheleft}
The map $\mu_{\bar{g}_{\pre}}$ is a weak homotopy equivalence.
\end{lemma}

\subsection{The right side of the diagram}
It remains for us to show that the inclusion: 
$$
\Riem_{\Estd}^{+}(W', \p W')_{g'}\hookrightarrow \Riem^{+}(W', \p W')_{g'},
$$
is a weak homotopy equivalence. It will suit us to have some flexibility regarding the neck-length of the torpedo factor of the standard part of the boundary metric. Recall at the end of section \ref{GLreview}, we defined the space $\Riem^{+}(W', \p W')_{g'(-)}$. Essentially this slightly extends $\Riem^{+}(W', \p W')_{g'}$ by allow the torpedo factor of the standard region to have arbitrary neck-length. Importantly, in Lemma \ref{stdincludeggen}, we showed that the inclusion: 
$$ \Riem^{+}(W', \p W')_{g'}\subset\Riem^{+}(W', \p W')_{g'(-)}$$
is a weak homotopy equivalence.
Analogously, we define $\Riem_{\Estd}^{+}(W', \p W')_{g'(\lambda)}$ to be the space obtained by replacing the standard piece $(\bar{N}_{\toe}'\cong D^{p+1}\times D^{q+1}, g_{\tor}^{p+1}+g_{\tor}^{q+1})$ with $(\bar{N}_{\toe}'\cong D^{p+1}\times D^{q+1}, g_{\tor}^{p+1}(1)_{\lambda}+g_{\tor}^{q+1})$. We then define the space $\Riem_{\Estd}^{+}(W', \p W')_{g'(-)}$ as:
$$ \Riem_{\Estd}^{+}(W', \p W')_{g'(-)}:=\bigcup_{\lambda>0}\Riem_{\Estd}^{+}(W', \p W')_{g'(\lambda)}.$$
The following proposition is immediate.
\begin{proposition}\label{Estdgen}
For any $\lambda>0$, the inclusion:
$$\Riem_{\Estd}^{+}(W', \p W')_{g'(\lambda)}\subset \Riem_{\Estd}^{+}(W', \p W')_{g'(-)}$$
is a homotopy equivalence.
\end{proposition}
We consolidate these observations in the following commutative diagram of inclusions.
\begin{equation}\label{genEstd}
\xymatrix{
&\Riem^{+}(W', \p W')_{g'} \ar@{^{(}->}[r]^{}&\Riem^{+}(W', \p W')_{g'(-)} \\
&\Riem_{\Estd}^{+}(W', \p W')_{g'}\ar@{^{(}->}[r]^{}\ar@{^{(}->}[u]^{} &\Riem_{\Estd}^{+}(W', \p W')_{g'(-)}\ar@{^{(}->}[u]^{}}
\end{equation}
From Lemma \ref{stdincludeggen} and Proposition \ref{Estdgen} we know that the horizontal maps above are both weak homotopy equivalences. Thus, to prove that the left vertical map is a weak homotopy equivalence (completing the proof of Theorem A) it suffices to show that the right vertical map is a weak homotopy equivalence. For the remainder of the section we will work on this problem.

We begin by specifying an intermediary space, $\Riem_{\std}^{+}(W', \p W')_{g'(-)}$, satisfying:
$$
\Riem_{\Estd}^{+}(W', \p W')_{g'(-)}\subset \Riem_{\std}^{+}(W', \p W')_{g'(-)}\subset \Riem^{+}(W', \p W')_{g'(-)},
$$
and show that the above inclusions are both weak homotopy equivalences. The space $\Riem_{\std}^{+}(W', \p W')_{g'(-)}$ (in a slightly more general setting) features in Lemma \ref{stdincludeg} and we define it here in essentially the same way. We recall that the smooth manifold $W'$ decomposes as the union:
$$W' :=W\cup T_{\phi}=W\cup (X\times I)\cup_{\phi} D^{p+1}\times D^{q+1},$$
where the latter union is obtained by gluing the partial boundary $S^{p}\times D^{q+1}\subset D^{p+1}\times D^{q+1}$ to $X\times\{1\}$ via the embedding $\phi:S^{p}\times D^{q+1}\hookrightarrow X$. 
We denote by $\bar{\phi}$, the map $\bar{\phi}:D^{p+1}\times D^{q+1}\rightarrow W'$ which extends $\phi$ and is defined by:
\begin{enumerate}
\item[] $ \bar{\phi} (x,y)=(\phi(x,y), 1)\in X\times\{1\}$ when $(x,y)\in S^{p}\times D^{q+1}$,
\item[] $ \bar{\phi} (x,y)=(x,y)$ when $(x,y)\in (D^{p+1}\times D^{q+1})\setminus (S^{p}\times D^{q+1}). $ 
\end{enumerate}
The families of embeddings $\phi_{\rho}$ and $\bar{\phi}_{\rho}$ as well as the spaces $N_{\rho}=\phi_{\rho}(S^{p}\times D^{q+1})$ and $\bar{N}_{\rho}=\bar{\phi}_{\rho}(D^{p+1}\times D^{q+1})$, where $\rho\in(0,1]$ are as in the preamble to Lemma \ref{stdincludeg}. Moreover, we may assume that the collar $c':X'\times [0,2)\rh W'$ has been specified so as to ensure compatibility with $\bar{\phi}$ in the sense of Lemma \ref{stdincludeg}. In particular, this means that the embedded disk $D^{q+1}:=\bar{\phi}(\{0\}\times D^{q+1})$ is particularly nicely embedded with respect to metrics in $\Riem_{\Estd}^{+}(W', \p W')_{g'(-)}$; see Fig. \ref{spherespecfirst}.
\begin{figure}[htb!]
\vspace{1cm}
\hspace{-7cm}
\begin{picture}(0,0)%
\includegraphics{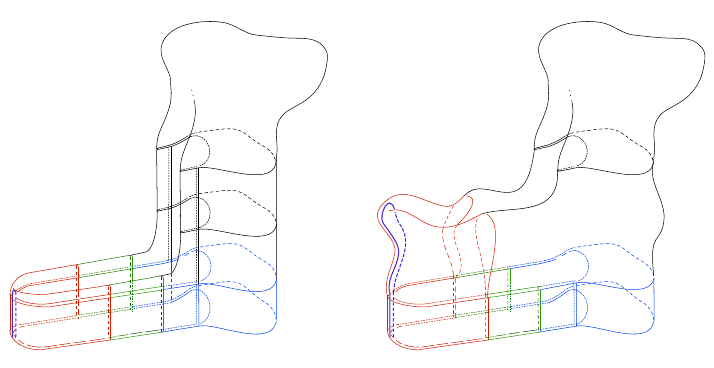}%
\end{picture}%
\setlength{\unitlength}{3947sp}%
\begingroup\makeatletter\ifx\SetFigFont\undefined%
\gdef\SetFigFont#1#2#3#4#5{%
  \reset@font\fontsize{#1}{#2pt}%
  \fontfamily{#3}\fontseries{#4}\fontshape{#5}%
  \selectfont}%
\fi\endgroup%
\begin{picture}(2079,2559)(1902,-5227)
\put(1000,-4350){\makebox(0,0)[lb]{\smash{{\SetFigFont{10}{8}{\rmdefault}{\mddefault}{\updefault}{\color[rgb]{0.7,0,0}$N_{\frac{1}{2}}\cong D^{p+1}\times D^{q+1}$}
}}}}
\put(1600,-4850){\makebox(0,0)[lb]{\smash{{\SetFigFont{10}{8}{\rmdefault}{\mddefault}{\updefault}{\color[rgb]{0.4,0,0.7}$D^{q+1}$}
}}}}
\end{picture}%
\caption{The embedded disk $D^{q+1}:=\bar{\phi}(\{0\}\times D^{q+1})$ (left) in the case of a metric from  $\Riem_{\Estd}^{+}(W', \p W')_{g'(-)}$ (left) and an arbitrary metric from $\Riem(W', \p W')_{g'(-)}$ (right)}
\label{spherespecfirst}
\end{figure} 
We now define the space $\Riem_{\std}^{+}(W', \p W')_{g'(-)}$ as:
$$\Riem_{\std}^{+}(W', \p W')_{g'(-)}:=\{\bar{h}\in \Riem^{+}(W', \p W')_{g'(-)}:\bar{\phi}_{\frac{1}{2}}^{*}\bar{h}=g_{\tor}^{p+1}(1)_{\lambda}+g_{\tor}^{q+1}\text{ for some } \lambda>0 \}.$$
A typical metric in this space is schematically depicted in Fig. \ref{spherespecsecond}.
\begin{figure}[htb!]
\vspace{1.2cm}
\hspace{-2cm}
\begin{picture}(0,0)%
\includegraphics{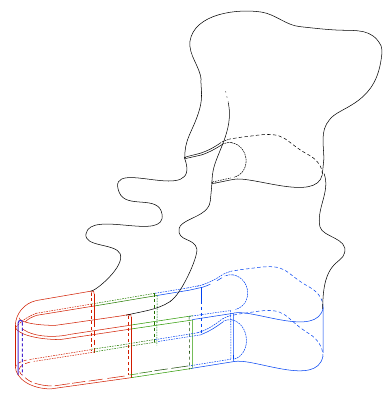}%
\end{picture}%
\setlength{\unitlength}{3947sp}%
\begingroup\makeatletter\ifx\SetFigFont\undefined%
\gdef\SetFigFont#1#2#3#4#5{%
  \reset@font\fontsize{#1}{#2pt}%
  \fontfamily{#3}\fontseries{#4}\fontshape{#5}%
  \selectfont}%
\fi\endgroup%
\begin{picture}(2079,2559)(1902,-5227)
\put(1000,-4300){\makebox(0,0)[lb]{\smash{{\SetFigFont{10}{8}{\rmdefault}{\mddefault}{\updefault}{\color[rgb]{0.7,0,0}$N_{\frac{1}{2}}\cong D^{p+1}\times D^{q+1}$}
}}}}
\put(1600,-4850){\makebox(0,0)[lb]{\smash{{\SetFigFont{10}{8}{\rmdefault}{\mddefault}{\updefault}{\color[rgb]{0.4,0,0.7}$D^{p+1}$}
}}}}
%\put(3100,-2000){\makebox(0,0)[lb]{\smash{{\SetFigFont{10}{8}{\rmdefault}{\mddefault}{\updefault}{\color[rgb]{0,0,0}$\textcolor{black}{\mathbb{R}^{q+1}}$}
%}}}}
%\put(800,-2700){\makebox(0,0)[lb]{\smash{{\SetFigFont{10}{8}{\rmdefault}{\mddefault}{\updefault}{\color[rgb]{0,0,0}$\textcolor{black}{dt^{2}+ds_{p}^{2}+\delta^{2}ds_{q}^{2}}$}
%}}}}
%\put(1500,-3300){\makebox(0,0)[lb]{\smash{{\SetFigFont{10}{8}{\rmdefault}{\mddefault}{\updefault}{\color[rgb]{0,0,0}$\textcolor{black}{dt^{2}+ds_{p}^{2}+g_{\tor}^{q+1}}$}
%}}}}
\end{picture}%
\caption{A typical element of  $\Riem_{\std}^{+}(W', \p W')_{g'(-)}$}
\label{spherespecsecond}
\end{figure} 
The following Lemma is now an immediate consequence of Lemma \ref{stdincludeggen}. Note that the roles of $p$ and $q$ are reversed from the statement of Lemma \ref{stdincludeggen}, which dealt with the analogous situation on the manifold $W$. However, as we assume both $p,q\geq 2$, the lemma goes through in the $W'$ case just as well.
\begin{lemma}\label{right1}
The inclusion: $$\Riem_{\std}^{+}(W', \p W')_{g'(-)}\subset \Riem^{+}(W', \p W')_{g'(-)}$$ is a weak homotopy equivalence.
\end{lemma}

It remains to show that the inclusion $\Riem_{\Estd}^{+}(W', \p W')_{g'(-)}\subset \Riem_{\std}^{+}(W', \p W')_{g'(-)}$
is a weak homotopy equivalence. To aid the reader, and to make sure it is clear that this inclusion makes sense, we begin by comparing an example of each type of metric in Fig. \ref{bootcomparison} below. Elements of $\Riem_{\Estd}^{+}(W', \p W')_{g'(-)}$ take a standard form on a larger region than elements of $\Riem_{\std}^{+}(W', \p W')_{g'(-)}$. This was the motivation for referring to them as ``extra-standard" metrics. To understand this, recall that a metric in $\Riem_{\Estd}^{+}(W', \p W')_{g'(-)}$ is the image of a metric in $\Riem_{\boot}^{+}(W, \p W)_{g_{\std}(l_2)}$. In turn, this pre-image metric takes the form: $h\cup \bar{g}_{\std\boot}$ for some $h\in \Riem^{+}(W, \p W)_{g}$. Let us denote by $C\subset W$, the region of $W$ on which the metric component $\bar{g}_{\std\boot}$ is defined. We then denote by $C'$, the analogous region of $W'$ obtained after surgery on $C$; see Fig. \ref{bootcomparison} where the region $C'$ is contained in the solid lines.

Let us now consider the pre-image metric. On $C$ it takes a form which was essentially obtained by pushing out boots with toe height $1$ from a product $g_{\std}+dt^{2}$ on $X\times [0, l_3]$ for some potentially large $l_3>1$. Thus, if we remove the subregion diffeomorphic to $N\times [0, l_3]$ on which the metric takes the form $ds_{p}^{2}+\bar{g}_{\boot}(1)_{\Lambda, \bar{l}}$, the remaining metric is isometric to $((X\setminus N)\times [0, l_3], g+dt^{2})$. Consider now the image metric in $\Riem_{\Estd}^{+}(W', \p W')_{g'(-)}$. There is a corresponding subregion of $C'$ (indicated by the thicker solid lines in Fig. \ref{bootcomparison}) where the metric also takes the form $((X\setminus N)\times [0, l_3], g+dt^{2})$. Fixing the region, but replacing the metric with one from $\Riem_{\std}^{+}(W', \p W')_{g'(-)}$ we see an immediate difference; see right of Fig. \ref{bootcomparison}. On the subregion corresponding to $(X\setminus N)\times [0, 1]$, these metrics agree. But outside of here the latter metric is arbitrary.

\begin{figure}[htb!]
\vspace{1.5cm}
\hspace{-11cm}
\begin{picture}(0,0)%
\includegraphics{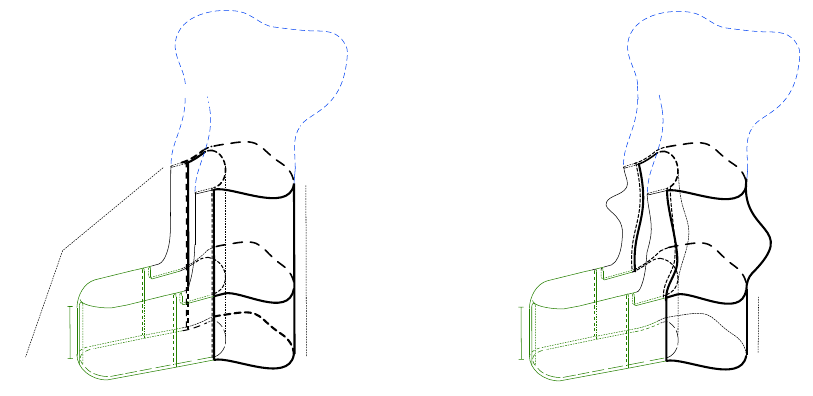}%
\end{picture}%
\setlength{\unitlength}{3947sp}%
\begingroup\makeatletter\ifx\SetFigFont\undefined%
\gdef\SetFigFont#1#2#3#4#5{%
  \reset@font\fontsize{#1}{#2pt}%
  \fontfamily{#3}\fontseries{#4}\fontshape{#5}%
  \selectfont}%
\fi\endgroup%
\begin{picture}(2079,2559)(1902,-5227)
\put(4400,-4250){\makebox(0,0)[lb]{\smash{{\SetFigFont{10}{8}{\rmdefault}{\mddefault}{\updefault}{\color[rgb]{0,0,0}$(X\setminus N)\times [0, l_3]$}
}}}}
\put(8000,-4760){\makebox(0,0)[lb]{\smash{{\SetFigFont{10}{8}{\rmdefault}{\mddefault}{\updefault}{\color[rgb]{0,0,0}$(X\setminus N)\times [0, 1]$}
}}}}
\put(3600,-2800){\makebox(0,0)[lb]{\smash{{\SetFigFont{10}{8}{\rmdefault}{\mddefault}{\updefault}{\color[rgb]{0,0,1}$X\setminus C'$}
}}}}
\put(7300,-2800){\makebox(0,0)[lb]{\smash{{\SetFigFont{10}{8}{\rmdefault}{\mddefault}{\updefault}{\color[rgb]{0,0,1}$X\setminus C'$}
}}}}

\put(2350,-4760){\makebox(0,0)[lb]{\smash{{\SetFigFont{10}{8}{\rmdefault}{\mddefault}{\updefault}{\color[rgb]{0,0.6,0}$1$}
}}}}
\put(2350,-4000){\makebox(0,0)[lb]{\smash{{\SetFigFont{10}{8}{\rmdefault}{\mddefault}{\updefault}{\color[rgb]{0,0,0}$C'$}
}}}}

\put(5950,-4760){\makebox(0,0)[lb]{\smash{{\SetFigFont{10}{8}{\rmdefault}{\mddefault}{\updefault}{\color[rgb]{0,0.6,0}$1$}
}}}}
\end{picture}%
\caption{An element of $\Riem_{\Estd}^{+}(W', \p W')_{g'(-)}$ and an arbitrary element of $\Riem_{\std}^{+}(W', \p W')_{g'(-)}$ (right) }
\label{bootcomparison}
\end{figure} 
We now focus on an arbitrary metric $h\in\Riem_{\std}^{+}(W', \p W')_{g'(-)}$.
Working only where this metric is standard (and thus where it agrees with every other metric in $\Riem_{\std}^{+}(W', \p W')_{g'(-)}$), we may specify a region $R_2$ diffeomorphic to $X\times I$ which decomposes $W'$ into $3$ pieces $R_1, R_2$ and $R_3$ where $R_1\sqcup R_2=W'\setminus \mathrm{int}(R_1)$; see Fig. \ref{bootspec}. The first piece, $R_1$, is diffeomorphic to the trace of the original surgery on $X$, while the third piece, $R_3$, is diffeomorphic to $W$. We will assume each piece contains its boundary and so $\{R_1, R_2, R_3\}$ is not strictly a partition of $W'$.

\begin{figure}[htb!]
\vspace{1.4cm}
\hspace{-3cm}
\begin{picture}(0,0)%
\includegraphics{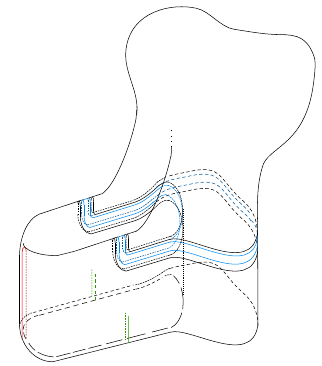}%
\end{picture}%
\setlength{\unitlength}{3947sp}%
\begingroup\makeatletter\ifx\SetFigFont\undefined%
\gdef\SetFigFont#1#2#3#4#5{%
  \reset@font\fontsize{#1}{#2pt}%
  \fontfamily{#3}\fontseries{#4}\fontshape{#5}%
  \selectfont}%
\fi\endgroup%
\begin{picture}(2079,2559)(1902,-5227)
\put(3900,-3150){\makebox(0,0)[lb]{\smash{{\SetFigFont{10}{8}{\rmdefault}{\mddefault}{\updefault}{\color[rgb]{0,0,0}$\textcolor{black}{R_3}$}
}}}}
\put(2500,-4560){\makebox(0,0)[lb]{\smash{{\SetFigFont{10}{8}{\rmdefault}{\mddefault}{\updefault}{\color[rgb]{0,0,0}$\textcolor{black}{R_1}$}
}}}}
\put(4200,-4240){\makebox(0,0)[lb]{\smash{{\SetFigFont{10}{8}{\rmdefault}{\mddefault}{\updefault}{\color[rgb]{0,0,0}$\textcolor{black}{R_2}$}
}}}}
\end{picture}%
\caption{ An arbitrary element $h\in\Riem_{\std}^{+}(W', \p W')_{g'(-)}$}
\label{bootspec}
\end{figure} 

We will take a closer look at the region $R_2$. As shown in Fig. \ref{adjustmentregion}, this region can be decomposed into $2$ pieces: $R_2'$ and $R_2''$.
One of these, $R_2'$, highlighted in Fig. \ref{adjustmentregion}, is diffeomorphic to $I\times S^{p}\times D^{q+1}$. The other, $R_2''$ is identified with $I\times (X\setminus N_{\frac{1}{2}})$ and thus diffeomorphic to $I\times (X\setminus (S^{p}\times D^{q+1}))$.
Thus on $R_2''$, any metric in $\Riem_{\std}^{+}(W', \p W')_{g'(-)}$ takes the form $dt^{2}+g_{\std}|_{X\setminus N_{\frac{1}{2}}}$. In particular, near $R_2'\cap R_2''\cong I\times S^{p}\times S^{q}$, any such metric is $dt^{2}+ds_{p}^{2}+ds_{q}^{2}$. 
\begin{figure}[htb!]
\vspace{1.2cm}
\hspace{-7cm}
\begin{picture}(0,0)%
\includegraphics{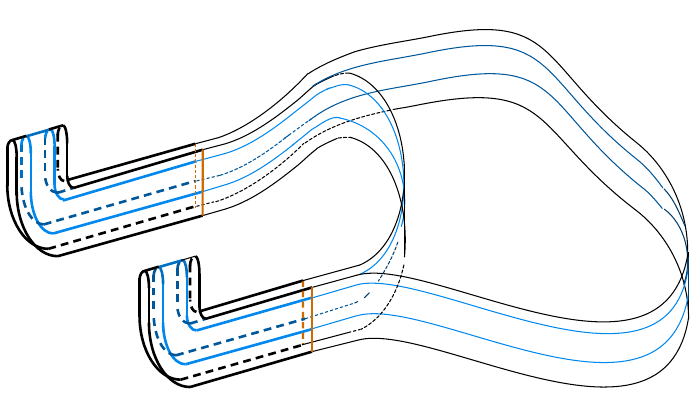}%
\end{picture}%
\setlength{\unitlength}{3947sp}%
\begingroup\makeatletter\ifx\SetFigFont\undefined%
\gdef\SetFigFont#1#2#3#4#5{%
  \reset@font\fontsize{#1}{#2pt}%
  \fontfamily{#3}\fontseries{#4}\fontshape{#5}%
  \selectfont}%
\fi\endgroup%
\begin{picture}(2079,2559)(1902,-5227)
\put(6000,-3850){\makebox(0,0)[lb]{\smash{{\SetFigFont{10}{8}{\rmdefault}{\mddefault}{\updefault}{\color[rgb]{0,0,0}$\textcolor{black}{R_2''}$}
}}}}
\put(2500,-4560){\makebox(0,0)[lb]{\smash{{\SetFigFont{10}{8}{\rmdefault}{\mddefault}{\updefault}{\color[rgb]{0,0,0}$\textcolor{black}{R_2'}$}
}}}}
%\put(4200,-4240){\makebox(0,0)[lb]{\smash{{\SetFigFont{10}{8}{\rmdefault}{\mddefault}{\updefault}{\color[rgb]{0,0,0}$\textcolor{black}{R_2}$}
%}}}}
\end{picture}%
\caption{The decomposition of the region $R_2$ into subregions $R_2'$ and $R_2''$}
\label{adjustmentregion}
\end{figure}
On the region $R_2'$, all metrics in $\Riem_{\std}^{+}(W', \p W')_{g'(-)}$ take the form of a restriction to a particular subregion of $\mathbb{R}^{p+1}\times \mathbb{R}^{q+1}$ equipped with the metric $g_{\tor}^{p+1}+g_{\tor}^{q+1}$. This subregion is depicted in Fig \ref{R2}. On $\mathbb{R}^{p+1}\times \mathbb{R}^{q+1}$, the cap of each torpedo is centered at the origin while the necks are infinite. The coordinates $s$ and $t$ in the diagram represent radial distance in $\mathbb{R}^{p+1}$ and $\mathbb{R}^{q+1}$ respectively. The boundary of $R_2'$ consists of $\partial_1(R_{2}')=R_2'\cap R_2''\cong I\times S^{p}\times S^{q}$ and a piece $\partial_2(R_{2}')\cong S^{p}\times D^{q+1}\times\{0\}\sqcup S^{p}\times D^{q+1}\times\{0\}$. We have further decomposed $R_2'$ into $3$ subregions. Near $\partial_1 R_{2}'$, there is a region we denote $R_{2}'(\mathrm{outer})$ on which the metric takes the form $ds^{2}+dt^{2}+ds_{p}^{2}+ds_{q}^{2}$. Next we have the region $R_{2}'(\mathrm{middle})$ in which the metric takes the same form but where the boundary curves in the diagram bend over an angle of $\frac{\pi}{2}$. Finally  we have remaining region, $R_{2}'(\mathrm{inner})$ where the metric takes the form $ds^{2}+ds_{p}^{2}+g_{\tor}^{q+1}$.
\begin{figure}[htb!]
\vspace{2cm}
\hspace{-2cm}
\begin{picture}(0,0)%
\includegraphics{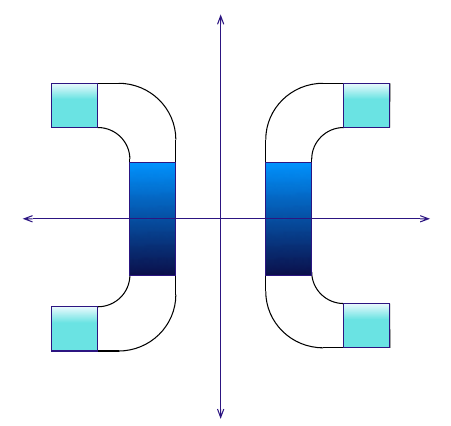}%
\end{picture}%
\setlength{\unitlength}{3947sp}%
\begingroup\makeatletter\ifx\SetFigFont\undefined%
\gdef\SetFigFont#1#2#3#4#5{%
  \reset@font\fontsize{#1}{#2pt}%
  \fontfamily{#3}\fontseries{#4}\fontshape{#5}%
  \selectfont}%
\fi\endgroup%
\begin{picture}(2079,2559)(1902,-5227)
\put(5000,-3450){\makebox(0,0)[lb]{\smash{{\SetFigFont{10}{8}{\rmdefault}{\mddefault}{\updefault}{\color[rgb]{0,0,0}$\textcolor{black}{\mathbb{R}^{p+1}}$}
}}}}
\put(3100,-2000){\makebox(0,0)[lb]{\smash{{\SetFigFont{10}{8}{\rmdefault}{\mddefault}{\updefault}{\color[rgb]{0,0,0}$\textcolor{black}{\mathbb{R}^{q+1}}$}
}}}}
\put(1400,-2600){\makebox(0,0)[lb]{\smash{{\SetFigFont{10}{8}{\rmdefault}{\mddefault}{\updefault}{\color[rgb]{0,0,0}$\textcolor{black}{R_{2}'(\mathrm{outer})}$}
}}}}
\put(2020,-3300){\makebox(0,0)[lb]{\smash{{\SetFigFont{10}{8}{\rmdefault}{\mddefault}{\updefault}{\color[rgb]{0,0,0}$\textcolor{black}{R_{2}'(\mathrm{inner})}$}
}}}}
\end{picture}%
\caption{Representing the metric on $R_2'$ as a subspace of $\mathbb{R}^{p+1}\times \mathbb{R}^{q+1}$ with the metric $g_{\tor}^{p+1}+g_{\tor}^{q+1}$}
\label{R2}
\end{figure} 

We will now describe an isotopy of the arbitrary metric $h\in \Riem_{\std}^{+}(W', \p W')_{g'(-)}$ to an element of $\Riem_{\Estd}^{+}(W', \p W')_{g'(-)}$. This isotopy is denoted: $$\sigma:I\longrightarrow \Riem_{\std}^{+}(W', \p W')_{g'(-)}$$ and will satisfy the following conditions.
\begin{enumerate}
\item[(i.)] $\sigma(0)=h$. 
\item[(ii.)] For all $t\in I$, the metric $\sigma(t)$ restricts on the region $R_2''$ to a cylinder metric of the form $dr^{2}+h_1|_{X\setminus N_{\rho'}}$ where $r\in[0,\lambda]$ for some smooth parameter $\lambda=\lambda(t)>0$.
\item[(iii.)] On the regions $R_1\cup R_3$, $\sigma(t)=h|_{R_1\cup R_3}$ for all $t\in I$.
\end{enumerate}
Thus, we make adjustments only on $R_2$. In order to satisfy condition (iii.) above, we in fact work away from the boundary of $R_2$, as suggested by the subregion inscribed inside $R_2$ in Fig. \ref{adjustmentregion}. The parameter $\lambda(t)$ in condition (ii.) allows us to stretch the cylindrical metric on $R_{2}''$ to be as long as we require. 
It remains to specify conditions for the isotopy $\sigma$ on the region $R_2'$. Obviously, near $R_2'\cap R_2''$, the metric $\sigma(t)|_{R_2'}$ must take the form $dr^{2}+ds_{p}^{2}+ds_{q}^2$ as in condition (ii.). Thus on $R_{2}'(\mathrm{outer})$ and $R_{2}'(\mathrm{middle})$ we do nothing more than carefully stretch out the cylinder in line with the above condition (ii.). The non-trivial part happens inside $R_{2}'(\mathrm{inner})$. Recall, here the metric takes the form $ds^{2}+ds_{p}^{2}+g_{\tor}^{q+1}(1)$. We leave the $ds_{p}^2$ factor untouched. On the cylinder factor $ds^{2}+g_{\tor}^{q+1}(1)$ we perform two iterations of the boot metric isotopy described in Lemma \ref{bootisot}. The resulting metric on this cylinder is depicted in Fig. \ref{bootstretch9} below while the result of the isotopy on $h$ is depicted in Fig. \ref{bootspec2}. 
\begin{figure}[htb!]
\vspace{1.7cm}
\hspace{-10cm}
\begin{picture}(0,0)%
\includegraphics{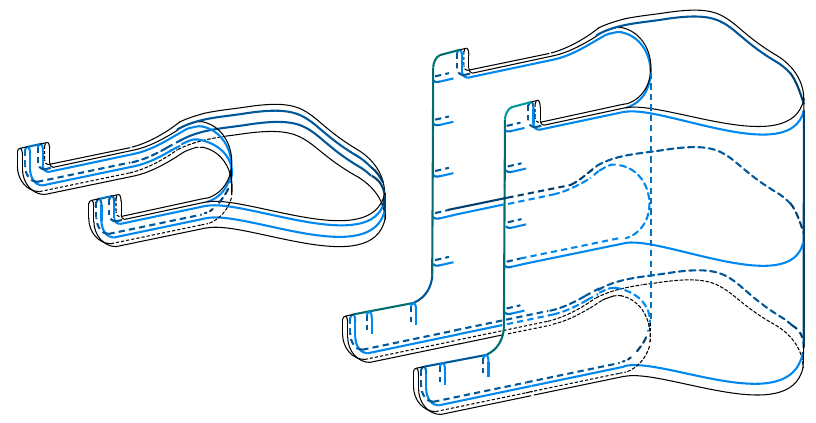}%
\end{picture}%
\setlength{\unitlength}{3947sp}%
\begingroup\makeatletter\ifx\SetFigFont\undefined%
\gdef\SetFigFont#1#2#3#4#5{%
  \reset@font\fontsize{#1}{#2pt}%
  \fontfamily{#3}\fontseries{#4}\fontshape{#5}%
  \selectfont}%
\fi\endgroup%
\begin{picture}(2079,2559)(1902,-5227)
\end{picture}%
\caption{The restriction of the metrics $\sigma(0)=h$ (left) and $\sigma(1)$ (right) to the region $R_{2}$}
\label{bootstretch9}
\end{figure} 
\begin{figure}[htb!]
\vspace{4.8cm}
\hspace{-9cm}
\begin{picture}(0,0)%
\includegraphics{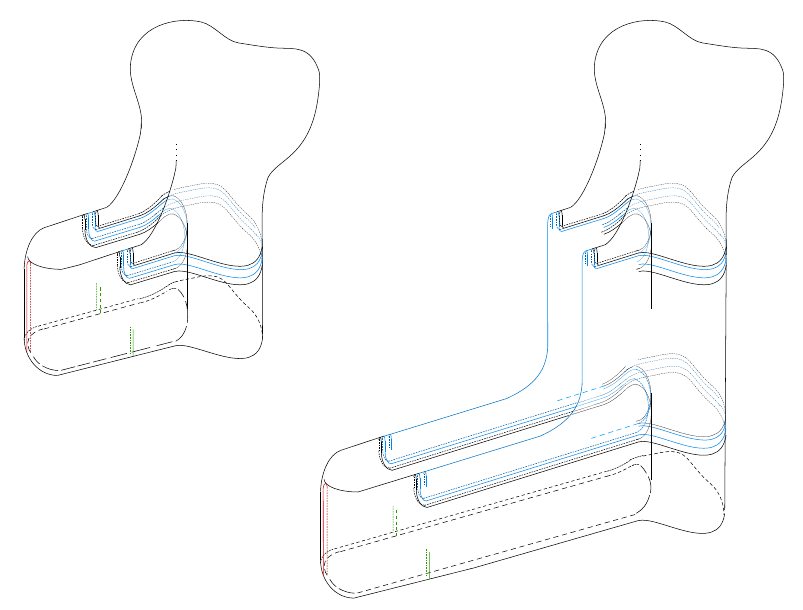}%
\end{picture}%
\setlength{\unitlength}{3947sp}%
\begingroup\makeatletter\ifx\SetFigFont\undefined%
\gdef\SetFigFont#1#2#3#4#5{%
  \reset@font\fontsize{#1}{#2pt}%
  \fontfamily{#3}\fontseries{#4}\fontshape{#5}%
  \selectfont}%
\fi\endgroup%
\begin{picture}(2079,2559)(1902,-5227)
\end{picture}%
\caption{ The elements $\sigma(0)=h\in\Riem_{\std}^{+}(W', \p W')_{g'(-)}$ and $\sigma(1)\in\Riem_{\Estd}^{+}(W', \p W')_{g'(-)}$}
\label{bootspec2}
\end{figure} 

\begin{remark}
The isotopy $\sigma$ necessarily alters the neck-lengths of the torpedo factor of the standard part of the boundary metric. This is the reason we opted to work inside the larger space $\Riem^{+}(W', \p W')_{g'(-)}$ rather than $\Riem^{+}(W', \p W')_{g'}$.
\end{remark}

The isotopy does indeed move metrics from $\Riem_{\std}^{+}(W', \p W')_{g'(-)}$ into $\Riem_{\Estd}^{+}(W', \p W')_{g'(-)}$. Moreover, as this isotopy only takes place where all elements of $\Riem_{\std}^{+}(W', \p W')_{g'(-)}$ agree and are standard, it is clear that it works for compact families of psc-metrics. There is a problem in using this isotopy to show the inclusion $\Riem_{\Estd}^{+}(W', \p W')_{g'(-)}\subset  \Riem_{\std}^{+}(W', \p W')_{g'(-)}$ is a weak homotopy equivalence. Metrics which are already in $\Riem_{\Estd}^{+}(W', \p W')_{g'(-)}$ may be temporarily moved out of that space. As everything happens on the standard region, the damage is not severe and we can get around this problem by, once again, introducing an intermediary space which contains $\Riem_{\Estd}^{+}(W', \p W')_{g'(-)}$, is invariant of the isotopy construction above and which is itself weakly homotopy equivalent to $\Riem_{\Estd}^{+}(W', \p W')_{g'(-)}$. This space is denoted $\Riem_{\step\std}^{+}(W', \p W')_{g'}(-)$ and will satisfy:
$$
\Riem_{\Estd}^{+}(W', \p W')_{g'(-)}\subset \Riem_{\step\std}^{+}(W', \p W')_{g'}(-)\subset  \Riem_{\std}^{+}(W', \p W')_{g'(-)}.
$$
It is here that we finally make use of our earlier work on step metrics, in section \ref{steps}.

We specify a cylinder $\cC\subset N$, where $\cC\cong S^{p}\times D^{q+1}\times I$ and whereon each element of $\Riem_{\Estd}^{+}(W', \p W')_{g'(-)}$ takes the form $dr^{2}+ds_{p}^{2}+g_{\tor}^{q+1}(1)$. We assume that the region $R_{2}'$ defined above is contained in $\cC$. This is depicted in the left image of Fig. \ref{steppy}. We now define the space $ \Riem_{\step\std}^{+}(W', \p W')_{g'}(-)$ to be the subspace of $\Riem_{\std}^{+}(W', \p W')_{g'(-)}$ consisting of psc-metrics which satisfy the following conditions.
\begin{enumerate}
\item[(i.)] Outside of $\cC$, metrics in $\Riem_{\step\std}^{+}(W', \p W')_{g'}(-)$ satisfy the same conditions as metrics in $\Riem_{\Estd}^{+}(W', \p W')_{g'(-)}$
\item[(ii.)] On $\cC$, metrics in $\Riem_{\step\std}^{+}(W', \p W')_{g'}(-)$ take the form $ds_{p}^{2}+g_{\step}^{q+2}(1)$ where $g_{\step}^{q+1}(1)$ is some element of $\Riem_{\step}^{+}(D^{q+1}\times I)$, the space of step metrics.
\end{enumerate}
Such elements are depicted in the right hand image of Fig. \ref{steppy}. We now have the following lemma.

\begin{lemma}\label{right2}
The inclusion: $$\Riem_{\Estd}^{+}(W', \p W')_{g'(-)}\subset  \Riem_{\std}^{+}(W', \p W')_{g'(-)}$$ is a weak homotopy equivalence.
\end{lemma}
\begin{proof}
It is easy to see that the sequence of inclusions: 
$$
\Riem_{\Estd}^{+}(W', \p W')_{g'(-)}\subset \Riem_{\step\std}^{+}(W', \p W')_{g'}(-)\subset  \Riem_{\std}^{+}(W', \p W')_{g'(-)},
$$
is valid. It is a simple application of Lemma \ref{stepprop} to show that the first of these inclusions is a weak homotopy equivalence. Finally, as the restriction of the isotopy described above on a cylinder of torpedo metrics is precisely the isotopy used to define step metrics, the space $\Riem_{\step\std}^{+}(W', \p W')_{g'}(-)$ is invariant under the above isotopy. \end{proof} 
\begin{figure}[htb!]
\vspace{6.8cm}
\hspace{-8cm}
\begin{picture}(0,0)%
\includegraphics{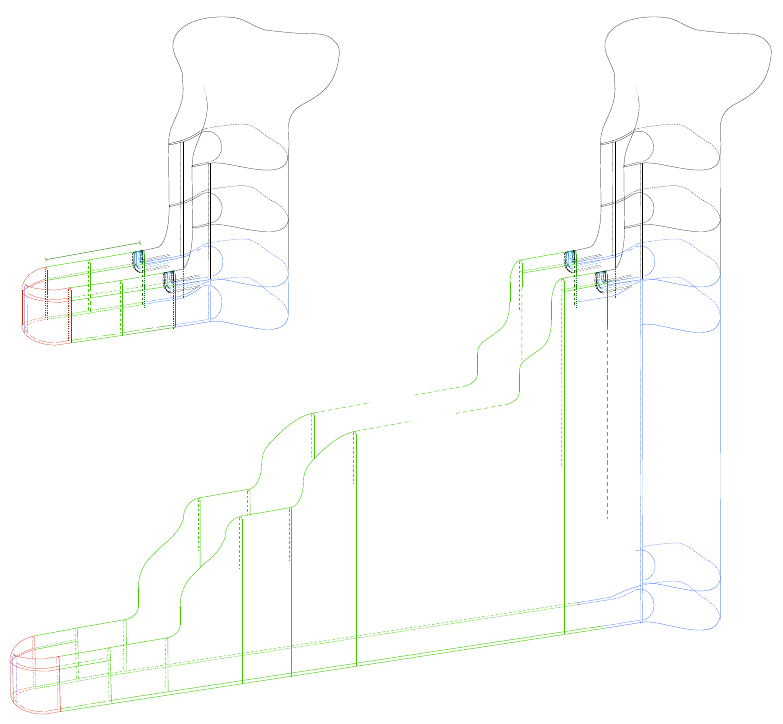}%
\end{picture}%
\setlength{\unitlength}{3947sp}%
\begingroup\makeatletter\ifx\SetFigFont\undefined%
\gdef\SetFigFont#1#2#3#4#5{%
  \reset@font\fontsize{#1}{#2pt}%
  \fontfamily{#3}\fontseries{#4}\fontshape{#5}%
  \selectfont}%
\fi\endgroup%
\begin{picture}(2079,2559)(1902,-5227)
\put(2600,-1400){\makebox(0,0)[lb]{\smash{{\SetFigFont{10}{8}{\rmdefault}{\mddefault}{\updefault}{\color[rgb]{0,0.7,0}$C$}
}}}}
\end{picture}%
\caption{Elements of $\Riem_{\Estd}^{+}(W', \p W')_{g'(-)}$ (left) and $\Riem_{\step\std}^{+}(W', \p W')_{g'}(-)$ (right) }
\label{steppy}
\end{figure} 
\noindent Combining Lemmas \ref{right1} and \ref{right2} with the observation made concerning diagram \ref{genEstd} gives us that the vertical map on the right side of diagram (\ref{bdydiag}) is a weak homotopy equivalence. This completes the proof of Theorem A, in the non-empty case. 

Finally, we consider the possibility that $\Riem^{+}(W, \p W)_{g}=\emptyset$. Assume that $\Riem^{+}(W', \p W')_{g'}\neq\emptyset$ and let $W''=W'\cup T_{\phi}^{-1} $ be the manifold obtained by attaching the ``upside-down" version of the trace of $\phi$, denoted $T_{\phi}^{-1}$ onto $W'$ in the obvious way. This upside-down trace also admits a Gromov-Lawson trace metric, $\bar{g}^{-1}$, between $g'$ and some psc-metric $g''$ on $X$. Thus, for some metric $h\in \Riem^{+}(W', \p W')_{g'}$ we have a psc-metric $h\cup\bar{g}^{-1}\in \Riem^{+}(W'')_{g''}$. By performing surgeries on the interior of $T_{\phi}\cup T_{\phi}^{-1}$ we may make this region cylindrical and so turn $W''$ back into $W$. Importantly, the fact that $p,q \geq 2$ means that such surgeries are in codimension at least three. Hence we may apply the original Gromov-Lawson construction, away from the boundary, to obtain from $h\cup\bar{g}^{-1}$, an element of $\Riem^{+}(W, \p W)_{g''}$. It follows from Theorem 0.4 of section 3.2 in \cite{Walsh1}, that the metrics $g$ and $g''$ are isotopic. Thus, by Lemma \ref{quasihomo}, the space $\Riem^{+}(W, \p W)_{g}\neq \emptyset$ and so we have that $\Riem^{+}(W, \p W)_{g}\neq\emptyset$ if and only if $\Riem^{+}(W', \p W')_{g'}\neq \emptyset$. This completes the proof of Theorem A.

\subsection{The proof of Theorem B}

This is essentially a mimic of the proof by Gromov and Lawson of their coincidentally named Theorem B in \cite{GL}.
\begin{proof}[Proof of Theorem B]
The weak homotopy equivalence in (i) follows from the same reasoning employed by Gromov and Lawson in their proof of Theorem B in \cite{GL}. Here, the authors showed that any two spin bordant simply connected $n$-manifolds $M$ and $N$ are mutually obtainable by surgeries in codimension at least three. They demonstrate this by performing surgeries in the interior of a spin manifold  $Y^{n+1}$, where $\partial Y=M\sqcup N$, to remove non-trivial elements in the integral homology groups ${H}_k(Y)$, where $k=1,2,n-1,n$. The spin and dimension conditions mean that all such elements are realised by embedded spheres with trivial normal bundle. In our case, as $X$ is null-bordant, we may perform surgeries on $W$ to obtain a manifold $W'$ with boundary $X'$ diffeomorphic to $S^{n}$. As $W'$ forms one half of a closed simply connected spin manifold (its double), there is a sequence of surgeries on the interior of $W'$, which transform it into the disk $D^{n+1}$. As with the Gromov-Lawson proof in \cite{GL}, all of these surgeries satisfy the appropriate co-dimension restrictions required by Theorem A and so the theorem follows.

Part (ii) is proved by the following argument. As $g$ is Gromov-Lawson cobordant to the round metric $ds_{n}^{2}$, part (i) gives us that $\Riem^{+}(W, \p W)_{g}$ is weakly homotopy equivalent to $\Riem^{+}(D^{n+1})_{ds_{n}^{2}}$. In turn the space $\Riem^{+}(D^{n+1})_{ds_{n}^{2}}$ is homeomorphic to a subspace $\Riem_{\std}^{+}(S^{n})\subset \Riem^{+}(S^{n})$, of psc-metrics which take the form of a standard torpedo on the southern hemisphere. The conclusion then follows from the fact, demonstrated in the proof of Theorem \ref{Chernysh}, that the inclusion $\Riem_{\std}^{+}(S^{n})\rh \Riem^{+}(S^{n})$ is a weak homotopy equivalence.
\end{proof}

\section{The Proof of Theorem C}\label{ThmBCDproof}
At this point it is useful to recall the following observation by Chernysh in \cite{Che2}. Recall the restriction map:
      \begin{equation*}
      \begin{split}
      \res: \Riem^{+}(W, \p W)&\longrightarrow \Riem^{+}(\partial W)\\
      g&\longmapsto g|_{\partial W.}
      \end{split}
      \end{equation*}
      Thus, for any $g\in\Riem^{+}(X)$, the space $\Riem^{+}(W, \p W)_g$ is precisely $ \res^{-1}(g)$.
      We next denote by $\Riem_0^{+}(\partial W)$, the image space $\res(\Riem^{+}(W, \p W))\subset \Riem^{+}(\partial W)$. Originally proved to be a quasifibration by Chernysh, \cite[Theorem 1.1]{Che2}, Ebert and Frenck have shown in \cite[Theorem 1.1]{EF}that the map:
      \begin{equation*}
      \res: \Riem^{+}(W, \p W)\longrightarrow \Riem_{0}^{+}(\partial W),
      \end{equation*}
      obtained by restriction of the codomain, is a Serre fibration. In particular this implies that for any psc-metrics $g\in \Riem^{+}(\partial W)$ and $\bar{g}\in \Riem^{+}(W, \p W)_g $, there are isomorphisms: 
      \begin{equation*}
      \pi_{k}(\Riem^{+}(W, \p W), \Riem^{+}(W, \p W)_g, \bar{g})\cong\pi_{k}(\Riem_{0}^{+}({\partial W}, g)),
      \end{equation*}
      for all $k\geq 0$. In particular, we obtain the following long exact sequence in homotopy groups.
{ \small
  $$ 
\xymatrix{  \cdots\ar@{->}[r(0.4)]^{}& {\pi_k(\Riem^{+}(W, \p W)_{g},\bar{g})} \ar@{->}[r(0.65)]^{}&
 {\pi_k (\Riem^{+}(W, \p W), \bar{g})}  \ar@{->}[r]^{\res_{*}} & {\pi_{k}({\Riem}_{0}^{+}(\partial W), g)}\ar@{->}[d]^{} & \\
\cdots{}&{\pi_{k-1}({\Riem}_{0}^{+}(\partial W), g)}\ar@{->}[l]^{} &{ \pi_{k-1} (\Riem^{+}(W, \p W), \bar{g})}\ar@{->}[l]_{\res_{*}} &{{\pi_{k-1}(\Riem^{+}(W, \p W)_{g},\bar{g})}}\ar@{->}[l]^{}.
}
$$}
\noindent \hspace{-0.15cm}We note here that the set ${\pi_{0}({\Riem}_{0}^{+}(\partial W), g)}$ is usually not trivial.

Recall that the space $\Riem_0^{+}(X)$, where $X=\partial W$, is of course the image of the above restriction map. Similarly, we define the space $\Riem_{0}^{+}(X')$ where $X'=\partial W'$, in the same way.
We will also denote by $\Riem_{0, \std}^{+}(X)$ and $\Riem_{0, \std}^{+}(X')$ the respective spaces $\Riem_0^{+}(X)\cap \Riem_{\std}^{+}(X)$ and $\Riem_0^{+}(X')\cap \Riem_{\std}^{+}(X')$. Recall, the spaces $ \Riem_{\std}^{+}(X)$ and $\Riem_{\std}^{+}(X')$ were the subject of Lemma \ref{easyhomeoclosed}.
 The following lemma is an easy consequence of Theorem \ref{Chernysh}.

\begin{lemma}\label{resChe}
Let $W$ and $W'$ satisfy the hypotheses of Theorem A. Then the inclusions: $$\Riem_{0, \std}^{+}(X)\rh \Riem_0^{+}(X) \text{ and } \Riem_{0, \std}^{+}(X')\rh \Riem_0^{+}(X')$$ are homotopy equivalences. In particular, the spaces $\Riem_{0}^{+}(X)$ and $\Riem_{0}^{+}(X')$ are homotopy equivalent.
\end{lemma}      
\begin{proof}
Suppose $g\in \Riem^{+}(X)$ extends to an element of $\bar{g}\in\Riem^{+}(W, \p W)_{g}$. Then, using the fact that isotopy implies concordance, any metric $h\in\Riem^{+}(X)$, which is psc-isotopic to $g$, extends to an element $\bar{h}\in\Riem^{+}(W, \p W)_{h}$. This partitions the path components of $\Riem^{+}(X)$ into the subsets of those consisting of psc-metrics which extend to elements of $\Riem^{+}(W, \p W)$ and those which do not. The former is of course the subspace $\Riem_{0}^{+}(W, \p W)$. Furthermore, if a psc-metric $g$ is an element of $\Riem_{0}^{+}(X)$, then any metric $g'$ obtained by Gromov-Lawson surgery is an element of $\Riem_0^{+}(X')$. In case this is unclear, recall that if $\bar{g}\in\Riem^{+}(W, \p W)_{g}$, the procedure described in Theorem A gives rise to a psc-metric $\bar{g}'\in\Riem^{+}(W', \p W')_{g'}$. Hence, the Gromov-Lawson construction gives rise to a one to one correspondence between path components of $\Riem^{+}(X)$ and $\Riem^{+}(X')$. In particular, subspaces $\Riem_{0}^{+}(X)$ and $\Riem_{0}^{+}(X')$ arise by simply removing corresponding (homotopy equivalent) path components of $\Riem^{+}(X)$ and $\Riem^{+}(X')$. The arguments in the proof of Theorem \ref{Chernysh} which demonstrate that the inclusion $\Riem_{\std}(X)\rh\Riem^{+}(X)$ is a homotopy equivalence go through just as well to show that $\Riem_{0, \std}(X)\rh\Riem_{0}^{+}(X)$ and $\Riem_{0, \std}(X')\rh\Riem_{0}^{+}(X')$ are homotopy equivalences. It follows immediately from Lemma \ref{easyhomeoclosed}, that the spaces $\Riem_{0, \std}(X)$ and $\Riem_{0, \std}(X')$ are homeomorphic, completing the proof. \end{proof} 

Suppose now that we fix psc-metrics $g\in\Riem_0^{+}(X)$ and $g_{\std}\in\Riem_{0, \std}^{+}(X)$ where $g_\std$ is obtained from $g$ in the usual way, by the Gromov-Lawson construction. Denoting by $\Riem^{+}(W, \p W)_{\std}$, the pre-image $\res^{-1}(\Riem_{0, \std{}}^{+}(X))$, we obtain the following commutative diagram where, as always, hooked arrows denote inclusion. 
$$
\xymatrix{
& \Riem^{+}(W, \p W)_{g}\ar@{->}[d]^{\mu_{\bar{g}_{\con}}} \ar@{^{(}->}[r]^{}& \Riem^{+}(W, \p W) \ar@{->}[r]^{\res} &\Riem_{0}^{+}(X)\\
& \Riem^{+}(W, \p W)_{g_{\std}} \ar@{^{(}->}[r]^{} &\Riem^{+}(W, \p W)_{\std} \ar@{^{(}->}[u]^{}\ar@{->}[r]^{\res}&\Riem_{0, \std{}}^{+}(X)\ar@{^{(}->}[u]^{}
}
$$

The vertical map on the left, $\mu_{{\bar{g}_{\con}}}$, is defined in the proof of Theorem A. In particular, it is the map that attaches to the boundary of $W$ the concordance between $g$ and $g_{\std}$ obtained from the Gromov-Lawson isotopy specified by Theorem \ref{GLcompact}.  By Lemma \ref{whe1}, we know that this map is a weak homotopy equivalence. The rightmost vertical map is inclusion and is also a homotopy equivalence by Lemma \ref{resChe} above. We will now show that the middle vertical map, which is also an inclusion, is a weak homotopy equivalence as well. 
\begin{lemma} \label{whe}
The inclusion map: $$\Riem^{+}(W, \p W)_{\std}\hookrightarrow  \Riem^{+}(W, \p W)$$ is a weak homotopy equivalence.
\end{lemma}
\begin{proof}
The map $\res:\Riem^{+}(W, \p W)_{\std}\longrightarrow \Riem_{0, \std{}}^{+}(X)$ is simply the pull-back under the inclusion $\Riem_{0,\std{}}^{+}(X)\hookrightarrow \Riem_{0}^{+}(X)$, of the known Serre fibration (from \cite[Theorem 1.1]{EF}) $\res: \Riem^{+}(W, \p W) \rightarrow \Riem_{0}^{+}(X)$. Thus, it is also a Serre fibration.
As both rows of the above diagram are Serre fibrations, we obtain the following diagram at the level of homotopy groups. 
The rows of this diagram are of course exact. We know in advance that the vertical homomorphisms are isomorphisms in all cases except $\pi_{k}(\Riem_{0, \std}^{+}(X))\rightarrow \pi_{k}(\Riem^{+}(W, \p W), \bar{h})$. That this is necessarily an isomorphism is implied by the $5$-lemma. 
{\tiny
$$\hspace{-0.7cm}
\xymatrix@C-=0.4cm{
&
\cdots  \ar@{->}[r]^{} \pi_{k}(\Riem^{+}(W, \p W)_{g}, \bar{h})\ar@{<->}[d]^{\cong} \ar@{->}[r]^{}& \pi_{k}(\Riem^{+}(W, \p W), \bar{g}) \ar@{->}[r]^{\res_{*}} &\pi_k(\Riem_{0}^{+}(X), g) \ar@{->}[r]^{}&\pi_{k-1}(\Riem^{+}(W, \p W)_{g}, \bar{g})\ar@{->}[r]^{}\cdots&
\\
&\cdots  \ar@{->}[r]^{} \pi_{k}(\Riem^{+}(W, \p W)_{g_{\std}}, \bar{g}_{\std}) \ar@{->}[r]^{} &\pi_{k}(\Riem^{+}(W, \p W)_{\std}, \bar{g}_{\std} )\ar@{->}[u]^{}\ar@{->}[r]^{\res_{*}}&\pi_{k}(\Riem_{0, \std{}}^{+}(X), g_{\std})\ar@{<->}[u]^{\cong}\ar@{->}[r]^{}&\pi_{k-1}(\Riem^{+}(W, \p W)_{g_{\std}}, \bar{g}_{\std})\ar@{<->}[u]^{\cong}\ar@{->}[r]^{}\cdots&
}
$$}
\end{proof}

Turning our attention to $W'$, we let $g'=g_{\std}'\in\Riem_{0, \std}^{+}(X')$ denote the psc-metric obtained by Gromov-Lawson surgery on $g$. Defining $\Riem^{+}(W', \p W')_{\std}:=\res^{-1}(\Riem_{0, \std{}}^{+}(X'))$, we obtain the following commutative diagram where the righthand vertical inclusion is a weak homotopy equivalence.
$$\label{diag2}
\xymatrix{
& \Riem^{+}(W', \p W')_{g'}\ar@{<->}[d]^{=} \ar@{^{(}->}[r]^{}& \Riem^{+}(W', \p W') \ar@{->}[r]^{\res} &\Riem_{0}^{+}(X')\\
& \Riem^{+}(W', \p W')_{g'} \ar@{^{(}->}[r]^{} &\Riem^{+}(W', \p W')_{\std} \ar@{^{(}->}[u]^{}\ar@{->}[r]^{\res}&\Riem_{0, \std{}}^{+}(X')\ar@{^{(}->}[u]^{}
}
$$
Analogously to before, the map $\res: \Riem^{+}(W', \p W')_{\std} \rightarrow\Riem_{0, \std{}}^{+}(X')$ is a Serre fibration and consequently, we have the following lemma.
\begin{lemma}\label{whe'}
The inclusion map: $$\Riem^{+}(W', \p W')_{\std}\rh \Riem^{+}(W', \p W')$$ is a weak homotopy equivalence.
\end{lemma}

Next, given some $g\in\Riem_{0}^{+}(X)$, we recall a number of spaces defined in the previous section. Firstly, we have the subspaces $\Riem_{\boot}^{+}(W, \p W)_{g_{\std}(l_2)}\subset \Riem^{+}(W, \p W)_{g_{\std}}$ and $\Riem_{\Estd}^{+}(W', \p W')_{g'}\subset \Riem_{\std}^{+}(W', \p W')_{g'}\subset \Riem^{+}(W', \p W')_{g'}$. Recall these came with a map:
$$\mu_{\bar{g}}: \Riem_{\boot}^{+}(W, \p W)_{g_{\std}(l_2)}\longrightarrow\Riem_{\Estd}^{+}(W', \p W')_{g'},$$
which, in Lemma \ref{homeobd} is shown to be a homeomorphism. The reader may wish to refer to Fig. \ref{bootstepping} for a schematic description of metrics in these spaces. We recall also the spaces $\Riem_{\boot}^{+}(W, \p W)_{{g_{\std}}(l_2)}$, $\Riem_{\boot}^{+}(W, \p W)_{{g_{\std}}(-)}$, $\Riem_{\step}^{+}(W, \p W)_{{g_{\std}}(L_2)}$, $\Riem_{\step}^{+}(W, \p W)_{{g_{\std}}(l_2)}$ and $\Riem^{+}(W, \p W)_{{g_{\std}}(-)}$ described in the preamble to Lemma \ref{bootstepwhe}. 
We will now describe extensions of these spaces which involve allowing the boundary metric to vary.

For each $g\in \Riem_{0, \std}^{+}(X)$, we define $\Riem_{\boot}^{+}(W, \p W)_{g(l_2)}$ to be the space of psc-metrics which is the images of the map $\mu_{\bar{g}_{\pre}}$ with one minor caveat. In this case, as the metric is already standard, we assume that the $\mu_{\bar{g}_{\con}}$ factor of the map $\mu_{\bar{g}_{\pre}}$ simply attaches a cylinder. We similarly define $\Riem_{\step}^{+}(W, \p W)_{{g}(L_2)}$. Finally, we define $\Riem_{\Estd}^{+}(W', \p W')_{g'}$ to be the image of $\mu_{\bar{g}}$ where $\bar{g}$ is the corresponding Gromov-Lawson trace with respect to the metric $g$.  

We are now able to define the spaces: 
\begin{equation*}
\begin{split}
\Riem_{\Estd}^{+}(W', \p W'):=&\bigcup_{g'\in\Riem_{0, \std}^{+}(X')}\Riem_{\Estd}^{+}(W', \p W')_{g'}\\
\Riem_{\boot(l_2)}^{+}(W, \p W):=&\bigcup_{g\in\Riem_{0, \std}^{+}(X)}\Riem_{\boot }^{+}(W, \p W)_{g(l_2)},\\
\Riem_{\boot}^{+}(W, \p W):=&\bigcup_{l_2}\Riem_{\boot(l_2)}^{+}(W, \p W) {\text{ where $l_2>0$ is admissible},}\\
\Riem_{\step(L_2)}^{+}(W, \p W):=&\bigcup_{g\in\Riem_{0, \std}^{+}(X)}\Riem_{\step}^{+}(W, \p W)_{g(L_2)}\\
\Riem_{\step}^{+}(W, \p W):=&\bigcup_{L_2>0}\Riem_{\step(L_2)}^{+}(W, \p W).
\end{split}
\end{equation*}
We would like to add a couple of other spaces here. With $\Riem_{\std}^{+}(X)$ as above we define for each $\lambda>0$ the space:
$$\Riem_{\std(\lambda)}^{+}(X):=\{g\in\Riem^{+}(X):\phi_{\frac{1}{2}}^{*}g=ds_{p}^{2}+g_{\tor}^{q+1}(1)_{\lambda}\}.$$
Thus, the neck length of the torpedo factor on the standard region is equal to $\lambda>0$. %Similarly, we define $\Riem_{\std(\lambda)}^{+}(X')$. 
We now have:
$$\Riem_{0, \std(\lambda)}^{+}(X):= \Riem_{\std(\lambda)}^{+}(X)\cap \Riem_{0}^{+}(X)\text{ and } \Riem^{+}(W, \p W)_{\std(\lambda)}:=\res^{-1}(\Riem_{0, \std(\lambda)}^{+}(X)).$$
Naturally, this leads to:
$$\Riem_{0, \std(-)}^{+}(X):=\bigcup_{\lambda>0}\Riem_{0, \std(\lambda)}^{+}(X) \text{ and } \Riem^{+}(W, \p W)_{\std(-)}:=\res^{-1}\Riem_{0, \std(-)}^{+}(X).$$
 
\begin{lemma}\label{bootstepwhegeneral}
Consider, for some admissible $l_2>0$, the following commutative diagram of inclusions.
$$\label{diag2}
\xymatrix{
&\Riem_{\boot(l_2)}^{+}(W, \p W) \ar@{^{(}->}[d]^{} \ar@{^{(}->}[r]^{}&  \Riem_{\step(l_2)}^{+}(W, \p W)    \ar@{^{(}->}[r]^{}\ar@{^{(}->}[d]^{} &  \Riem^{+}(W, \p W)_{\std(l_2)} \ar@{^{(}->}[d]^{}& \\
& \Riem_{\boot(-)}^{+}(W, \p W) \ar@{^{(}->}[r]^{} &    \Riem_{\step(-)}^{+}(W, \p W) \ar@{^{(}->}[r]^{}&\Riem^{+}(W, \p W)_{\std(-)}&\Riem^{+}(W, \p W)_{\std} \ar@{_{(}->}[l]^{}
}
$$
Every map in this diagram is a weak homotopy equivalence.
\end{lemma}
\begin{proof}
With the exception of the bottom right map, this is essentially a repeat of the proof of Lemma \ref{bootstepwhe}. As all adjustments take place on metrics which are already standard, the fact that we have increased the size of the spaces from those of the earlier lemma makes no difference. The bottom right map in fact forms part of a deformation retract; simply scale all torpedo neck lengths down to length $1$.
\end{proof}

We now define a map: 
\begin{equation*}
\begin{split}
\mu: \Riem_{\boot(l_2)}^{+}(W, \p W)&\longrightarrow \Riem_{\Estd}^{+}(W', \p W')\\
h&\longmapsto\mu_{\overline{h|_{\partial W}}}(h),
\end{split}
\end{equation*}
where $\mu_{\overline{h|_{\partial W}}}$ is the lower horizontal map from diagram \ref{bdydiag} with respect to $h|_{\partial W}$.
This map simply removes the boots (the metric $ds_{p}^{2}+\hat{g}_{\tor}^{q+2}(1)_{1, l_2}$) from elements of $\Riem_{\std}^{+}(W, \p W)$ and attaches $g_{\tor}^{p+1}+g_{\tor}^{q+1}$. Again, as this map only affects the standard region of the metric, it really is just the earlier map. The following lemma, the obvious generalisation of Lemma \ref{homeobd} is now immediate.
\begin{lemma}\label{genhomeo}
The map $\mu: \Riem_{\boot(l_2)}^{+}(W, \p W)\longrightarrow \Riem_{\Estd}^{+}(W', \p W')$ is a homeomorphism. 
\end{lemma}

\begin{proof}[Proof of Thereom C]
We begin by consolidating what we have proved so far in the following diagram.
$$\label{diag2}
\xymatrix{
& \Riem_{\boot(l_2)}^{+}(W, \p W)\ar@{->}[d]^{\mu (\cong)} \ar@{^{(}->}[r]^{}& \Riem^{+}(W, \p W)_{\std} \ar@{^{(}->}[r] & \Riem^{+}(W, \p W)\\
&  \Riem_{\Estd}^{+}(W', \p W') \ar@{^{(}->}[r]^{} &\Riem^{+}(W', \p W')_{\std} \ar@{^{(}->}[r]& \Riem^{+}(W', \p W')
}
$$
The vertical map is a homeomorphism by Lemma \ref{genhomeo}.  It remains to deal with the horizontal maps. These maps are all inclusions, the rightmost two of which are weak homotopy equivalences by Lemmas \ref{whe} and \ref{whe'}. That the top left horizontal map is a weak homotopy equivalence follows from Lemma \ref{bootstepwhegeneral}. 
Weak homotopy equivalence in the case of the bottom left horizontal map follows from application of the arguments used in the proof of Lemma \ref{right2}. Again, as all adjustments take place only in the standard region, the methods of this Lemma apply equally well here.
\end{proof}

\noindent Finally, Corollary D follows easily by the argument used to prove part (i) of Theorem B.

\bibliographystyle{amsplain}

\printindex

\end{document}